\documentclass[12pt,reqno]{amsart}

\usepackage{eucal,color,graphicx,subfigure,mathrsfs}
\usepackage{amssymb,amsmath,amsbsy,,amsfonts,latexsym}

\newcommand{\vf}[1]{\boldsymbol {#1}}

\numberwithin{equation}{section}

\newtheorem{theorem}{Theorem}[section]
\newtheorem{lemma}[theorem]{Lemma}
\theoremstyle{definition}

\theoremstyle{remark}
\newtheorem{remark}[theorem]{Remark}

\numberwithin{equation}{section}

\newcommand{\vertiii}[1]{{\left\vert\kern-0.25ex\left\vert\kern-0.25ex\left\vert #1 
    \right\vert\kern-0.25ex\right\vert\kern-0.25ex\right\vert}}

\newcommand{\bld}[1]{\boldsymbol{#1}}

\newcommand{\dO}{{\partial \Omega}}

\newcommand{\Oh}{{\mathcal{T}_h}}
\newcommand{\dOh}{{\partial \Oh}}
\newcommand{\dK}{{\partial K}}

\newcommand{\Eh}{{\mathcal{E}_h}}           
\newcommand{\Ehi}{{\mathcal{E}^o_h}}

\newcommand{\divs}{\mathop{\nabla\cdot}}

\newcommand{\bint}[3]{\langle#1,#2\rangle_{#3}}
\newcommand{\pol}[2]{\mathcal{P}_{#1}(#2)} 
\newcommand{\bpol}[2]{\bld{\mathcal{P}}_{#1}(#2)}

\newcommand{\pv}{\bld{\varPi}_h}%_{\scriptscriptstyle{\bld{V}}}}
\newcommand{\eq}{\bld{\varepsilon}^{{\ensuremath{\scriptscriptstyle{\bld{q}}}}}_h}
\newcommand{\eu}{\varepsilon^{{\ensuremath{\scriptscriptstyle{u}}}}_h}
\newcommand{\euhat}{{\varepsilon}^{{\ensuremath{\scriptscriptstyle{\,\widehat{{u}}}}}}_h}

\title{An analysis of HDG methods for convection--dominated diffusion problems}

\author{Guosheng Fu}
\address{School of Mathematics, University of Minnesota,
Minneapolis, MN 55455, USA}
\email{fuxxx165@math.umn.edu}

\author{Weifeng Qiu}
\address{Department of Mathematics, City University of Hong Kong,
83 Tat Chee Avenue, Kowloon, Hong Kong, China}
\email{weifeqiu@cityu.edu.hk}

\author{Wujun Zhang}
\address{Department of Mathematics, University of Maryland,
College Park, MD 20742-4015, USA}
\email{wujun@umd.edu}

\begin{document}

\markboth{ G. Fu, W. Qiu, W. Zhang}{HDG methods for convection--dominated diffusion problems}
\begin{abstract}
We present the first a priori error analysis of the $h$--version of the hybridizable discontinuous Galkerin (HDG) methods applied to convection--dominated diffusion problems. 
We show that, when using polynomials of degree no greater than $k$, the $L^2$--error of the scalar variable 
converges with order $k + 1/2$ on general conforming quasi--uniform simplicial meshes, just as for conventional DG methods.
We also show that the method achieves the optimal $L^2$--convergence order of $k+1$ on special meshes. 
Moreover, we discuss a new way of  implementing the HDG methods for which
 the spectral condition number of the global matrix is independent of the diffusion 
coefficient. Numerical experiments are presented which verify our theoretical results.
\end{abstract}
\maketitle

\section{Introduction}
In this paper, we present the first a priori error analysis of the $h$--version
 of the HDG methods for the following convection--dominated diffusion model problem:
\begin{subequations}
\label{cd_eqs}
\begin{align}
\label{cd_1}
-\epsilon\Delta u + \boldsymbol{\beta}\cdot \nabla u = &\; f  \quad \text{ in $\Omega$, }
\\
u = &\; g \quad \text{ on $\partial \Omega$, }
\end{align}
\end{subequations}
where $\Omega \in \mathbb{R}^d$ ($d = 2,3$) is a polyhedral domain, $\epsilon \ll |\bld \beta|_{L^\infty(\Omega)}$, $f\in L^2 (\Omega)$ and $g\in H^{1/2}(\partial \Omega)$. 
As in  \cite{AyusoMarini:cdf}, we assume that the velocity field $\boldsymbol{\beta}\in W^{1,\infty}(\Omega)$ has neither closed curves nor
stationary points, i.e.,
\begin{align}
\label{assump_beta0}
 \bld\beta \in W^{1,\infty}(\Omega) \text{ has no closed curves},\quad \bld \beta(\bld x)\not=\bld 0\quad\forall\bld x\in\Omega.
\end{align}
This implies that there exists 
a smooth function $\psi$ %$\psi \in W^{k+1, \infty} (\Omega)$ 
so that 
%\begin{subequations}
\begin{align}
\label{beta_assumps}
\boldsymbol{\beta}\cdot \nabla\psi(x) & \geq b_{0}\qquad \forall x\in \Omega,
\end{align}
%\end{subequations}
for some constant $b_0 > 0$, see \cite{DevinatzEllisFriedman} or \cite[Appendix A]{AyusoMarini:cdf} for a proof.
We also assume that 
\begin{align}
\label{assump_beta1}
 - \divs \bld\beta(\bld x) \geq 0\quad \forall \bld x\in\Omega,
\end{align}
which means that the ``effective'' reaction is non--negative 
since  
\[
%\left( \bld \beta\cdot \nabla u\right) u =\; \divs \left( \bld \beta u^2/2\right) -(\divs \bld \beta)u^2/2.
 \bld \beta\cdot \nabla u=\; \divs \left( \bld \beta u\right) -(\divs \bld \beta)u.
\]
Let us remark that assumption \eqref{assump_beta0} ensures the well--posedness
of the continuous problem in the pure hyperbolic limit ($\epsilon = 0$), see \cite[Chapter 3]{Goering83}
for details. 
It is also well--known \cite{Eckhaus72,Goering83} that solutions to the problem \eqref{cd_eqs} may develop layers, whose approximation
is the major difficulty of designing high--order, robust numerical schemes. 
We refer to \cite{Roos08,Roos12} for a comprehensive information on different numerical  techniques for \eqref{cd_eqs}.

%Since 1970's when discontinuous Galerkin (DG) method was first introduced in \cite{ReedHill73} for hyperbolic equations, 
%there has been an active development of DG methods for hyperbolic and nearly hyperbolic problems. 
%While the methods are successfully applied in scientific and engineering computing, a main disadvantage, compared to other methods, is that DG methods require a higher number 
%of globally--coupled degrees of freedom for the same mesh.
In the last decade, the discontinuous Galerkin methods \cite{Cockburn99,CockburnShu01}  
%have become more and more popular, and
have been extensively considered for convection--diffusion equations. 
For example, see the local discontinuous Galerkin (LDG) methods \cite{CockburnShu98,CockburnDawson00,
HoustonSchwabSuli2002,CastilloCockburnSchotzauSchwab02,CockburnDong07}, the method of Baumann and Oden \cite{BaumannOden99}, 
the  interior--penalty discontinuous Galerkin (IP--DG) methods
\cite{ZarinRoos2005,AyusoMarini:cdf}, the multiscale discontinuous Galerkin method \cite{HughesScovazziBochevBuffa2006,
BuffaHughesSangalli2006}, the mixed--hybrid--discontinuous Galerkin (MH--DG) method \cite{Egger2010}, and the HDG methods
 \cite{CockburnDongGuzmanMarcoRiccardo09,NguyenPeraireCockburnHDGLCD09,NguyenPeraireCockburnHDGNCD09}. 
%The references listed here represent
%the authors' own intereses and are by no means exaustive.
On the other hand, for steady--state problems, the main disadvantage of conventional DG methods, compared to other methods, is that they require a higher number 
of globally--coupled degrees of freedom for the same mesh.
In order to address this issue, the HDG methods 
were introduced in \cite{CockburnGopalakrishnanLazarov09} in the framework of second--order uniformly elliptic problems.
The methods are such that the globally--coupled degrees of freedom are only those of the numerical traces on the mesh skeleton. A  similar idea was used 
in \cite{Egger2010} to obtain the MH--DG method. Hence, the use of the hybridization technique eliminates the main disadvantage
of DG methods to a significant extent.

In \cite{CockburnDongGuzmanSFH08,CockburnGopalakrishnanSayas09}, 
it was shown that, for the purely diffusive model problem, the numerical approximation of HDG methods achieves the same order of convergence 
as that of mixed methods.
% by using 
%a projection based error analysis based on an energy argument and a duality argument. 
More precisely, when using polynomials of degree no greater than $k$, the $L^2$--error for both the scalar  and 
flux approximation converges optimally with order
$k+1$, and a postprocessed scalar approximation converges with  order $k+2$ for $k\ge 1$. 
%These results indicate that HDG methods outperform conventional DG methods in terms of implementation and 
%numerical accuracy for the pure elliptic problems. 
Recently in \cite{ChenCockburnHDGI,ChenCockburnHDGII}, 
similar results have been proven for the convection--diffusion equation when 
the diffusion coefficient is  comparable to the convection coefficient, with 
 variable--degree approximations and nonconforming meshes.

In this work, we focus on the analysis of the convection--dominated case, that is, when
$\epsilon\ll |\bld \beta|_{L^\infty(\Omega)}$. 
We show that for the HDG methods using polynomial degree $k\ge 1$ with a suitably chosen stabilization function,
 we have, for general meshes, that
\begin{equation}
\label{our_result1}
\Vert u_h-u\Vert_{L^{2}(\Omega)} \le C\, h^{k}(\epsilon h^{-1/2}+\epsilon^{1/2} +h^{1/2})\vert u\vert_{H^{k+1}(\Omega)},
\end{equation}
and, for meshes (almost) aligned with the direction of $\bld \beta$, that
\begin{equation}
\label{our_result2}
\Vert u_{h}-u\Vert_{L^{2}(\Omega)} \le C\, h^{k}(\epsilon h^{-1/2}+\epsilon^{1/2} +h)\vert u\vert_{H^{k+1}(\Omega)},
\end{equation}
where $C$ is a  constant independent of $\epsilon$ and $h$.
Note that if $\epsilon \leq \mathcal{O}(h^2)$, we obtain optimal convergence for $\|u_h-u\|_{L^2(\Omega)}$
in \eqref{our_result2}, which can be considered as an extension of a similar result for the pure hyperbolic case 
\cite{CockburnDongGuzman2008,CockburnDongGuzmanQian2010}.
We also show that, with a suitably chosen stabilization function,  the condition number of the global matrix for the scaled
numerial traces is $\mathcal{O}(h^{-2})$, independent of $\epsilon$.

To prove these estimates, we cannot use the approach used in \cite{ChenCockburnHDGI,ChenCockburnHDGII} 
because the constants in the error estimates in \cite{ChenCockburnHDGI,ChenCockburnHDGII} 
may blow up as $\epsilon$ approaches $0$. For example, the constant $\Upsilon_K^{\max}$ in \cite[Theorem 2.1]{ChenCockburnHDGI} 
is of order 
$\mathcal{O}(\epsilon^{-1})$.
%, and the constant in the duality estimate is of order $\mathcal{O}(\epsilon^{-1})$. 
%Actually, we expect to loose superconverges for the scalar variable as $\epsilon$ goes to $0$.
In order to obtain an estimate that is robust with respect to $\epsilon$, we need to  modify the 
energy argument used in \cite{ChenCockburnHDGI,ChenCockburnHDGII} by using test functions 
similar to that used in \cite{Johnson86,AyusoMarini:cdf}. In \cite{Johnson86}, a weighted 
test function was used to obtain the $L^2$--stability of the original DG method \cite{ReedHill73} for the pure hyperbolic equation. 
In \cite{AyusoMarini:cdf}, the idea was extended to convection--diffusion--reaction equations using the IP--DG method. 
We also need to use a new projection to obtain error estimates with less restrictive regularity assumptions.

Next,  we would like to compare our results with those obtained for the IP--DG method in \cite{AyusoMarini:cdf}.
Our convergence result for $\|u_h-u\|_{L^2(\Omega)}$ on general meshes is the same as that in \cite{AyusoMarini:cdf}, while
the optimal convergence on special meshes is new. Also, our method has 
less globally--coupled degrees of freedom, and  our choice of the stabilization function is 
determined clearly in the numerical formulation; there is no need to choose it empirically as in the IP--DG method.

Now, let us compare our results with those for the MH--DG method \cite{Egger2010}.
The MH--DG method uses a combination of upwind techniques used in  DG methods for hyperbolic problems with conservative discretizations of mixed methods for elliptic problems. 
To the best of  our knowledge, \cite{Egger2010} 
is the first paper which utilizes hybrid formulations for the mixed and  DG methods to make them compatible. 
We show that our method is quite similar to the MH--DG method. Actually, in Appendix~\ref{appendix-1}, we show that, using the 
same approximation spaces as those of the MH--DG method, the HDG method becomes 
exactly the same as the MH--DG method by suitably choosing the stabilization function. The new features of our analysis with respect to that of  
\cite{Egger2010} are that we can deal with variable velocity 
field $\bld \beta$ and that we 
have an estimate of $\|u_h-u\|_{L^2(\Omega)}$, which is not obtained in \cite{Egger2010}. 
Moreover, we prove that the condition number for the global linear system 
can be rendered to be 
independent of $\epsilon$ and of order $\mathcal{O}(h^{-2})$. 

A well--known stabilization technique for convection--dominated diffusion problems in the finite element method literature 
is  residual--based stabilization, see the
SUPG  \cite{BrooksHughes} and  residual--free bubbles \cite{BrezziHughesMariniRussoSuli1999,BrezziMariniSuli2000} methods. 
The main disadvantages of 
residual--based stabilization are that they are not locally conservative and that the performance of the methods relies heavily 
on a proper choice of the stabilization parameter, which might be hard to determine or expensive to compute.
%Actually, SUPG is among the first successful finite element methods that gives reasonable approximations for convection--dominated 
%problems using quite coarse meshes.  
We refer readers to \cite{HoustonSchwabSuli2000} for a detailed comparison of the hp--version of the DG methods with 
the SUPG methods in the pure hyperbolic case, and to \cite{Egger2010} for a detailed comparision of the MH--DG method with the SUPG methods for 
the convection--diffusion case. 

%Also, the stability results
%in our analysis is stronger than that in \cite{Egger2010} when $\epsilon = \mathcal{O}(h^2)$. 
%In addition, there is no analysis of whether the condition number of global stiffness matrix depends on $\epsilon$ in \cite{Egger2010}.
%Thus, we conclude that it is more efficient to use HDG method with the second example of stabilization function to solve 
%convection dominated diffusion problem.

The rest of the paper is organized as follows. In Section $2$, we introduce the HDG method and state and discuss the main theoretical results. 
In Section $3$, we give a characterization of the HDG method,
and show that, after scaling, the condition number of the global matrix is 
independent of $\epsilon$. 
In Section $4$, we present the convergence analysis of 
the HDG method. Finally, in Section $5$, we display numerical experiments which verify our theoretical results. 
%And finally in Appendix, we prove the relation between MH--DG and HDG, and prove the estimates for the condition number.

\section{The HDG method and  main results}
In this section, we present the HDG method  and state and discuss 
our main theoretical results.
%\subsection{Preliminary assumptions on the PDE}
%We assume the velocity field $\boldsymbol{\beta}$ is $\mathcal{O}(1)$ in the domain
% $\Omega$ that satisfies assumptions \eqref{assump_beta0} and \eqref{assump_beta1}.
%We also assume that $\epsilon \ll 1$, that is, we are in the convection dominated region.
\subsection{The mesh}
Let $\Oh$ be a conforming, quasi--uniform simplicial 
triangulation of $\Omega$.
Given an element (triangle/tetrahedron) $K \in \Oh$, which we assume to be an
open set,
$\dK$ denotes the set of its edges in the two
dimensional case and of its faces in the three dimensional case.
Elements of $\dK$ will be generally referred to as faces,
{regardless of the spatial dimension}, and denoted by $F$. The set of all
{(interior)}
faces of the triangulation will be denoted $\Eh$ {($\Ehi$)}. 
We distinguish functions defined on the faces of the triangulation (the
skeleton) by saying that they are defined on $\Eh$ from functions
defined on the boundaries of the elements (and therefore having the
ability to display two different values on interior faces) by saying
that they are defined on $\dOh$. Hence the spaces $L^2(\Eh)$ and
$L^2(\dOh)$ have different meanings.
For each element $K\in\mathcal{T}_h$, we set $h_K := |K|^{\frac{1}{d}}$,  and for each
of face $F$, $h_F := |F|^{\frac{1}{d-1}}$, where $|\cdot|$ denotes the Lebesgue measure in $d$ or $d -1$ dimensions.
{We define $h = \max_{K\in\mathcal{T}_{h}}h_{K}$.} 
%{\em As usual, we always assume the mesh 
%is a conforming quasi--uniform shape--regular simplicial one, which is not stated in the statement of the theorems and lemmas for simplicity.}

Moreover, we also consider special meshes that satisfy the following assumption: there exists a constant $C$ so that
\begin{align}
\label{mesh_assumps}
\max(\sup_{x\in F}\boldsymbol{\beta}(\bld x)\cdot\boldsymbol{n},0)\leq Ch_{K}, \quad \forall F \in \partial K \setminus F_{K}^{+}, \forall K\in\Oh,
\end{align}
where $F_K^+$ is the face of $K$ such that
$
\sup_{x\in F^+_K}\boldsymbol{\beta}(\bld x)\cdot\boldsymbol{n} = \max_{F\in \dK} \sup_{x\in F}\boldsymbol{\beta}(\bld x)\cdot\boldsymbol{n}.
$
% called the outflow face of $K$. 
%Roughly speaking, this assumption says that for each element $K$, besides $F_K^+$
%we only allow one outflow face up to an $\mathcal{O}(h)$ perturbation.
These meshes have been introduced in \cite{CockburnDongGuzman2008} (see also \cite{CockburnDongGuzmanQian2010})
for the analysis of the original DG method. In appendix~\ref{section_appb}, we sketch how  to generate 
a triangulation satisfying assumption (\ref{mesh_assumps}). 

\subsection{The HDG method}
In order to define the HDG method, 
we  first rewrite our model problem
 (\ref{cd_eqs}) as the following first-order system by introducing $\bld q = -\epsilon \nabla u$ as a
new unknown:
\begin{subequations}
\label{cd_first_order}
\begin{align}
\label{cd_first_order1}
\epsilon^{-1}\boldsymbol{q} + \nabla u & = 0\text{ in }\Omega,\\
\label{cd_first_order2}
\nabla\cdot \boldsymbol{q} + \boldsymbol{\beta}\cdot \nabla u & = f \text{ in }\Omega,\\
\label{cd_first_order3}
u & = g \text{ on }\partial\Omega.
\end{align}
\end{subequations}

Let us also define  the following finite element spaces:
\begin{subequations}
\label{fem_spaces}
\begin{align}
\label{fem_space1}
\boldsymbol{V}_{h}&=\{\boldsymbol{r}\in L^{2}(\Omega;\mathbb{R}^{d}):\boldsymbol{r}|_{K}\in P_{k}(K;\mathbb{R}^{d})
\quad \forall K\in \mathcal{T}_{h}\},\\
\label{fem_space2}
W_{h}&=\{w\in L^{2}(\Omega):w|_{K}\in P_{k}(K)
\quad \forall K\in \mathcal{T}_{h}\},\\
\label{fem_space3}
M_{h}&=\{\mu\in L^{2}(\mathcal{E}_{h}):\mu |_{F} \in P_{k}(F)
\quad \forall F\in \mathcal{E}_{h}\},\\
\label{fem_space4}
M_{h}(g)&=\{\mu\in M_{h}: \langle \mu, \xi \rangle_{\partial\Omega} = \langle g , \xi \rangle_{\partial\Omega}  \quad \forall \xi\in M_{h}\},
\end{align}
\end{subequations}
where $P_{k}(D)$ is the space of polynomials of total degree not larger than $k\ge 0$ defined on $D$, and 
\[
 \langle \xi, \eta \rangle_{\partial\Omega} = \sum_{F\in \dO}\int_{F}\xi\,\eta\,\mathrm{ds}.
\]

The HDG method seeks an approximation $(\boldsymbol{q}_{h},u_{h},\widehat{u}_{h})\in 
\boldsymbol{V}_{h}\times W_{h}\times M_{h}$ by requiring that
\begin{subequations}
\label{cd_hdg_eqs}
\begin{align}
\label{cd_hdg_eq1}
(\epsilon^{-1}\boldsymbol{q}_{h},\boldsymbol{r})_{\mathcal{T}_{h}} -(u_{h},\nabla\cdot \boldsymbol{r})_{\mathcal{T}_{h}}
+\langle \widehat{u}_{h},\boldsymbol{r}\cdot \boldsymbol{n}\rangle_{\partial\mathcal{T}_{h}}&=0,\\
\label{cd_hdg_eq2}
-(\boldsymbol{q}_{h}+\boldsymbol{\beta}u_{h},\nabla w)_{\mathcal{T}_{h}} - (\nabla\cdot\boldsymbol{\beta}u_{h},w)_{\mathcal{T}_{h}}
+\langle (\widehat{\vf q}_{h}+\widehat{\boldsymbol{\beta}u_{h}})\cdot \boldsymbol{n}, w\rangle_{\partial\mathcal{T}_{h}}&=(f,w)_{\mathcal{T}_{h}},\\
\label{cd_hdg_eq3}
\langle \widehat{u}_{h} , \mu\rangle_{\partial\Omega}&=\langle g , \mu \rangle_{\partial\Omega},\\
\label{cd_hdg_eq4}
\langle (\widehat{\vf q}_{h}+\widehat{\boldsymbol{\beta}u_{h}})\cdot \boldsymbol{n} , \mu \rangle_{\partial\mathcal{T}_{h}\backslash\partial\Omega}&=0,
\end{align}
\text{for all $(\boldsymbol{r},w,\mu)\in \boldsymbol{V}_{h}\times W_{h}\times M_{h}$, where 
%the normal component of 
the numerical trace 
$(\widehat{\vf{q}}_{h}+\widehat{\boldsymbol{\beta}u_{h}})\cdot\bld n$ is given by}
\begin{equation}
\label{cd_hdg_eq5}
(\widehat{\vf{q}}_{h}+\widehat{\boldsymbol{\beta}u_{h}})\cdot\bld n=\boldsymbol{q}_{h}\cdot\bld n+\boldsymbol{\beta}\cdot\bld n\,
\widehat{u}_{h}+\tau (u_{h}-\widehat{u}_{h}) 
\text{  on }\partial\mathcal{T}_{h},
\end{equation}
\end{subequations}
and {\em the stabilization function $\tau$ is piecewise, nonnegative constant defined on $\dOh$}.
{Here we write $\left(\eta,\zeta\right)_{\mathcal{T}_{h}} := \sum_{K \in \mathcal{T}_{h}} \int_K \eta\,\zeta\,\mathrm{dx},$ 
%where $(\eta,\zeta)_D$ denotes the integral of $\eta\zeta$ over the domain $D \subset \mathbb{R}^n$. We also write
and 
$\langle \eta, \zeta \rangle_{\partial\mathcal{T}_{h}} := \sum_{K \in \mathcal{T}_{h}} 
\int_\dK \eta\,\zeta\,\mathrm{ds}$.
%where $\langle \eta \,,\,\zeta \rangle_{D}$ denotes the integral of $\eta \zeta$ over the domain $D \subset \mathbb{R}^{n-1}$.}
In Section~\ref{sec:hybrid}, we show that the linear system \eqref{cd_hdg_eqs} can be efficiently implemented
 so that the only global unknowns are related to 
the numerical trace $\widehat{u}_{h}$. 

The HDG method \eqref{cd_hdg_eqs} has a unique solution provided that the stabilization function $\tau$ 
in (\ref{cd_hdg_eq5}) satisfies the following assumption:
\begin{align}
\label{assump_tau_00}
\inf_{\bld x\in F}\left(\tau - \dfrac{1}{2} \boldsymbol{\beta(x)}\cdot \boldsymbol{n}\right)\ge 0, \quad
\forall F\in\dK,\forall K\in\mathcal{T}_{h},
\end{align}
where in each element, the strict inequality holds at least on one face; see \cite[Theorem 3.1]{NguyenPeraireCockburnHDGLCD09} for a proof.

\subsection{Assumptions on the stabilization function}
Next, we present our assumptions on the stabilization function $\tau$ 
verifying the inequality \eqref{assump_tau_00}. We then construct two examples satisfying them.

To do that, we need to introduce some notation.
%Given a stabilization function $\tau$ satisfies \eqref{assump_tau_00}, we define some quantities related to $\tau$ that will be used in our analysis.
Let $F_K^\star$ be the face of $K$ on which $\tau$ attains its maximum, and $F_K^s$ be the face of $K$ on which $\inf_{\bld x\in F}\left( \tau - \frac{1}{2}\bld \beta(\bld x)\cdot\bld n\right)$
attains its maximum, that is,
\begin{subequations}
 \label{notation_F}
\begin{align}
\tau(F_K^\star) &: =\; \max_{F\in\dK} \tau(F)& \forall K\in \Oh,\\
\inf_{\bld x\in F_K^s}\left( \tau - \frac{1}{2}\bld \beta(\bld x)\cdot\bld n\right) &: =\; 
 \max_{F\in \dK}\inf_{\bld x\in F}\left( \tau - \frac{1}{2}\bld \beta(\bld x)\cdot\bld n\right)& \forall K\in \Oh,
\end{align}
\end{subequations}
%We define the following union of faces,
%\begin{subequations}
% \label{notation_UF}
%\begin{align}
%\label{UF1}
% \partial \mathcal{T}_h^\star &: =\; \cup_{K\in \Oh}\cup_{F\in \dK\backslash F_K^\star} F,\\
%\label{UF2}
% \partial \mathcal{T}_h^s &: =\; \cup_{K\in \Oh}F_K^s,
%\end{align}
%\end{subequations}
and set 
\begin{subequations}
%\label{tttt}
\begin{align*}
 \tau_K^w : =  &\;\max_{F\in\dK\backslash F_K^\star} \tau(F), & \tau^w : = \max_{K\in \Oh} \tau_K^w,\\
 \tau_K^{\bld v} : = &\;\inf_{\bld x\in F_K^s}\left( \tau - \frac{1}{2}\bld \beta(\bld x)\cdot\bld n\right), & \tau^{\bld v} : = \min_{K\in \Oh} \tau_K^{\bld v}.
\end{align*} 
\end{subequations}

%Now, we present some restrictions on $\tau$ that will be used throughout the paper.
We assume that there exists universal positive constants $C_0, C_1,C_2$ so that
\begin{subequations}
 \label{assumption_tau_general}
\begin{align}
\label{assumption_tau_0}
\tau_K^w \le &\;C_0&\forall K\in \Oh, \\
\label{assumption_tau_3}
%\tau_K^{\bld v} \ge &\;C_1\frac{\epsilon}{h_{K}}& \forall K\in \Oh,\\
\tau_K^{\bld v} \ge &\;C_1\min (\frac{\epsilon}{h_{K}}, 1) & \forall K\in \Oh, \\
\label{assumption_tau_1}
\inf_{\bld x\in F}\left( \tau - \frac{1}{2}\bld \beta(\bld x)\cdot\bld n\right) \ge &\;C_2 \max_{x\in F}\left |\bld \beta(\bld x)\cdot\bld n\right |& \forall 
F\in \dK, \forall K\in \Oh.
\end{align}
\end{subequations}

In order to get an improved estimate, % (Theorem~\ref{MainTh2}),  
we need to replace \eqref{assumption_tau_0} by
 the following, more restrictive assumption 
on $\tau_K^w$: assume there exists a positive constant $C$ so that
\begin{align}
 \label{assump_tau_s}
\tau_K^w \le &\;C h_K \quad\quad \forall K\in\Oh. 
\end{align}
We remark this assumption might not be compatible with \eqref{assumption_tau_1} on general meshes, but
it can hold for the meshes that satisfy assumption \eqref{mesh_assumps}.

Now, let us show that it is quite easy to construct $\tau$  satisfying assumptions \eqref{assumption_tau_general} 
%(and assumption \eqref{assump_tau_s} for special meshes) 
by displaying two of them. 
%The first one, which was first considered in \cite{Egger2010}, needs a very mild assumption on $\bld \beta$ and the mesh, while
%the second one, which seems to be new, needs no further assumptions. 
The first  example of the stabilization function is
\begin{align}
 \label{tau1}
\tau_1(F) = \max(\sup_{\bld x \in F} \,\bld \beta(\bld x)\cdot \bld n,0),\quad \forall F\in \dK, \forall K\in \Oh.
\end{align}
Assumptions \eqref{assumption_tau_0} and \eqref{assumption_tau_1} are always satisfied, and assumption \eqref{assumption_tau_3} holds 
provided 
\begin{align*}
%\label{assump_beta}
\max_{F\in \dK}\inf_{\bld x\in F}\left(- \bld \beta(\bld x)\cdot\bld n \right)\ge &\;C & \forall K\in \Oh,
\end{align*}
for some positive constant $C$; this is true, for example, for piecewise-constant $\bld \beta$.
%Moreover, if $\bld \beta$ is piecewise constant, 
%which is the case consided in \cite{Egger2010},
%assumption \eqref{assump_beta} is always true for shape--regular meshes. 
Moreover,  assumption
\eqref{assump_tau_s} is also satisfied if the mesh satisfies \eqref{mesh_assumps}.

The second  example  is
\begin{align}
\label{tau2}
\tau_2(F) = \max(\sup_{\bld x \in F} \,\bld \beta(\bld x)\cdot \bld n,0) + \min (\rho_0\frac{\epsilon}{h_K}, 1 ),\quad \forall F\in \dK, \forall K\in \Oh,
\end{align}
where $\rho_0 >0$ is a constant typically chosen to be less than or equal to $1$. 
Assumptions \eqref{assumption_tau_general} are always satisfied in this case. 
Moreover, assumption \eqref{assump_tau_s} is satisfied
provided the mesh satisfies \eqref{mesh_assumps} and we take $\epsilon \leq \mathcal{O}(h^2)$.

Let us conclude the discussion on $\tau$ by remarking that if we replace $\max(\sup_{\bld x \in F} \,\bld \beta(\bld x)\cdot \bld n,0)$ with 
$\sup_{\bld x \in F} |\,\bld \beta(\bld x)\cdot \bld n|$ in the definition of $\tau$ in \eqref{tau1} and \eqref{tau2}, assumptions \eqref{assumption_tau_general}
will be satisfied, while \eqref{assump_tau_s} is no longer true for the special meshes. 
%{\em However, our numerical results not presented in this paper shows that 
%there is no significant difference for this change.}

%{The second example of the stabilization function  is defined as follows,
%\begin{align}
%\label{tau2}
%\tau_2(F) = \max(\sup_{\bld x \in F} \,\bld \beta(\bld x)\cdot \bld n,0) + \min (\rho_0\frac{\epsilon}{h_F}, 1 ),\quad \forall F\in \dK, \forall K\in \Oh,
%\end{align}
%where $\rho_0 >0$ is a constant typically chosen to be less than $1$. Since we are in the region that $\epsilon \ll h$, we can always
%choose $\rho_0$ small enough so that  $\min (\rho_0\frac{\epsilon}{h_F}, 1 ) = \rho_0\frac{\epsilon}{h_F}$. Hence, assumptions \eqref{assumption_tau_general}
%are satisfied by the definition of $\tau_2$. Moreover, if $\epsilon = \mathcal{O}(h^2)$, we have $\min (\rho_0\frac{\epsilon}{h_F}, 1)  \le C h $, hence 
%\eqref{assump_tau_s} is also satisfied for the special mesh.
%}

\subsection{The main theoretical results}

From now on, we use $C$ to denote a generic constant, which may be dependent on the polynomial degree $k$, 
%the shape regularity constant of the elements,
and/or the velocity field $\boldsymbol\beta$. The value $C$ at different occurrences may differ.

We proceed to state our main theoretical results. 
We will show convergence estimates in the following norm
\begin{align*}
% \label{norm_ee}
\vertiii{ (\boldsymbol{r}, w, \mu)}_{e} := \left(\|\epsilon^{-1/2}\boldsymbol{r}\|^2_{\mathcal{T}_h} + \|w\|^2_{\mathcal{T}_h} + 
\left\| |\tau-\frac{1}{2} \boldsymbol{\beta}\cdot \boldsymbol{n}|^{1/2}(w-\mu)\right\|_{\partial\mathcal{T}_h}^2 \right)^{1/2},
\end{align*}
where $\|\cdot\|_{D}$ is the standard $L^2$--norm in the domain $D$.

\begin{theorem}
\label{MainTh1}
Let $(\bld q, u)$ be the solution to the boundary--value problem \eqref{cd_first_order}, 
and let $(\bld q_h, u_h,\widehat u_h)$ be the solution to the HDG method 
\eqref{cd_hdg_eqs} where the stabilization function $\tau$ satisfies assumptions \eqref{assumption_tau_general}.
Then, there exists $h_{0}$, independent of $\epsilon$, 
such that when $h<h_{0}$, we have
\begin{align*}
%\label{main1}
& \vertiii{(\bld q-\bld q_h, u-u_h, u-\widehat{u}_h)}_e \\
\nonumber
\le & \; C \epsilon^{1/2}h^{s_{\bld v}+1/2}(\epsilon^{1/2}+h^{1/2})|u |_{H^{s_{\bld v}+2}(\Oh;\mathbb{R}^d)} 
+  C h^{s_{w}+1/2}|u |_{H^{s_{w}+1}(\Oh)},
\end{align*}
for all $s_{\bld v}\in [0,k]$ and $s_w \in [0,k]$. 

%Moreover, if $k\ge 1$, $u\in H^{k+1}(\Omega)$ and $\epsilon = \mathcal{O}(h)$, we have
%\[
% \|{u-u_h}\|_{\Oh} \le \; 
% C h^{k+1/2}|u |_{H^{k+1}(\Oh)},
%\]
\end{theorem}

\begin{theorem}
\label{MainTh2}
Let $(\bld q, u)$ be the solution to the boundary--value problem \eqref{cd_first_order}, and let
 $(\bld q_h, u_h,\widehat u_h)$ be the solution to the HDG method 
\eqref{cd_hdg_eqs} where the stabilization function $\tau$ satisfies assumptions \eqref{assumption_tau_general} and \eqref{assump_tau_s}.
Then, there exists $h_{0}$, independent of $\epsilon$, 
such that when $h<h_{0}$, we have
\begin{align*}
%\label{main2}
\|{u-u_h}\|_{\Oh} &\le \; C \epsilon^{1/2}h^{s_{\bld v}+1/2}(\epsilon^{1/2}+h^{1/2})|u |_{H^{s_{\bld v}+2}(\Oh;\mathbb{R}^d)}  + 
 C h^{s_{w}+1}|u |_{H^{s_{w}+1}(\Oh)},
\end{align*}
for all $s_{\bld v}\in [0,k]$ and $s_w \in [0,k]$. 

%Moreover, if $k\ge 1$, $u\in H^{k+1}(\Omega)$ and $\epsilon = \mathcal{O}(h^2)$, we have
%\[
% \|{u-u_h}\|_{\Oh} \le \; 
% C h^{k+1}|u |_{H^{k+1}(\Oh)}.
%\]
\end{theorem}

%Let us makes several remarks about these remarkable results.
\begin{remark}
If $\tau$ satisfies assumptions \eqref{assumption_tau_general},  $k\ge 1$, $u\in H^{k+1}(\Omega)$ and $\epsilon \leq \mathcal{O}(h)$, we get 
\[
 \|{u-u_h}\|_{\Oh} \le \; 
 C h^{k+1/2}|u |_{H^{k+1}(\Oh)}
\]
by choosing $s_{\bld v} = k-1, s_w = k$ in Theorem~\ref{MainTh1}. 
If $\epsilon = 0$, our method collapses to the original DG method \cite{ReedHill73}. Since the best $L^2$--error
of the DG method for pure convection problems on general meshes is $\|u-u_h\|_\Oh\le C h^{k+1/2}|u|_{H^{k+1}},$
see \cite{Peterson91}; it is reasonable to expect $\|u-u_h\|_\Oh$ to be of
order $h^{k+1/2}$ when $\epsilon \ll 1$.
\end{remark}

\begin{remark} If $\tau$ satisfies assumptions \eqref{assumption_tau_general} and \eqref{assump_tau_s}, 
$k\ge 1$, $u\in H^{k+1}(\Omega)$ and $\epsilon \leq \mathcal{O}(h^2)$, we have
\[
 \|{u-u_h}\|_{\Oh} \le \; 
 C h^{k+1}|u |_{H^{k+1}(\Oh)},
\]
by choosing $s_{\bld v} = k-1, s_w = k$ in Theorem~\ref{MainTh2}. Note that we can construct $\tau$ 
satisfing \eqref{assumption_tau_general} and 
\eqref{assump_tau_s} provided that the mesh satisfies \eqref{mesh_assumps}. Hence, our result can be considered as a generalization of the results in 
\cite{CockburnDongGuzman2008,CockburnDongGuzmanQian2010} in which the authors obtained optimal $L^2$--convergence 
of the original DG method for convection--reaction equations on special meshes.
\end{remark}
\begin{remark}
 It is shown in Appendix~\ref{appendix-1} that we can recover the MH--DG method \cite{Egger2010} from our formulation by 
suitably choosing the stabilizaiton
function $\tau$ and the approximation spaces $\bld V_h, W_h, M_h$. Hence, our results can be directly applied to the MH--DG method. 
In particular, we gain the 
$L^2$--control of $u_h$ and obtained optimal order of convergence for $\|u-u_h\|_{\Oh}$ for special meshes.

\end{remark}

\section{A characterization of the HDG method}
\label{sec:hybrid}

%As it is typical of hybridized versions of mixed methods \cite{Egger2010}, the
%HDG method has additional unknowns, in this case,
%$\widehat{u}_h$. 
Here, we show how to eliminate, in an elementwise manner,
the unknowns $\bld{q}_h$ and $u_h$ from the equations \eqref{cd_hdg_eqs} and rewrite
the original system solely in terms of the unknown $\widehat{u}_h$, see also \cite{CockburnDongGuzmanMarcoRiccardo09,NguyenPeraireCockburnHDGLCD09}. In this way,
we do not have to deal with the large linear system generated by \eqref{cd_hdg_eqs}, 
but with the inversion of a sparser matrix of
remarkably smaller size.

\subsection{The local problems}
We begin by showing how to express the unknowns
$\bld{q}_h$ and  $u_h$ in terms of the unknown $\widehat{u}_h$.

Given ${\lambda}\in L^2(\Eh)$ and ${f}\in L^2(\Oh)$, 
consider the solution to the set of local problems in
each $K \in \Oh$: find
\[
(\bld q_h,{u}_h)\in \bld V(K)\times W(K),
\]
where $\bld V(K) : = P_k(K;\mathbb{R}^d)$ and  $W(K) := P_k(K),$
such that
\begin{subequations}
\begin{align}
\label{local1}
(\epsilon^{-1}\boldsymbol{q}_{h},\boldsymbol{r})_{K} -(u_{h},\nabla\cdot \boldsymbol{r})_{K}
&=
-\langle \lambda,\boldsymbol{r}\cdot \boldsymbol{n}\rangle_{\dK},\\
\label{local2}
(\nabla\cdot \boldsymbol{q}_{h}, w)_{K} - (u_{h},\nabla\cdot(\boldsymbol{\beta}w))_{K}
+\langle \tau u_h, w\rangle_{\dK}&= \langle (\tau-\bld{\beta}\cdot \bld n) \lambda, w\rangle_{\dK} 
+ (f,w)_{K},
\end{align}
\text{for all $(\boldsymbol{r},w)\in \boldsymbol{V}(K)\times W(K)$.}
\end{subequations}

We denote by $(\bld{q}_h^{f},u_h^{f})$
the solution of the above local problem when we take $\lambda=0$. Similarly,
we denote $(\bld{q}_h^{\lambda},u_h^{\lambda})$ the solution
when $f=0$. We can thus write that
\[
(\bld q_h,u_h)=(\bld{q}_h^{\lambda},
u_h^{\lambda})+ (\bld{q}_h^{f},u_h^{f}).
\]

\subsection{The global problem}
If we now take $\lambda:=\widehat{u}_h$, we see that $(\bld q_h,u_h)$
is expressed in terms of $\widehat{u}_h$ (and $f$). We can thus eliminate those
two unknowns from the equations and solve for $\widehat{u}_h$ only. The global problem that determines $\widehat{u}_h$ is not difficult to find.

We have that $\widehat{u}_h\in M_h(g)$ must satisfy
\[
a_h(\widehat{u}_h,\mu)=b_h(\mu)\qquad\forall\mu\in M_h(0),
\]
where
\begin{alignat*}{1}
a_h(\lambda,\mu)&:=
-\bint{\bld{q}^{\lambda}_h\cdot\bld{n}}{\mu}{\dOh} - \bint{\tau(u_h^{\lambda} - \lambda)}{\mu}{\dOh},
\\
b_h(\mu)&:=
\bint{\bld{q}^f_h\cdot\bld{n}}{\mu}{\dOh} + \bint{\tau u^f_h}{\mu}{\dOh}.
\end{alignat*}
Indeed, note that the definition of $M_h(g)$
incorporates the boundary condition \eqref{cd_hdg_eq3}, and that the last equation
is nothing but a rewriting of the transmission condition \eqref{cd_hdg_eq4} by observing that 
\[\bint{\bld \beta \cdot \bld n\lambda}{\mu}{\dOh} = 0,\quad \forall \lambda\in M_h(g), \forall \mu\in M_h(0).\]

\subsection{A characterization of  the approximate solution}
The above results suggest the following charaterization of the
approximate solution of the HDG method. We leave the proof to the interesed readers as an exercise, see also 
\cite{CockburnDongGuzmanMarcoRiccardo09,NguyenPeraireCockburnHDGLCD09}.

\begin{theorem}\label{character}
The approximate solution of the HDG method
satisfies
\[
(\bld q_h,u_h)=(\bld{q}_h^{\widehat{u}_h},
u_h^{\widehat{u}_h})+ (\bld{q}_h^{f},u_h^{f}).
\]
Moreover, $\widehat{u}_h\in M_h(g)$ is the solution of
\begin{alignat}{2}
\label{hybrid-ee} 
a_h (\widehat{u}_h,\mu)  =&b_h(\mu) \qquad \forall \mu\in M_h(0).
\end{alignat}
Also, we have that
\begin{align*}
a_h(\lambda,\mu) &=\; (\epsilon^{-1}\,\bld{q}^\lambda_h,{\bld q}^\mu_h)_\Oh
-(u_h^\lambda, \nabla\cdot (\bld \beta u^\mu_h))_\Oh +\bint{\bld \beta \cdot \bld n\lambda}{ u^\mu_h}{\dOh} 
+\bint{\tau(u^\lambda_h- \lambda)}{ u^\mu_h- \mu}{\dOh},\\
b_h(\mu) & = \; (f,u_h^\mu)_\Oh + \bint{\bld \beta \cdot \bld n u_h^f}{ \mu}{\dOh} + (u_h^f, \bld\beta \cdot \nabla u_h^\mu)_\Oh
- (\nabla u_h^f, \bld\beta  u_h^\mu)_\Oh.
\end{align*}
\end{theorem}
%{\color{red}
%Let us close this Section by mention that if we scale the global unknowns from $\widehat{u}_h$ to $\widetilde{u}_h: =\tau^{1/2}\widehat{u}_h$, we can actually
%prove that the condition number of  the scaled linear system is $\mathcal{O}(h^2)$, see Appendix \ref{appendix} for details.}

\subsection{The conditioning of the HDG method}
We note that both examples of stabilization function $\tau$ in \eqref{tau1} and \eqref{tau2} on a face $F$ can be very small if 
$\boldsymbol{\beta}\cdot\boldsymbol{n}|_{F}$
and $\epsilon$ are very small. 
In this case, the condition number of the global matrix generated by $a_{h}$ in 
(\ref{hybrid-ee}) might blow up as $\epsilon$ goes to zero.

In order to make the condition number independent of $\epsilon$, we need a new assumption on $\tau$, namely,
\begin{align}
\label{assump_cond}
\inf_{\bld x\in F}\left( \tau - \frac{1}{2}\bld \beta(\bld x)\cdot\bld n\right) \ge &\;C_2 
\min (\dfrac{\epsilon}{h_{F}},1)& \forall 
F\in \dK, \forall K\in \Oh.
\end{align}
If we introduce 
\begin{align}
\label{change_variables}
\tilde{\lambda} =  & \Lambda_{\epsilon}\lambda,\quad \tilde{\mu} = \Lambda_{\epsilon}\mu,\qquad \forall 
\lambda \in M_{h}(g),\mu\in M_{h}(0),\\
\nonumber
&\text{where }\Lambda_{\epsilon}|_{F} = \left(\sup_{x\in F}\vert \boldsymbol{\beta}\cdot\boldsymbol{n}(x)\vert
+\min(\dfrac{\epsilon}{h_{F}},1) \right)^{1/2},\quad \forall F\in\mathcal{E}_{h},
\end{align}
the preferred form for implementation for the HDG method 
is to find $\tilde{\lambda}\in M_{h}({\Lambda_{\epsilon}g})$ satisfying
\begin{align*}
%\label{reduced_system2}
& \tilde{a}_{h}(\tilde{\lambda}, \tilde{\mu}) = b_{h}(\Lambda_{\epsilon}^{-1}\tilde{\mu})
\end{align*}
for all $\tilde{\mu}\in M_{h}(0)$. Here, 
\begin{align}
\label{reduced_matrix2}
\tilde{a}_{h}(\tilde{\lambda},\tilde{\mu}) = a_{h}(\Lambda_{\epsilon}^{-1}\tilde{\lambda}, \Lambda_{\epsilon}^{-1}\tilde{\mu}).
\end{align}

%We have the estimate of the condition number of the global matrix generated 
%by $\tilde{a}_{h}$ in the following Theorem.
We have the following theorem concerning the condition number of the scaled global matrix in \eqref{reduced_matrix2}.
\begin{theorem}
\label{Thm_conditioning}
Let the stabilization function $\tau$
satisfy assumptions \eqref{assumption_tau_general} and \eqref{assump_cond}, and let $\epsilon \leq \mathcal{O}(h)$. 
Let $\kappa$ be the spectral condition number of the global matrix generated by $\tilde{a}_h$ in \eqref{reduced_matrix2}. 
Then there is $h_{0}>0$, 
which is independent of $\epsilon$ and $h$,  such that when $h< h_{0}$,
\begin{align*}
&\kappa \leq Ch^{-2}.
\end{align*}
\end{theorem}
We present a detailed proof of Theorem~\ref{Thm_conditioning} in Appendix~\ref{appendix}.

\begin{remark}
Obviously, assumption (\ref{assump_cond}) is satisfied by the second 
stabilization function (\ref{tau2}) but not by the first one (\ref{tau1}) on meshes that aligned with the direction of $\bld \beta$. 
Theorem~\ref{Thm_conditioning} shows that the condition number of the global matrix of the HDG method 
for convection--dominated diffusion problems 
is the same as that of HDG methods for elliptic problems in \cite{CockburnDuboisGopalakrishnanTan2013}.
\end{remark}

\section{Convergence analysis}
\label{analysis}
In this section, we  prove Theorem~\ref{MainTh1} and Theorem~\ref{MainTh2}.
We begin by introducing the following bilinear form:
\begin{align}
\label{bilinear_form}
 B((\boldsymbol{q},u,\lambda),(\boldsymbol{r},w,\mu))%\\
= & \;(\epsilon^{-1}\boldsymbol{q},\boldsymbol{r})_{\mathcal{T}_{h}}-(u,\nabla\cdot\boldsymbol{r})_{\mathcal{T}_{h}}
+\langle \lambda, \boldsymbol{r}\cdot \boldsymbol{n}\rangle_{\partial\mathcal{T}_{h}}\\
\nonumber
& -(\boldsymbol{q}+\boldsymbol{\beta}u,\nabla w)_{\mathcal{T}_{h}}
+\langle (\boldsymbol{q}+\boldsymbol{\beta}\lambda)\cdot\boldsymbol{n}
+\tau (u-\lambda),w\rangle_{\partial\mathcal{T}_{h}}\\
\nonumber
& -((\nabla\cdot\boldsymbol{\beta})u,w)_{\mathcal{T}_{h}}
-\langle(\boldsymbol{q}+\boldsymbol{\beta}\lambda)\cdot\boldsymbol{n}
+\tau (u-\lambda), \mu
\rangle_{\partial\mathcal{T}_{h}},
\end{align}
for all 
$(\boldsymbol{q},u,\lambda) \text{ and } (\boldsymbol{r},w,\mu)\in H^{1}(\mathcal{T}_{h};\mathbb{R}^{d})\times H^{1}(\mathcal{T}_{h})\times L^{2}(\mathcal{E}_{h})$.
It's easy to see that the HDG method \eqref{cd_hdg_eqs} can be recasted in the following compact form:
Find $(\boldsymbol{q}_h, u_h, \widehat{u}_h)\in \boldsymbol{V}_{h}\times W_{h}\times M_{h}(g)$ so that 
\begin{align}
\label{compact}
  B((\boldsymbol{q}_h,u_h,\widehat{u}_h),(\boldsymbol{r},w,\mu)) & = (f, w)_{\mathcal{T}_{h}},
\end{align}
for all $(\boldsymbol{r},w,\mu)\in \boldsymbol{V}_{h}\times W_{h}\times M_{h}(0)$.
%We note that, by taking $\boldsymbol{r}_{h}=\boldsymbol{q}_{h},w_{h}=u_{h},\widehat{\mu}_{h}=\widehat{u}_{h}$, thanks to the definition of numerical flux \eqref{cd_hdg_eq5}, the bilinear form is nothing but the sum of the equations \eqref{cd_hdg_eq1} - \eqref{cd_hdg_eq4}  defining the HDG method.
%\textcolor{red}{For our analysis, we denote $\Vert \cdot \Vert_{\mathcal{T}_{h}}$ and $\Vert \cdot\Vert_{\partial\mathcal{T}_{h}}$ by 
%the $L^{2}(\mathcal{T}_{h})$ norm and $L^{2}(\partial\mathcal{T}_{h})$ norm, respectively.}

\subsection{Stability property for the HDG method}
It is well known that we have the following result regarding the stability of the convection--dominated diffusion problem,\begin{align}\label{stablity}
\epsilon \| \nabla u \|^2_{L^{2}(\Omega)} +   \| u \|^2_{L^{2}(\Omega)} \leq C  \| f \|^2_{L^{2}(\Omega)},
\end{align}
provided that $\boldsymbol{\beta}$ satisfies assumption \eqref{beta_assumps} and $g = 0$ on $\partial \Omega$, see \cite{AyusoMarini:cdf}. 
On the other hand, by taking $(\boldsymbol{r},w,\mu) = (\boldsymbol{q}_h,u_h,\widehat{u}_h)$ in \eqref{compact}, the standard energy argument 
only gives the following estimate:
\begin{align*}
 (\epsilon^{-1}\boldsymbol{q}_h,\boldsymbol{q}_h)_{\mathcal{T}_{h}}
 +\langle (\tau-\frac{1}{2} \boldsymbol{\beta}\cdot \boldsymbol{n})(u_h-\widehat{u}_h),u_h-\widehat{u}_h\rangle -\frac{1}{2}((\nabla\cdot\boldsymbol{\beta})u_h,u_h)_{\mathcal{T}_{h}}
= (f, u_h)_{\mathcal{T}_h}.
\end{align*}
Hence, we do not have control of the $L^2$--norm of $u_h$ by the standard energy argument when the velocity field $\boldsymbol\beta$ is divergence-free.
The main idea of our stability analysis is to achieve the control of the $L^2$--norm of $u_h$ 
by mimicking the proof of the stability property \eqref{stablity} at the discrete level. We shall proceed in the following three steps.

%We will show stability in the following norm
%\begin{align}
% \label{norm}
%\|| (\boldsymbol{r}, w, \mu)|\|_{e} := \left(\epsilon^{-1}\|\boldsymbol{r}\|^2_{\mathcal{T}_h} + \|w\|^2_{\mathcal{T}_h} + 
%\| |\tau-\frac{1}{2} \boldsymbol{\beta}\cdot \boldsymbol{n}|^{1/2}(w-\mu)\|_{\partial\mathcal{T}_h} \right)^{1/2}.
%\end{align}

{\bf Step One.}
In view of assumption \eqref{assump_beta0}, we define a function
\begin{equation}
\label{weight_function}
\varphi := e^{-\psi}+\chi,
\end{equation}
where $\chi$ is a positive constant to be determined later.
Mimicking the proof of stability results carried out for the continuous problem, we obtain the following lemma.
\begin{lemma}
\label{lemma_ideal_infsup}
Let $\varphi$ be given in \eqref{weight_function} where $\chi\geq 1+2b_{0}^{-1}\Vert e^{-\psi}\Vert_{L^{\infty}(\Omega)}
\cdot \Vert \nabla\psi\Vert_{L^{\infty}(\Omega)}^{2}$. Also, let $\tau$ satisfy assumption \eqref{assump_tau_00}.
Then for all 
$(\boldsymbol{q}_{h},u_{h},{\lambda}_{h})\in \boldsymbol{V}_{h}\times W_{h}\times M_{h}(0)$, the following inequality holds
\begin{align*}
%\label{ideal_infsup}
%B((\boldsymbol{q}_{h},u_{h},{\lambda}_{h}),(\boldsymbol{q},u,\lambda)) \geq & C\left[\epsilon^{-1}\Vert \boldsymbol{q}_{h}\Vert_{\mathcal{T}_{h}}^{2}+\Vert u_{h}\Vert_{\mathcal{T}_{h}}^{2}
%+ \chi\Vert \vert \tau-\frac{1}{2}\boldsymbol{\beta}\cdot\boldsymbol{n}\vert^{1/2}
%(w_{h}-{\lambda}_{h})\Vert_{\partial \mathcal{T}_{h}}^{2}\right]
B((\boldsymbol{q}_{h},u_{h},{\lambda}_{h}),(\boldsymbol{q}_{\varphi},u_\varphi,\lambda_\varphi)) \geq & \; C \vertiii{ (\boldsymbol{q}_h, u_h, \lambda_h)}_{e}^2, 
\end{align*}
where 
$\boldsymbol{q}_\varphi=\varphi \boldsymbol{q}_{h}$, $u_\varphi = \varphi u_{h}$ and
$\lambda_\varphi = \varphi {\lambda}_{h}  $.
%\begin{align*}
%\mu = 
%\begin{cases} 
%-\varphi \left[ (\boldsymbol{r}_{h}+\boldsymbol{\beta}\widehat{\mu}_{h})\cdot\boldsymbol{n}
% + \tau (w_{h}-\widehat{\mu}_{h})\right] 
%& \text{on $\partial\Omega$,}
%\\
%-\varphi \widehat{\mu}_{h} 
%&\text{on $\partial\mathcal{T}_{h}\backslash \partial\Omega$.}
%\end{cases}
%\end{align*}
\end{lemma}

\begin{proof}
With $(\boldsymbol{q}_\varphi,u_\varphi,\lambda_\varphi)$ given above, we have that
\begin{align*}
B((\boldsymbol{q}_{h},u_{h},{\lambda}_{h}),(\boldsymbol{q}_\varphi,u_\varphi,\lambda_\varphi))
=  (\epsilon^{-1}\boldsymbol{q}_{h},\varphi\boldsymbol{q}_{h})_{\mathcal{T}_{h}} + T_1 + T_2 + T_3 ,
\end{align*}
where
\begin{align*}
T_1  = &\; -(u_{h},\nabla\cdot (\varphi\boldsymbol{q}_{h}))_{\mathcal{T}_{h}}
+\langle {\lambda}_{h}, \varphi\boldsymbol{q}_{h}\cdot \boldsymbol{n}\rangle_{\partial\mathcal{T}_{h}}
 -(\boldsymbol{q}_{h},\nabla (\varphi u_{h}))_{\mathcal{T}_{h}} + \langle \boldsymbol{q}_{h}\cdot\boldsymbol{n},\varphi( u_{h} - {\lambda}_{h}) \rangle_{\partial\mathcal{T}_{h}} \\
T_2  = &\; -(\boldsymbol{\beta}u_{h},\nabla(\varphi u_{h}))_{\mathcal{T}_{h}} 
- 
((\nabla\cdot\boldsymbol{\beta})u_{h},\varphi u_{h})_{\mathcal{T}_{h}}
+\langle \boldsymbol{\beta}\cdot\boldsymbol{n} {\lambda}_{h},\varphi u_{h} \rangle_{\partial\mathcal{T}_{h}}
\\
T_3  =  &\; \langle
\tau (u_{h}-{\lambda}_{h}),\varphi( u_{h} - {\lambda}_{h}) \rangle_{\partial\mathcal{T}_{h}} .
\end{align*}

By integration by parts, we obtain
\begin{align*}
T_1  = &\; -(u_{h},\nabla\cdot (\varphi\boldsymbol{q}_{h}))_{\mathcal{T}_{h}}
+\langle {\lambda}_{h}, \varphi\boldsymbol{q}_{h}\cdot \boldsymbol{n}\rangle_{\partial\mathcal{T}_{h}}
 -(\boldsymbol{q}_{h},\nabla (\varphi u_{h}))_{\mathcal{T}_{h}} + \langle \boldsymbol{q}_{h}\cdot\boldsymbol{n},\varphi( u_{h} - {\lambda}_{h}) \rangle_{\partial\mathcal{T}_{h}} \\
= & \; -(u_{h},\nabla\varphi \cdot\boldsymbol{q}_{h})_{\mathcal{T}_{h}} - (\varphi u_{h},\nabla\boldsymbol{q}_{h})_{\mathcal{T}_{h}}
 -(\boldsymbol{q}_{h},\nabla (\varphi u_{h}))_{\mathcal{T}_{h}} + \langle \boldsymbol{q}_{h}\cdot\boldsymbol{n},\varphi u_{h} \rangle_{\partial\mathcal{T}_{h}} \\
= & \; -(u_{h},\nabla\varphi \cdot\boldsymbol{q}_{h})_{\mathcal{T}_{h}} \\
=&\;  (u_{h},e^{-\psi}\nabla\psi \cdot\boldsymbol{q}_{h})_{\mathcal{T}_{h}}
\\
T_2  = &\; -(\boldsymbol{\beta}u_{h},\nabla(\varphi u_{h}))_{\mathcal{T}_{h}} 
- 
((\nabla\cdot\boldsymbol{\beta})u_{h},\varphi u_{h})_{\mathcal{T}_{h}}
+\langle \boldsymbol{\beta}\cdot\boldsymbol{n} {\lambda}_{h},\varphi u_{h} \rangle_{\partial\mathcal{T}_{h}}\\
=&\; -(\boldsymbol{\beta}\cdot\nabla\varphi, u_h^2 )_{\mathcal{T}_{h}} -(\boldsymbol{\beta} \varphi, \nabla\frac{u_h^2}{2} )_{\mathcal{T}_{h}}
-
((\nabla\cdot\boldsymbol{\beta})\varphi ,u_{h}^2)_{\mathcal{T}_{h}} + \langle \boldsymbol{\beta}\cdot\boldsymbol{n} {\lambda}_{h},\varphi u_{h} \rangle_{\partial\mathcal{T}_{h}}\\
=&\; -\frac{1}{2}(\boldsymbol{\beta}\cdot\nabla\varphi, u_h^2 )_{\mathcal{T}_{h}} - \frac{1}{2}\langle \boldsymbol{\beta}\cdot\boldsymbol{n} {u}_{h},\varphi u_{h} \rangle_{\partial\mathcal{T}_{h}}-
\frac{1}{2}((\nabla\cdot\boldsymbol{\beta})\varphi ,u_{h}^2)_{\mathcal{T}_{h}} + \langle \boldsymbol{\beta}\cdot\boldsymbol{n} {\lambda}_{h},\varphi u_{h} \rangle_{\partial\mathcal{T}_{h}}\\
=&\; \frac{1}{2}(\boldsymbol{\beta}\cdot\nabla\psi, e^{-\psi}u_h^2 )_{\mathcal{T}_{h}} -\frac{1}{2}((\nabla\cdot\boldsymbol{\beta})\varphi ,u_{h}^2)_{\mathcal{T}_{h}} 
- \frac{1}{2}\langle \boldsymbol{\beta}\cdot\boldsymbol{n} (u_h-{\lambda}_{h}),\varphi (u_{h}-\lambda_h) \rangle_{\partial\mathcal{T}_{h}},
\end{align*}
where in the last step, we used $\langle \boldsymbol{\beta}\cdot\boldsymbol{n} {\lambda}_{h},\varphi \lambda_{h} \rangle_{\partial\mathcal{T}_{h}} = 0 $ due to the 
fact that $\lambda_{h}$ is single valued on the interior faces and $\lambda_{h}=0$ on $\partial\Omega$.

Combining $T_1$, $T_2$ and $T_3$, we have that
\begin{align*}
B((\boldsymbol{q}_{h},u_{h},{\lambda}_{h}),(\boldsymbol{q}_\varphi,u_\varphi,\lambda_\varphi))
= & \;(\epsilon^{-1}\boldsymbol{q}_{h},\varphi\boldsymbol{q}_{h})_{\mathcal{T}_{h}}
+(u_{h},e^{-\psi}\nabla\psi\cdot\boldsymbol{q}_{h})_{\mathcal{T}_{h}}\\
&+\dfrac{1}{2}([\boldsymbol{\beta}\cdot \nabla\psi]u_{h},e^{-\psi}u_{h})_{\mathcal{T}_{h}}
-\dfrac{1}{2}((\nabla\cdot\boldsymbol{\beta})u_{h},\varphi u_{h})_{\mathcal{T}_{h}}\\
&+\langle(\tau-\dfrac{1}{2}\boldsymbol{\beta}\cdot\boldsymbol{n})\varphi(u_{h}-\lambda_{h}),
u_{h}-\lambda_{h}\rangle_{\partial\mathcal{T}_{h}}
\end{align*}
Invoking assumptions \eqref{beta_assumps} and \eqref{assump_beta1}, and $\varphi\geq \chi$, we obtain
\begin{align*}
B((\boldsymbol{q}_{h},u_{h},\lambda_{h}),(\boldsymbol{q}_\varphi,u_\varphi,\lambda_\varphi))
\geq & (\epsilon^{-1}\boldsymbol{q}_{h},\varphi\boldsymbol{q}_{h})_{\mathcal{T}_{h}}
+(u_{h},e^{-\psi}\nabla\psi\cdot\boldsymbol{q}_{h})_{\mathcal{T}_{h}}
+\dfrac{1}{2}b_{0}(u_{h},e^{-\psi}u_h)_{\mathcal{T}_{h}}\\
&+\langle  (\tau-\frac{1}{2}\boldsymbol{\beta}\cdot\boldsymbol{n})\varphi(u_{h}-\lambda_{h}),
u_{h}-\lambda_{h}\rangle_{\partial\mathcal{T}_{h}}\\
\geq & \chi(\epsilon^{-1}\boldsymbol{q}_{h},\boldsymbol{q}_{h})_{\mathcal{T}_{h}}
+(u_{h},e^{-\psi}\nabla\psi\cdot\boldsymbol{q}_{h})_{\mathcal{T}_{h}}
+\dfrac{1}{2}b_{0}(u_{h},e^{-\psi}u_{h})_{\mathcal{T}_{h}}\\
&+\chi \langle  (\tau-\frac{1}{2}\boldsymbol{\beta}\cdot\boldsymbol{n})(u_{h}-\lambda_{h}),
u_{h}-\lambda_{h}\rangle_{\partial\mathcal{T}_{h}}.
\end{align*}
Using the Cauchy--Schwartz inequality, we have
\begin{equation*}
(u_{h},e^{-\psi}\nabla\psi\cdot\boldsymbol{q}_{h})_{\mathcal{T}_{h}} 
\leq \dfrac{1}{2}\left[ 
\delta^{-1}\Vert \nabla\psi\Vert_{L^{\infty}(\Omega)}^{2}(e^{-\psi}\boldsymbol{q}_{h},
\boldsymbol{q}_{h})_{\mathcal{T}_{h}}
+\delta(e^{-\psi}u_{h},u_{h})_{\mathcal{T}_{h}}^{2}
\right]
\end{equation*}
for any $\delta>0$.
Taking $\chi \geq 1+2b_{0}^{-1}\Vert e^{-\psi}\Vert_{L^{\infty}(\Omega)}\cdot \Vert \nabla\psi\Vert_{L^{\infty}(\Omega)}^{2}$ 
and $\delta = b_{0}/2$,
we get
\begin{align*}
B((\boldsymbol{q}_{h},u_{h},\lambda_{h}),(\boldsymbol{q}_\varphi,u_\varphi,\lambda_\varphi))
\geq & \epsilon^{-1}\dfrac{\chi}{2}(\boldsymbol{q}_{h},\boldsymbol{q}_{h})_{\mathcal{T}_{h}}
+\dfrac{b_{0}}{4}(e^{-\psi}u_{h},u_{h})_{\mathcal{T}_{h}}\\
& +\chi\Vert \vert \tau-\frac{1}{2}\boldsymbol{\beta}\cdot\boldsymbol{n}\vert^{1/2} (u_{h}-\lambda_{h})
\Vert_{\partial\mathcal{T}_{h}}^{2}.
\end{align*}
To complete the proof, we simply absorb $\chi$, $e^{-\psi}$ and $b_0$ into the generic constant $C$.
\end{proof}

{\bf Step Two.} 
%\red{
We note that the test function $(\boldsymbol{q}_\varphi,u_\varphi,\lambda_\varphi)=(\varphi \boldsymbol{q}_{h}, \varphi u_{h}, \varphi {\lambda}_{h})  $
 in Lemma \ref{lemma_ideal_infsup} is not in the discrete space 
$ 
\boldsymbol{V}_h \times W_h \times M_h(0). 
$
To establish a discrete stability property, we shall consider taking the discrete
 test functions as a projection of $(\boldsymbol{q}_\varphi,u_\varphi,\lambda_\varphi)$ 
onto the spaces $ \boldsymbol{V}_h \times W_h \times M_h$, denoted by $\pv \bld q_\varphi, \varPi_h u_\varphi, 
P_M\lambda_\varphi$. 
Here  $P_M$ is the $L^2$--projection  onto $M_h$. And $\pv$ and $\varPi_h$ are the projections from $H^1(\Oh;\mathbb{R}^d)$ and 
$H^1(\Oh)$ onto 
$\bld V_h$ and $W_h$ respectively satisfying
\begin{subequations}
  \label{eq:projI}
  \begin{align}
    \label{eq:projI1}
    (\pv \bld{q}, \bld{v} )_K
    &= ( \bld{q}, \bld{v} )_K
    && \forall\; \bld{v} \in \bpol{k-1}{K},
    \\
    \label{eq:projI2}
    \bint{\pv \bld q\cdot \bld n}{\mu}{F} &= \bint{\bld q\cdot \bld n}{\mu}{F}
    && \forall \; \mu\in \pol{k}{F},\;\;\forall \; F\in\dK\backslash F_K^s,\\
    \label{eq:projII1}
    (\varPi_h u, w )_K
    &= ( u, w )_K
    && \forall\; w \in \pol{k-1}{K},
    \\
    \label{eq:projII2}
    \bint{\varPi_h u}{\mu}{F_K^\star} &= \bint{u}{\mu}{F_K^\star},&& \forall \; \mu\in \pol{k}{F_K^\star}.
  \end{align}
\end{subequations}
where $F_K^s$ and $F_K^\star$ are defined in \eqref{notation_F}. We have the following optimal approximation property for $\pv$ and $\varPi_h$, whose proof was
available in \cite[Proposition 2.1]{CockburnDongGuzman2008}.
\begin{lemma}
 \label{approx-q}
Assume that $\bld q \in H^{s+1}(K; \mathbb{R}^d)$ for $s\in [0,k]$ on an element $K\in \Oh$. Then
\[
 \|\pv \bld q-\bld q\|_{K} \le C\,h^{s+1}|\bld q|_{H^{s+1}(K;\mathbb{R}^d)}.
\]
Assume that $u \in H^{s+1}(K)$ for $s\in [0,k]$ on an element $K\in \Oh$. Then
\[
 \|\varPi_h u-u\|_{K} \le C\,h^{s+1}|u|_{H^{s+1}(K)}.
\]
\end{lemma}

We also need to estimate the difference between $\boldsymbol{q}$, $u$ and  $\lambda$ and their corresponding projections. 
Such an estimate is established in the following lemma. We refer the readers to Lemma~$4.2$ in \cite{AyusoMarini:cdf} for a detailed proof.
\begin{lemma}
\label{lemma_non_constant_ineqs}
Let $K\in\mathcal{T}_{h}$ and $\eta\in C^{1}(\bar{K})\cap W^{k+1,\infty}(K)$. Then, for any $(\bld v, v)\in P_k(K;\mathbb{R}^d)
\times P_k(K) $ and $\chi\in\mathbb{R}$, 
%\begin{subequations}
%\label{non_constant_ineqs}
\begin{align*}
%\label{non_constant_ineq1}
& \Vert \pv ((\eta+\chi) \bld v)-(\eta+\chi) \bld v\Vert_{K}\leq C h_{K} \Vert \eta\Vert_{W^{k+1,\infty}(K)}\Vert\bld v\Vert_{K},\\
%\label{non_constant_ineq2}
& \Vert \pv ((\eta+\chi) \bld v)-(\eta+\chi) \bld v\Vert_{F}\leq C h_{K}^{1/2} \Vert \eta\Vert_{W^{k+1,\infty}(K)}\Vert\bld v\Vert_{K},
\qquad \forall F\in \dK,\\
%\label{non_constant_ineq3}
& \Vert \varPi_h ((\eta+\chi) v)-(\eta+\chi) v\Vert_{K}\leq C h_{K} \Vert \eta\Vert_{W^{k+1,\infty}(K)}\Vert v\Vert_{K},\\
%\label{non_constant_ineq4}
& \Vert \varPi_h ((\eta+\chi) v)-(\eta+\chi) v\Vert_{F}\leq C h_{K}^{1/2} \Vert \eta\Vert_{W^{k+1,\infty}(K)}\Vert v\Vert_{K},
\qquad \forall F\in \dK.%\\
%\label{non_constant_ineq5}
%& \Vert P_{M}((\eta+\chi) v) - (\eta+\chi) v \Vert_{F}\leq Ch_{F}\Vert \eta \Vert_{W^{k+1,\infty}(F)}\Vert v\Vert_{F}.
\end{align*}
%\end{subequations}
\end{lemma}

Now, we go back to the stability estimate in Lemma \ref{lemma_ideal_infsup}, and divide the left hand side of the inequality into two terms, namely,
\begin{align*}
B((\boldsymbol{q}_{h},u_{h},{\lambda}_{h}),(\boldsymbol{q}_\varphi,u_\varphi,\lambda_\varphi)) 
&\; = 
B((\boldsymbol{q}_{h},u_{h},{\lambda}_{h}),(\pv \boldsymbol{q}_\varphi, \varPi_h u_\varphi, P_M \lambda_\varphi))
\\
&\; +
B((\boldsymbol{q}_{h},u_{h},{\lambda}_{h}),((\mathsf{Id} - \pv) \boldsymbol{q}_\varphi, 
(\mathsf{Id} - \varPi_h) u_\varphi, (\mathsf{Id} - P_M)\lambda_\varphi)).
\end{align*}

{\bf Step Three. }
 We define the union of faces to simplify the presentation:
\begin{align*}
 \partial \mathcal{T}_h^\star &: =\; \cup_{K\in \Oh}\cup_{F\in \dK\backslash F_K^\star} F,\\
 \partial \mathcal{T}_h^s &: =\; \cup_{K\in \Oh}F_K^s,
\end{align*}
where $F_K^\star$ and $F_K^s$ are defined in \eqref{notation_F}.
Now, we are ready to derive the discrete stability result for the HDG method.
\begin{lemma}
\label{lemma_practical_infsup}
Let $\tau$ satisfies assumptions \eqref{assumption_tau_general}, 
then there exists $h_0$, independent of $\epsilon$, so that
for any $h<h_0$, we have the following stability estimate:
for all 
$(\boldsymbol{q}_{h},u_{h},{\lambda}_{h})\in \boldsymbol{V}_{h}\times W_{h}\times M_{h}(0)$,
\begin{align*}
%\label{practical_infsup}
\sup_{0\not =(\boldsymbol{r}_h,w_h,\mu_h)\in \bld V_h\times W_h\times M_h(0)}
\frac{B((\boldsymbol{q}_{h},u_{h},{\lambda}_{h}),(\boldsymbol{r}_h,w_h,\mu_h))}{\|| (\boldsymbol{r}_h, w_h, \mu_h)|\|_{e}} 
\geq  \; C \|| (\boldsymbol{q}_h, u_h, \lambda_h)|\|_{e}.
\end{align*}
\end{lemma}

\begin{proof}
For any $(\boldsymbol{r},w,\mu)\in H^1(\Oh;\mathbb{R}^d)\times H^1(\Oh)
\times L^2(\Eh)$ with $\mu=0 $ on $\dO$, define 
\[
\bld {\delta r} := \bld {r}-\pv \bld {r}, \delta w := w - \varPi_hw, \delta\mu := \mu-P_M\mu.
\]
Using integration by parts and the definition of 
the projections, we get
\begin{align*}
%\label{proj-err1}
& B((\boldsymbol{q}_h,u_h,\lambda_h),(\boldsymbol{\delta r},\delta w,\delta\mu))\\
= & (\epsilon^{-1}\boldsymbol{q}_h,\boldsymbol{\delta r})_{\mathcal{T}_{h}}-(u_h,\nabla\cdot\boldsymbol{\delta r})_{\mathcal{T}_{h}}
+\langle \lambda_h, \boldsymbol{\delta r}\cdot \boldsymbol{n}\rangle_{\partial\mathcal{T}_{h}} \nonumber\\
& -(\boldsymbol{q}_h+\boldsymbol{\beta}u_h,\nabla \delta w)_{\mathcal{T}_{h}}
+\langle (\boldsymbol{q}_h+\boldsymbol{\beta}\lambda_h)\cdot\boldsymbol{n}
+\tau (u_h-\lambda_h),\delta w\rangle_{\partial\mathcal{T}_{h}}\nonumber\\
& -((\nabla\cdot\boldsymbol{\beta})u_h,\delta w)_{\mathcal{T}_{h}}
-\langle(\boldsymbol{q}_h+\boldsymbol{\beta}\lambda_h)\cdot\boldsymbol{n}
+\tau (u_h-\lambda_h), \delta\mu
\rangle_{\partial\mathcal{T}_{h}}\nonumber\\
= & (\epsilon^{-1}\boldsymbol{q}_h,\boldsymbol{\delta r})_{\mathcal{T}_{h}}+(\nabla u_h,\boldsymbol{\delta r})_{\mathcal{T}_{h}}
+\langle \lambda_h - u_h, \boldsymbol{\delta r}\cdot \boldsymbol{n}\rangle_{\partial\mathcal{T}_{h}}\nonumber\\
& +(\nabla\cdot \boldsymbol{q}_h,\delta w)_{\mathcal{T}_{h}} +(\bld \beta\cdot\nabla u_h,\delta w)_{\mathcal{T}_{h}}
+\langle (\tau- \bld \beta\cdot\bld n) (u_h-\lambda_h),\delta w\rangle_{\partial\mathcal{T}_{h}}\nonumber\\
&
-\langle(\boldsymbol{q}_h+\boldsymbol{\beta}\lambda_h)\cdot\boldsymbol{n}
+\tau (u_h-\lambda_h), \delta\mu
\rangle_{\partial\mathcal{T}_{h}}\nonumber\\
= & (\epsilon^{-1}\boldsymbol{q}_h,\boldsymbol{\delta r})_{\mathcal{T}_{h}}
+\langle \lambda_h - u_h, \boldsymbol{\delta r}\cdot \boldsymbol{n}\rangle_{\partial\mathcal{T}_{h}^s}
+((\bld  \beta -\bld P_{0,h}\bld \beta)\cdot\nabla u_h,\delta w)_{\mathcal{T}_{h}}\nonumber\\
& +\langle (\tau- \bld \beta\cdot \bld n)(u_h-\lambda_h),\delta w\rangle_{\partial\mathcal{T}_{h}^\star}
-\langle \bld \beta\cdot\bld n (u_h-\lambda_h),\delta w\rangle_{\partial\mathcal{T}_{h}\backslash\partial\mathcal{T}_{h}^\star},\nonumber
\end{align*}
where $\bld P_{0,h}$ is the vectorial piecewise-constant projection. 
%So we have the following estimate
%\begin{align*}
%& B((\boldsymbol{q}_h,u_h,\lambda_h),(\boldsymbol{\delta r},\delta w,\delta\mu))\\
%= & (\epsilon^{-1}\boldsymbol{q}_h,\boldsymbol{\delta r})_{\mathcal{T}_{h}}
%+\langle \lambda_h - u_h, \boldsymbol{\delta r}\cdot \boldsymbol{n}\rangle_{\partial\mathcal{T}_{h}^\star}\\
%& +((\bld  \beta -\bld P_{0.h}\bld \beta)\cdot\nabla u_h,\delta w)_{\mathcal{T}_{h}}
%+\langle (\tau- \bld \beta\cdot\bld n) (u_h-\lambda_h),\delta w\rangle_{\partial\mathcal{T}_{h}}\\
%\le & \frac{1}{2}h^{1/2}(\epsilon^{-1}\boldsymbol{q}_h,\boldsymbol{q}_h)_{\mathcal{T}_{h}} +
%\frac{1}{2} h^{-1/2}(\epsilon^{-1}\boldsymbol{\delta r},\boldsymbol{\delta r})_{\mathcal{T}_{h}}
%+\frac{1}{2} h^{1/4}\langle|\bld \beta\cdot \bld n| \lambda_h - u_h, \lambda_h - u_h\rangle_{\partial\mathcal{T}_{h}^\star}\\
%& +\frac{1}{2} h^{-1/4}\langle|\bld \beta\cdot \bld n|^{-1} \boldsymbol{\delta r}, \boldsymbol{\delta r}
%\rangle_{\partial\mathcal{T}_{h}^\star}\\
%& +((\bld  \beta -\bld P_{0.h}\bld \beta)\cdot\nabla u_h,\delta w)_{\mathcal{T}_{h}}
%+\langle (\tau- \bld \beta\cdot\bld n) (u_h-\lambda_h),\delta w\rangle_{\partial\mathcal{T}_{h}}
%\end{align*}

Now, we take $(\bld r, w,\mu) = (\boldsymbol{q}_\varphi,u_\varphi,\lambda_\varphi)$ as in Lemma \ref{lemma_ideal_infsup}. By  Cauchy--Schwartz inequality and the approximation 
results in Lemma \ref{lemma_non_constant_ineqs},
we have 
\begin{align*}
(\epsilon^{-1}\boldsymbol{q}_h,\boldsymbol{\delta q}_\varphi)_{\mathcal{T}_{h}}\le &\;
\|\epsilon^{-1/2}\boldsymbol{q}_h\|_{\Oh}\|\epsilon^{-1/2}\boldsymbol{\delta q}_\varphi\|_{\mathcal{T}_{h}} \\
\le &\; C h \|\epsilon^{-1/2}\boldsymbol{q}_h\|_{\Oh}^2,%\\
\end{align*}
 \begin{align*}
 \langle \lambda_h - u_h, \boldsymbol{\delta q}_\varphi\cdot \boldsymbol{n}\rangle_{\partial\mathcal{T}_{h}^s} 
\le &\; \left\||\tau-\frac{1}{2}\bld \beta\cdot\bld n|^{1/2} (\lambda_h - u_h)\right\|_{\partial\mathcal{T}_{h}^s}
\left\||\tau-\frac{1}{2}\bld \beta\cdot\bld n|^{-1/2} \bld{\delta q}_\varphi\right\|_{\partial\mathcal{T}_{h}^s}\\
\le &\; C \left(\frac{\epsilon}{\tau^{\bld v}}\right)^{1/2}\left\||\tau - \frac{1}{2}\bld \beta\cdot\bld n|^{1/2}
 (\lambda_h - u_h)\right\|_{\partial\mathcal{T}_{h}^s} \|\epsilon^{-1/2}\bld{\delta q}_\varphi\|_{
\partial\mathcal{T}_h^s} \\
%\quad \quad
%\text{ by \eqref{assumption_tau_3}}\\
\le &\; C \left(\frac{\epsilon h}{\tau^{\bld v}}\right)^{1/2}\left\||\tau - \frac{1}{2}\bld \beta\cdot\bld n|^{1/2} 
(\lambda_h - u_h)\right\|_{\partial\mathcal{T}_{h}^s} 
\|\epsilon^{-1/2}\bld{q}_h\|_{\Oh}\\
\le &\; C (h^2+\epsilon h)^{1/2}\left\||\tau - \frac{1}{2}\bld \beta\cdot\bld n|^{1/2} (\lambda_h - u_h)\right\|_{\partial\mathcal{T}_{h}^s} 
\|\epsilon^{-1/2}\bld{q}_h\|_{\Oh}\quad  \text{ by \eqref{assumption_tau_3}},%\\
\end{align*}
\begin{align*}
((\bld  \beta -\bld P_{0,h}\bld \beta)\cdot\nabla u_h,\delta u_\varphi)_{\mathcal{T}_{h}} 
\le &\; C h \|\nabla u_h\|_{\Oh}\|\delta u_\varphi\|_{\Oh}\\
\le &\; Ch \|u_h\|_{\Oh}^2\\
\langle (\tau-\bld \beta\cdot\bld n) (u_h-\lambda_h),\delta u_\varphi\rangle_{\partial\mathcal{T}_{h}^\star}
\le &\;  \left\||\tau-\bld \beta\cdot\bld n|^{1/2} (\lambda_h - u_h)\right\|_{\partial\mathcal{T}_{h}^\star} 
\left\||\tau-\bld \beta\cdot\bld n|^{1/2} \delta u_\varphi\right\|_{\partial\mathcal{T}_{h}^\star}\\
\le &\; C \left\||\tau-\frac{1}{2}\bld \beta\cdot\bld n|^{1/2} (\lambda_h - u_h)\right\|_{\partial\mathcal{T}_{h}^\star} 
\left\| \vert \tau -\bld\beta\cdot \bold n \vert^{1/2} \delta u_\varphi\right\|_{\partial\mathcal{T}_{h}^\star} \quad
\text{ by \eqref{assumption_tau_1}}\\
\le &\; C(h(\tau^w+1))^{1/2} \left\||\tau-\frac{1}{2}\bld \beta\cdot\bld n|^{1/2} (\lambda_h - u_h)
\right\|_{\partial\mathcal{T}_{h}^\star} \|u_h\|_{\Oh}\\
\le &\; C h^{1/2} \left\||\tau-\frac{1}{2}\bld \beta\cdot\bld n|^{1/2} (\lambda_h - u_h)\right\|_{\partial\mathcal{T}_{h}} \|u_h\|_{\Oh} \quad
\text{ by \eqref{assumption_tau_0}}\\
\langle \bld \beta\cdot\bld n (u_h-\lambda_h),\delta u_\varphi
\rangle_{\partial\mathcal{T}_{h}\backslash\partial\mathcal{T}_{h}^\star}
\le &\;  \left\||\bld \beta\cdot\bld n|^{1/2}(\lambda_h - u_h)\right\|_{\partial\mathcal{T}_{h}\backslash\partial\mathcal{T}_{h}^\star} 
\left\|\delta u_\varphi\right\|_{\partial\mathcal{T}_{h}\backslash\partial\mathcal{T}_{h}^\star}\\
\le &\; C \left\||\tau-\frac{1}{2}\bld \beta\cdot\bld n|^{1/2}(\lambda_h - u_h)\right\|_{\partial\mathcal{T}_{h}\backslash\partial\mathcal{T}_{h}^\star} 
\left\|\delta u_\varphi\right\|_{\partial\mathcal{T}_{h}\backslash\partial\mathcal{T}_{h}^\star}\quad
\text{ by \eqref{assumption_tau_1}}\\
\le &\; C h^{1/2}\left\||\tau-\frac{1}{2}\bld \beta\cdot\bld n|^{1/2}(\lambda_h - u_h)\right\|_{\partial\mathcal{T}_{h}} 
\left\|u_h\right\|_{\Oh}.\\
\end{align*}
Summing the above inequalities all together, we get
\begin{align*}
B((\boldsymbol{q}_h,u_h,\lambda_h),(\boldsymbol{\delta q}_\varphi,\delta u_\varphi,\delta\lambda_\varphi))
\le & \;C h^{1/2} \vertiii{(\bld q_h, u_h,\lambda_h)}_e^2.
% + C h^{3/2} \left\|\lambda_h - u_h\right\|_{\partial\mathcal{T}_{h}\backslash\partial\mathcal{T}_{h}^\star} 
%\left\|u_h\right\|_{\Oh}.
\end{align*}
Hence, choosing $h$ sufficiently small, we can ensure that 
\begin{align*}
 B((\boldsymbol{q}_h,u_h,\lambda_h),
(\boldsymbol{\delta q}_\varphi,\delta u_\varphi,\delta\lambda_\varphi)) \le  
\frac{1}{2}B((\boldsymbol{q}_h,u_h,\lambda_h),(\boldsymbol{q}_\varphi,u_\varphi,\lambda_\varphi)).
\end{align*}
Consequently, we obtain 
\begin{align*}
%\label{es1}
 B((\boldsymbol{q}_h,u_h,\lambda_h),(\pv\boldsymbol{ q}_\varphi,P_h u_\varphi,P_M \lambda_\varphi)) \ge C  \vertiii{(\bld q_h, u_h,\lambda_h)}_e^2.
\end{align*}
On the other hand, it is easy to obtain the following estimates
\begin{align*}
%\label{es2}
\vertiii{ (\pv\boldsymbol{ q}_\varphi,P_h u_\varphi,P_M \lambda_\varphi)}_e \le &\;C \vertiii{ (\boldsymbol{ q}_{h},u_{h},\lambda_{h})}_e .
\end{align*}
We conclude the proof by combining these two estimates. %estimates \eqref{es1} and \eqref{es2}.
\end{proof}   
%Let us remark that the same stability result can be derived by relaxing assumption \eqref{assumption_tau_3} 
%to the following one: assume there exists a positive constant $C$ so that
%\begin{align*}
% \inf_{\bld x\in F}\left
%( \tau - \frac{1}{2}\bld \beta(\bld x)\cdot\bld n\right) \ge &\;C \epsilon& \exists F \in \dK ,\forall K\in \Oh.
%\end{align*} 
%In this case, the bound for the second term in \eqref{proj-err1} would be 
%\[
% \langle \lambda_h - u_h, \boldsymbol{\delta q}\cdot \boldsymbol{n}\rangle_{\partial\mathcal{T}_{h}^\star} 
%\le \; C h^{1/2}\left\||\tau - \frac{1}{2}\bld \beta\cdot\bld n|^{1/2} (\lambda_h - u_h)\right\|_{\partial\mathcal{T}_{h}^\star} 
%\|\epsilon^{-1/2}\bld{q}_h\|_{\Oh}. 
%\]

\subsection{The Error equation}
Here, we obtain the equation satisfied by the errors. 
Note that by Galerkin--orthogonality, we have
\begin{align}
\label{galerkin-o}
B((\boldsymbol{q} - \boldsymbol{q}_h ,u - u_h,u - \widehat{u}_h), (\boldsymbol{r},w,\mu)) = 0 \quad \forall (\bld r, w, \mu)\in V_h\times W_h\times M_h(0),
\end{align}
where $(\bld q, u)$ is the exact solution of equations \eqref{cd_first_order}.

We define the following quantities that will be used in the analysis:
%\begin{subequations}
%\label{projection_errors}
\begin{align*}
%\label{projection_error2}
&\eq:= \boldsymbol{q}_{h}- \pv \boldsymbol{q}, \quad \bld{\delta q} = \boldsymbol{q}- \pv \boldsymbol{q}, \\
%\label{projection_error1}
&\eu :=  u_{h} - \varPi_h u,\quad \delta u := u - \varPi_h u, \\
%\label{projection_error3}
&\euhat :=\widehat{u}_{h}- P_{M}u, \quad \widehat{\delta u} = u -P_{M}u .
\end{align*}
%\end{subequations}
%Here, $\Pi_h$ is any bounded projection from $H^{1}(\mathcal{T}_{h})$ onto $W_{h}$, which satisfies
%\begin{align*}
%(\Pi_h v - v, w)_{K}=0,\quad \forall w\in P_{k-1}(K), K\in\mathcal{T}_{h}.
%\end{align*}
Recall that  $\pv$ and $\varPi_h$ are the projections defined in \eqref{eq:projI}, and $P_{M}$ is the $L^{2}$--projection 
from $L^{2}(\mathcal{E}_{h})$ onto $M_{h}$. 

Now, we are ready to present our error equation.
\begin{lemma}
 \label{lemma-error-eq1}
The error equation takes the following form.
\begin{align}
 \label{error-eq1}
B((\eq  ,\eu ,\euhat), (\boldsymbol{r},w,\mu))  %\\
= &\;(\epsilon^{-1}\bld{\delta q},\boldsymbol{r})_{\mathcal{T}_{h}} 
+\langle \boldsymbol{\delta q}\cdot\boldsymbol{n}
,w - \mu \rangle_{\partial\mathcal{T}_{h}^s} -(\boldsymbol{\beta}\,\delta u,\nabla  w)_{\mathcal{T}_{h}}  \\
& -((\nabla\cdot\boldsymbol{\beta})\delta u,w)_{\mathcal{T}_{h}}
+\langle\boldsymbol{\beta}\cdot\boldsymbol{n}\widehat{\delta u}
, w
\rangle_{\partial\mathcal{T}_{h}}+\langle\tau \delta u
, w - \mu
\rangle_{\partial\mathcal{T}_{h}^\star}\nonumber,
\end{align}
for all $(\bld r, w, \mu)\in V_h\times W_h\times M_h(0)$.
\end{lemma}

\begin{proof}
We use the Galerkin--orthogonality \eqref{galerkin-o} and the definition of the projections to prove the result. 
For all  $(\bld r, w, \mu)\in V_h\times W_h\times M_h(0)$, we have
 \begin{align*}
%\label{proj-err2}
B((\eq,\eu,\euhat),(\boldsymbol{r}, w,\mu)) =&\;B((\bld{\delta q},\delta u,\widehat{\delta u}),(\boldsymbol{r},w,\mu)) \\
= & \;(\epsilon^{-1}\bld{\delta q},\boldsymbol{r})_{\mathcal{T}_{h}}-(\delta u,\nabla\cdot\boldsymbol{r})_{\mathcal{T}_{h}}
+\langle \widehat{\delta u}, \boldsymbol{r}\cdot \boldsymbol{n}\rangle_{\partial\mathcal{T}_{h}} \nonumber\\
& -(\bld{\delta q}+\boldsymbol{\beta}\delta u,\nabla w)_{\mathcal{T}_{h}}
+\langle (\bld{\delta q}+\boldsymbol{\beta}\widehat{\delta u})\cdot\boldsymbol{n}
+\tau (\delta u-\widehat{\delta u}),w-\mu\rangle_{\partial\mathcal{T}_{h}}\nonumber\\
& -((\nabla\cdot\boldsymbol{\beta})\delta u_h,w)_{\mathcal{T}_{h}}\nonumber \\
= &\; (\epsilon^{-1}\bld{\delta q},\boldsymbol{r})_{\mathcal{T}_{h}} 
+\langle \boldsymbol{\delta q}\cdot\boldsymbol{n}
,w - \mu \rangle_{\partial\mathcal{T}_{h}^s} -(\boldsymbol{\beta}\,\delta u,\nabla  w)_{\mathcal{T}_{h}}  \nonumber\\
& -((\nabla\cdot\boldsymbol{\beta})\delta u,w)_{\mathcal{T}_{h}}
+\langle\boldsymbol{\beta}\cdot\boldsymbol{n}\widehat{\delta u}
, w - \mu
\rangle_{\partial\mathcal{T}_{h}}+\langle\tau \delta u
, w - \mu
\rangle_{\partial\mathcal{T}_{h}^\star} \nonumber \\
=  &\; (\epsilon^{-1}\bld{\delta q},\boldsymbol{r})_{\mathcal{T}_{h}} 
+\langle \boldsymbol{\delta q}\cdot\boldsymbol{n}
,w - \mu \rangle_{\partial\mathcal{T}_{h}^s} -(\boldsymbol{\beta}\,\delta u,\nabla  w)_{\mathcal{T}_{h}}  \nonumber\\
& -((\nabla\cdot\boldsymbol{\beta})\delta u,w)_{\mathcal{T}_{h}}
+\langle\boldsymbol{\beta}\cdot\boldsymbol{n}\widehat{\delta u}
, w
\rangle_{\partial\mathcal{T}_{h}}+\langle\tau \delta u
, w - \mu
\rangle_{\partial\mathcal{T}_{h}^\star}, \nonumber
\end{align*}
where in the last step we used the fact that $\langle\boldsymbol{\beta}\cdot\boldsymbol{n}\widehat{\delta u}
, \mu
\rangle_{\partial\mathcal{T}_{h}} = 0$ for all $\mu \in M_h(0)$.
\end{proof}

\subsection{The error analysis}
Now, we are ready to prove our main results, Theorem~\ref{MainTh1} and Theorem~\ref{MainTh2}. 
%We show detailed proof for Theorem~\ref{MainTh1} and sketch the proof of the improved results in Theorem \ref{MainTh2}. 
%Actually, the proof of Theorem~\ref{MainTh1} and Theorem~\ref{MainTh2} are identically the same, hence we only show the proof of Theorem~\ref{MainTh1}.
%\begin{proof}
In order to prove Theorem~\ref{MainTh1} and Theorem~\ref{MainTh2}. We only need to bound the right hand side of the error equation \eqref{error-eq1} to get the error estimates.
%Let us choose $\Pi_h$ in \eqref{projection_error1} to be the $L^2$ projection. 
For all  $(\bld r, w, \mu)\in V_h\times W_h\times M_h(0)$, we have
%\begin{subequations}
%\label{error-est}
\begin{align*}
(\epsilon^{-1}\bld{\delta q},\boldsymbol{r})_{\mathcal{T}_{h}}  \le &\; 
\|\epsilon^{-1/2}\bld{\delta q} \|_{\Oh}\|\epsilon^{-1/2}\boldsymbol{r}\|_{\mathcal{T}_{h}},\\
%\label{difficult_term}
 \langle \boldsymbol{\delta q}\cdot\boldsymbol{n}
,w - \mu \rangle_{\partial\mathcal{T}_{h}^s}
\le &\;\left\||\tau-\frac{1}{2}\bld \beta\cdot\bld n|^{-1/2} \bld{\delta q}\right\|_{\partial\mathcal{T}_{h}^s}
 \left\||\tau-\frac{1}{2}\bld \beta\cdot\bld n|^{1/2} (w - \mu)\right\|_{\partial\mathcal{T}_{h}^s}
 \nonumber\\
\le &\; C \left(\frac{\epsilon}{\tau^{\bld v}}\right)^{1/2}\|\epsilon^{-1/2}\bld{\delta q}\|_{\partial\mathcal{T}_{h}^{s}} \left\||
\tau - \frac{1}{2}\bld \beta\cdot\bld n|^{1/2} (w - \mu)\right\|_{\partial\mathcal{T}_{h}^s},\\
%\nonumber\\
%\le &\; C h^{1/2}\|\epsilon^{-1/2}\bld{\delta q}\|_{\partial\mathcal{T}_{h}^{s}} \left\||
%\tau - \frac{1}{2}\bld \beta\cdot\bld n|^{1/2} (w - \mu)\right\|_{\partial\mathcal{T}_{h}^s}\quad\text{ {by \eqref{assumption_tau_3}}} \\
(\boldsymbol{\beta}\,\delta u,\nabla  w)_{\mathcal{T}_{h}}
= &\;\left((\boldsymbol{\beta}-\bld P_{0,h}\bld \beta)\,\delta u,\nabla  w\right)_{\mathcal{T}_{h}} \nonumber\\
\le &\;  C h \|\delta u\|_{\Oh}\|\nabla w\|_{\Oh}\nonumber\\
\le &\; C \|\delta u\|_{\Oh}\|w\|_{\Oh},
\end{align*}
\begin{align*}
((\nabla\cdot\boldsymbol{\beta})\delta u,w)_{\mathcal{T}_{h}} \le &\; C \|\delta u\|_{\Oh}\|w\|_{\Oh},\\
\langle\boldsymbol{\beta}\cdot\boldsymbol{n}\widehat{\delta u}
, w
\rangle_{\partial\mathcal{T}_{h}}= &\; 
\langle(\boldsymbol{\beta}-\bld P_{0,h}\bld \beta)\cdot\boldsymbol{n}\widehat{\delta u}
, w
\rangle_{\partial\mathcal{T}_{h}}\nonumber\\
\le & C h \|\widehat{\delta u}\|_{\dOh}\|w\|_{\dOh}\nonumber\\
\le & C h^{1/2} \|\widehat{\delta u}\|_{\dOh}\|w\|_{\Oh}\\
%\label{trouble-term}
\langle\tau \delta u
, w - \mu
\rangle_{\partial\mathcal{T}_{h}^\star}
\le & \|\tau^{1/2} \delta u\|_{\partial\mathcal{T}_{h}^\star}\|\tau^{1/2}(w-\mu)\|_{\partial\mathcal{T}_{h}^\star}\nonumber\\
\le & C \|\tau^{1/2}\delta u\|_{\partial\mathcal{T}_{h}^\star}\left\||\tau-\frac{1}{2}\bld \beta\cdot\bld n|^{1/2} (w-\mu)\right\|_{\partial\mathcal{T}_{h}^\star}
\quad\qquad\text{ {by \eqref{assumption_tau_1}}}\nonumber\\
\le & C (\tau^w)^{1/2} \|\delta u\|_{\partial\mathcal{T}_{h}^\star}
\left\||\tau-\frac{1}{2}\bld \beta\cdot\bld n|^{1/2} (w-\mu)\right\|_{\partial\mathcal{T}_{h}^\star}.
%\\
%\le & C \|\delta u\|_{\partial\mathcal{T}_{h}^\star}\left\||\tau-\frac{1}{2}\bld \beta\cdot\bld n|^{1/2} (w-\mu)\right\|_{\partial\mathcal{T}_{h}^\star}
% \quad\quad\quad\qquad\text{ {by \eqref{assumption_tau_1}}}
\end{align*}
%\end{subequations}
Adding up these estimates all together, we obtain
\begin{align*}
& B((\eq,\eu,\euhat),(\boldsymbol{r}, w,\mu)) \\
\le &\;C (\|\epsilon^{-1/2}\bld{\delta q} \|_{\Oh}+ 
\left(\frac{\epsilon}{\tau^{\bld v}}\right)^{1/2}\|\epsilon^{-1/2}\bld{\delta q}\|_{\partial\mathcal{T}_{h}^{s}}+ \|\delta u\|_{\Oh} \nonumber\\
&\quad + h^{1/2} \|\widehat{\delta u}\|_{\dOh}+ (\tau^w)^{1/2}\|\delta u\|_{\partial\mathcal{T}_{h}^{\star}}
) \vertiii{(\bld r, w,\mu)}_e.\nonumber
\end{align*}
Note that $\euhat\in M_{h}(0)$ because $\euhat|_{\partial \Omega}=0$.
Using Lemma \ref{lemma_practical_infsup}, we immediately get 
\begin{align*}
 \vertiii{(\eq, \eu,\euhat)}_e \le & \;C (\|\epsilon^{-1/2}\bld{\delta q} \|_{\Oh}+ 
\left(\frac{\epsilon}{\tau^{\bld v}}\right)^{1/2}\|\epsilon^{-1/2}\bld{\delta q}\|_{\partial\mathcal{T}_{h}^{s}}+ \|\delta u\|_{\Oh} \nonumber\\
&\quad + h^{1/2} \|\widehat{\delta u}\|_{\dOh}+ (\tau^w)^{1/2}\|\delta u\|_{\partial\mathcal{T}_{h}^{\star}}
),
\end{align*}
The approximation properties of the projections gives the following estimates,
\begin{align*}
\|\epsilon^{-1/2}\bld{\delta q} \|_{\Oh} \le &\;C \epsilon^{-1/2} h^{s+1} |\bld q |_{H^{s+1}(\Oh;\mathbb{R}^d)}\\
\left (\frac{\epsilon}{\tau^{\bld v}}\right)^{1/2}\|\epsilon^{-1/2}\bld{\delta q}\|_{\partial\mathcal{T}_{h}^{s}} \le &\; 
C \left (\frac{h}{\tau^{\bld v}}\right)^{1/2}h^{s} |\bld q |_{H^{s+1}(\Oh;\mathbb{R}^d)} \\
 \|\delta u\|_{\Oh} \le &\; C h^{s+1} |u |_{H^{s+1}(\Oh)} \\
h^{1/2} \|\widehat{\delta u}\|_{\dOh}\le &\; C  h^{s+1} |u |_{H^{s+1}(\Oh)}\\
 (\tau^w)^{1/2}\|\delta u\|_{\partial\mathcal{T}_{h}^{\star}} \le &\; C(\tau^w)^{1/2} h^{s+1/2}|u|_{H^{s+1}(\Oh)},
\end{align*}
for all $s\in [0,k]$. 
%, choosing $s = k-1$ in the estimates involving $\bld q$, and $s = k$ in the estimates involving $u$,
Now, using assumption \eqref{assumption_tau_general} on $\tau$, we obtain the following estimates,
\begin{align*}
%\label{project-1}
\vertiii{(\eq, \eu,\euhat)}_e\le &\; C (\epsilon^{-1/2}h^{s_{\bld v}+1} + h^{s_{\bld v}+1/2}) |\bld q |_{H^{s_{\bld v}+1}(\Oh;\mathbb{R}^d)} + 
 C h^{s_{w}+1/2}|u |_{H^{s_{w}+1}(\Oh)},
\end{align*}
for all $s_{\bld v}\in [0,k]$ and $s_w \in [0,k]$.
Using the 
fact that $|\bld q |_{H^{k}(\Oh;\mathbb{R}^d)}= \epsilon |u|_{H^{k+1}(\Oh)}$, we get 
\begin{align}
\label{project-ee}
\vertiii{(\eq, \eu,\euhat)}_e\le &\; C (\epsilon^{1/2}h^{s_{\bld v}+1} + \epsilon h^{s_{\bld v}+1/2}) |u |_{H^{s_{\bld v}+2}(\Oh;\mathbb{R}^d)} + 
 C h^{s_{w}+1/2}|u |_{H^{s_{w}+1}(\Oh)},
\end{align}
for all $s_{\bld v}\in [0,k]$ and $s_w \in [0,k]$. 
Moreover, by approximation properties of the projection, we can easily get
\begin{align}
\label{approx-ee}
 \vertiii{(\bld{\delta q}, \delta u,\widehat{\delta u})}_e\le &\;  C \epsilon^{1/2}h^{s_{\bld v}+1}|u |_{H^{s_{\bld v}+2}(\Oh;\mathbb{R}^d)} + 
 C h^{s_{w}+1/2}|u |_{H^{s_{w}+1}(\Oh)},
\end{align}
for all $s_{\bld v}\in [0,k]$ and $s_w \in [0,k]$. 
Combining \eqref{project-ee}, \eqref{approx-ee} and using the triangle inequality, we  obtain 
\begin{align*}
%\label{est--1}
&  \vertiii{(\bld{q-q}_h, u-u_h,u-\widehat{u}_h)}_e \\
%\nonumber
 \le &\; C (\epsilon^{1/2}h^{s_{\bld v}+1} + \epsilon h^{s_{\bld v}+1/2})
 |u |_{H^{s_{\bld v}+2}(\Oh;\mathbb{R}^d)} + 
 C h^{s_{w}+1/2}|u |_{H^{s_{w}+1}(\Oh)},
\end{align*}
and 
\begin{align*}
%\label{est--2}
  \|u-{u}_h\|_{\Oh}\le &\; \vertiii{(\eq, \eu,\euhat)}_e+\|\delta u\|_{\Oh}\\
\le &\; C(\epsilon^{1/2}h^{s_{\bld v}+1} + \epsilon h^{s_{\bld v}+1/2})|u |_{H^{s_{\bld v}+2}(\Oh;\mathbb{R}^d)} + 
 C h^{s_{w}+1/2}|u |_{H^{s_{w}+1}(\Oh)}.\nonumber
\end{align*}
%Finally, for $k \ge 1$ and $u\in H^{k+1}(\Omega)$, we choose $s_{\bld v} = k-1$ and $s_w = k$, and 
%obtain 
%\begin{align}
%\label{project-1}
%\vertiii{(\bld{q-q}_h, u-u_h,u-\widehat{u}_h)}_e\le &\; C h^k(\epsilon h^{-1/2}+\epsilon^{1/2}+h^{1/2})|u |_{H^{k+1}(\Oh)},\\
% \|u-{u}_h\|_{\Oh} \le&\; C h^k(\epsilon h^{-1/2}+\epsilon^{1/2}+h^{1/2})|u |_{H^{k+1}(\Oh)}.
%\end{align}
This completes the proof of Theorem~\ref{MainTh1}.

For the proof of Theorem~\ref{MainTh2}, everything is exactly the same, except that 
\begin{align*}
(\tau^w)^{1/2}\|\delta u\|_{\partial\mathcal{T}_{h}^{\star}} \le &\; C h^{s+1}|u|_{H^{s+1}(\Oh)},
\end{align*}
because  assumption \eqref{assump_tau_s} ensures $\tau^w \leq \mathcal{O}(h)$.
 Hence we get 
\begin{align*}
 \vertiii{(\eq, \eu,\euhat)}_e\le &\; C (\epsilon^{1/2}h^{s_{\bld v}+1} + \epsilon h^{s_{\bld v}+1/2})|u |_{H^{s_{\bld v}+2}(\Oh;\mathbb{R}^d)} + 
 C h^{s_{w}+1}|u |_{H^{s_{w}+1}(\Oh)}\\
\|u-{u}_h\|_{\Oh}\le &\; \vertiii{(\eq, \eu,\euhat)}_e+\|\delta u\|_{\Oh}\\
\le &\; C (\epsilon^{1/2}h^{s_{\bld v}+1} + \epsilon h^{s_{\bld v}+1/2})|u |_{H^{s_{\bld v}+2}(\Oh;\mathbb{R}^d)} + 
 C h^{s_{w}+1}|u |_{H^{s_{w}+1}(\Oh)}.
\end{align*}
%So if $\epsilon = \mathcal{O}(h^2)$, $k\ge 1$ and $u\in H^{k+1}(\Omega)$, choosing $s_{\bld v}=k-1, s_w = k$, we get 
%\[
% \|u-u_h\|_{\Oh} = C h^{k+1}|u|_{H^{k+1}(\Oh)}.
%\]
This completes the proof of Theorem~\ref{MainTh2}.
%\end{proof}

%\begin{remark}
%If we replace $\boldsymbol{V}_{h}$ by $\{\boldsymbol{r}\in L^{2}(\Omega;\mathbb{R}^{d}):\boldsymbol{r}|_{K}\in P_{k}(K;\mathbb{R}^{d}) +\bld x P_k(K)
%\quad \forall K\in \mathcal{T}_{h}\}$, which is used in \cite{Egger2010}, 
%then we can use standard Raviart-Thomas projection $\boldsymbol{\Pi}_{h}^{RT}$ 
%in (\ref{projection_error2}). Thus, the term (\ref{difficult_term}) vanishes. Consequently, for $k \ge 1$ and $u\in H^{k+1}(\Omega)$, we have 
%the following estimate
%\begin{align}
%\label{project-rt}
%\vertiii{(\bld{q-q}_h, u-u_h,u-\widehat{u}_h)}_e\le &\; C h^k(\epsilon^{1/2}+h^{1/2})|u |_{H^{k+1}(\Oh)},\\
% \|u-{u}_h\|_{\Oh} \le&\; C h^k(\epsilon^{1/2}+h^{1/2})|u |_{H^{k+1}(\Oh)}.
%\end{align}
%\end{remark}

\section{Numerical results}\label{sec:num}

In this section, we present numerical studies using simple model problems in 
2D to verify our theoretical results and  display the performance of the HDG methods when the exact solution exibit layers. 
Our test problems are similar to those studied in \cite{AyusoMarini:cdf}.
We fix the domain to be the unit square in all the experiments, and run simulations of  the HDG methods \eqref{cd_hdg_eqs} with the following three choices of approximation spaces and 
stabilization functions:

\begin{itemize}
 \item [1.] The approximation spaces are  
%\begin{subequations}
\begin{align*}
\boldsymbol{V}_{h}&=\{\boldsymbol{r}\in L^{2}(\Omega;\mathbb{R}^{d}):\boldsymbol{r}|_{K}\in P_{k}(K;\mathbb{R}^{d})
\quad \forall K\in \mathcal{T}_{h}\},\\
W_{h}&=\{w\in L^{2}(\Omega):w|_{K}\in P_{k}(K)
\quad \forall K\in \mathcal{T}_{h}\},\\
M_{h}&=\{\mu\in L^{2}(\mathcal{E}_{h}):\mu |_{F} \in P_{k}(F)
\quad \forall F\in \mathcal{E}_{h}\},\\
\end{align*}
%\end{subequations}
while the stabilization function is given by
\begin{align*}
 \tau(F) = \max(\sup_{\bld x \in F} \,\bld \beta(\bld x)\cdot \bld n,0),\quad \forall F\in \dK, \forall K\in \Oh.
\end{align*}
We denote this choice as $P_k$-$HDG1$.
\item [2.] The approximation spaces are the same as in the previous case, while the stabilization function is given by 
\begin{align*}
\tau(F) = \max(\sup_{\bld x \in F} \,\bld \beta(\bld x)\cdot \bld n,0) + \min (0.1 \frac{\epsilon}{h_F}, 1 ),\quad \forall F\in \dK, \forall K\in \Oh.
\end{align*}
We denote this choice as $P_k$-$HDG2$.
\item [3.] The approximation spaces are  given as follows:
%\begin{subequations}
\begin{align*}
\boldsymbol{V}_{h}&=\{\boldsymbol{r}\in L^{2}(\Omega;\mathbb{R}^{d}):\boldsymbol{r}|_{K}\in P_{k}(K;\mathbb{R}^{d}) +\bld x P_k(K)
\quad \forall K\in \mathcal{T}_{h}\},\\
W_{h}&=\{w\in L^{2}(\Omega):w|_{K}\in P_{k}(K)
\quad \forall K\in \mathcal{T}_{h}\},\\
M_{h}&=\{\mu\in L^{2}(\mathcal{E}_{h}):\mu |_{F} \in P_{k}(F)
\quad \forall F\in \mathcal{E}_{h}\},\\
\end{align*}
%\end{subequations}
while the sabilization function is the same as $P_k$-$HDG1$.
We denote this method as $P_k$-$HDG3$.
\end{itemize}
We remark that the method $P_k$-$HDG3$ is exactly the MH--DG method considered in \cite{Egger2010} when $\bld \beta$ is piecewise-constant, 
which is proven in Appendix~\ref{appendix-1}.
%For a detailed comparison of the HDG method and the streamline diffusion method \cite{BrooksHughes} we refer to \cite[Section 5]{Egger2010} and 
%the reference therein. 

%We also remark that we present the history of convergence in the first and forth tests only in terms of $L^2$-- error on $u_h$ not only for the reason that it
%error comes natually from our analysis, see Theorem~\ref{MainTh1} and Theorem~\ref{MainTh2}, but also for the reason that it is a lot easier to implement than the stronger norm 
%$\vertiii{ \cdot}_e$ defined
%in \eqref{norm_ee}.

\subsection{A smooth solution test}
\label{smooth}
We take the velocity field $\boldsymbol{\beta} = [1, 2]^T$, 
and the diffusion coefficient $\epsilon$ as $1, 10^{-3}, 10^{-9}$. The source term $f$ is chosen so that
the exact solution is $ u(x,y) = \sin(2\pi\, x)\sin(2\pi\,y)$. 
We obtain the computational meshes by uniform refinement of a mesh that consists of 
a structured $5\times 5\times 2$ triangular elements, where the slanted edges are pointing in 
the northeast direction.

Let us remark that when $\epsilon = 1$ and $k\ge 1$, %where $k$ is the polynomial degree of the approximation spaces, 
we can follow \cite{ChenCockburnHDGI} to use superconvergence results to locally postprocess the solution 
to get a new approximation of the scalar variable $u_h^\star$, 
which converges faster than $u_h$. 
Here the definition of $u_h^\star\in P_{k+1}(K)$ for each element $K\in \Oh$ is as follows:
% \begin{subequations}
% \label{postprocessing}
 \begin{align*}
%  (c^{-1}\,\nabla u_h^\star,\nabla w)_K =&\; (f,w)_K-\bint{\bld q_h\cdot n}{w}{\dK}&&\quad \text{ for all } 
%  w\in \pol{k+1}{K}^0,\\
%\label{postu1}
  (\nabla u_h^\star,\nabla w)_K =&\; - (\epsilon^{-1}\,\bld q_h,\nabla w)_K&&\quad \text{ for all } 
  w\in P_{k+1}(K),\\
%\label{postu2}
 (u_h^\star,1)_K = &\; (u_h,1)_K.
 \end{align*}
%\end{subequations}
Table~\ref{smooth_test1} and Table~\ref{smooth_test_post} show the $L^2$--convergence 
results for $u_h$ and $u_h^\star$ for the three HDG methods when $\epsilon = 1$. 
For all the methods,  we observe convergence order of $k+1$ for $u_h$ and convergence order of $k+2$ for 
$u_h^\star$ (when $k\ge 1$). Also note that
the errors for the postprocessing $u_h^\star$  for the three methods are very close to each other.

When $\epsilon \ll 1 $, there is no superconvergence result for the HDG methods. Hence we only  show 
$\|u-u_h\|_\Oh$ in Table~\ref{smooth_test2} 
when $\epsilon = 10^{-3}$ and $\epsilon = 10^{-9}$. Again, optimal converge rates are recovered; better than the one predicted by the theoretical 
result in Theorem~\ref{MainTh1} which predict 
the loss of half order of accuracy. Moreover, our numerical results in Table \ref{smooth_test2} show
that the performance of  these three HDG methods are
equally good. 
Hence, we prefer to use $P_k$-$HDG1$ and $P_k$-$HDG2$ rather than $P_k$-$HDG3$ because $P_k$-$HDG3$ has more degrees of freedom for the local problems.

\begin{table}[htbp]
\footnotesize
  \begin{center}
\scalebox{0.9}{%
    $\begin{array}{|c|c||c c|c c|c c|}
    \hline
 \phantom{\big|} \mbox{degree} & \mbox{mesh} & \multicolumn{6}{|c|}{\epsilon=1\phantom{\big|}} 
\\
\cline{3-8}
 \phantom{\big|}  k& h^{-1} & \mbox{error}  & \mbox{order} & \mbox{error}  & \mbox{order} & \mbox{error}  & \mbox{order} \\
\hline
& & \multicolumn{2}{|c}{HDG1} &  \multicolumn{2}{|c}{HDG2}&
 \multicolumn{2}{|c|}{HDG3} 
\\
\hline
  &  5  &  1.74\mbox{e-}0  &  --     &  7.60\mbox{e-}1  &  --    &  2.06\mbox{e-}1  &  --    \\
  & 10  &  9.41\mbox{e-}1  &  0.88   &  3.33\mbox{e-}1  &  1.20  &  1.06\mbox{e-}1  &  0.97 \\
 0& 20  &  4.83\mbox{e-}1  &  0.96   &  1.72\mbox{e-}1  &  0.95  &  5.29\mbox{e-}2  &  1.00 \\
  & 40  &  2.44\mbox{e-}1  &  0.99   &  8.71\mbox{e-}2  &  0.98  &  2.64\mbox{e-}2  &  1.00 \\
\hline
  &  5  &  3.75\mbox{e-}1  &  --     &  1.72\mbox{e-}1  &  --    &  4.88\mbox{e-}2  &  --    \\
  & 10  &  1.01\mbox{e-}1  &  1.89   &  3.88\mbox{e-}2  &  2.15  &  1.26\mbox{e-}2  &  1.95 \\
 1& 20  &  2.59\mbox{e-}2  &  1.97   &  9.96\mbox{e-}3  &  1.96  &  3.18\mbox{e-}3  &  1.99  \\
  & 40  &  6.52\mbox{e-}3  &  1.99   &  2.51\mbox{e-}3  &  1.99  &  7.96\mbox{e-}4  &  2.00  \\
\hline
  &  5  &  6.19\mbox{e-}2  &  --     &  2.88\mbox{e-}2  &  --    &  8.60\mbox{e-}3  &  --   \\
  & 10  &  8.26\mbox{e-}3  &  2.90   &  3.20\mbox{e-}3  &  3.16  &  1.12\mbox{e-}3  &  2.95 \\
 2& 20  &  1.05\mbox{e-}3  &  2.97   &  4.09\mbox{e-}4  &  2.97  &  1.41\mbox{e-}4  &  2.99  \\
  & 40  &  1.33\mbox{e-}4  &  2.99   &  5.16\mbox{e-}5  &  2.99  &  1.77\mbox{e-}5  &  3.00  \\
\hline
  &  5  &  8.35\mbox{e-}3  &  --     &  3.90\mbox{e-}3  &  --    &  1.21\mbox{e-}3  &  --    \\
  & 10  &  5.53\mbox{e-}4  &  3.92   &  2.16\mbox{e-}4  &  4.18  &  7.81\mbox{e-}5  &  3.95 \\
 3& 20  &  3.52\mbox{e-}5  &  3.98   &  1.37\mbox{e-}5  &  3.97  &  4.92\mbox{e-}6  &  3.99  \\
  & 40  &  2.21\mbox{e-}6  &  3.99   &  8.64\mbox{e-}7  &  3.99  &  3.08\mbox{e-}7  &  4.00  \\
 \hline
 \end{array} $
}
\end{center}{$\phantom{|}$}
     \caption{History of convergence for $\|u - u_h\|_{L^2(\mathcal{T}_h)}$ when $\epsilon = 1$.}
   \label{smooth_test1}
\end{table}

\begin{table}[htbp]
\footnotesize
  \begin{center}
\scalebox{0.9}{%
    $\begin{array}{|c|c||c c|c c|c c|}
    \hline
 \phantom{\big|} \mbox{degree} & \mbox{mesh} & \multicolumn{6}{|c|}{\epsilon=1\phantom{\big|}} 
\\
\cline{3-8}
 \phantom{\big|}  k& h^{-1} & \mbox{error}  & \mbox{order} & \mbox{error}  & \mbox{order} & \mbox{error}  & \mbox{order} \\
\hline
& & \multicolumn{2}{|c}{HDG1} &  \multicolumn{2}{|c}{HDG2}&
 \multicolumn{2}{|c|}{HDG3} 
\\
\hline
  &  5  &  2.25\mbox{e-}2  &  --     &  1.70\mbox{e-}2  &  --    &  1.39\mbox{e-}2  &  --    \\
  & 10  &  3.08\mbox{e-}3  &  2.87   &  2.14\mbox{e-}3  &  2.99  &  1.70\mbox{e-}3  &  3.04 \\
 1& 20  &  3.94\mbox{e-}4  &  2.96   &  2.65\mbox{e-}4  &  3.02  &  2.08\mbox{e-}4  &  3.03  \\
  & 40  &  4.96\mbox{e-}5  &  2.99   &  3.28\mbox{e-}5  &  3.01  &  2.56\mbox{e-}5  &  3.02  \\
\hline
  &  5  &  2.49\mbox{e-}3  &  --     &  2.13\mbox{e-}3  &  --    &  1.92\mbox{e-}3  &  --   \\
  & 10  &  1.59\mbox{e-}4  &  3.97   &  1.35\mbox{e-}4  &  3.98  &  1.23\mbox{e-}4  &  3.97 \\
 2& 20  &  9.95\mbox{e-}6  &  4.00   &  8.45\mbox{e-}6  &  4.00  &  7.71\mbox{e-}6  &  3.99  \\
  & 40  &  6.22\mbox{e-}7  &  4.00   &  5.28\mbox{e-}7  &  4.00  &  4.82\mbox{e-}7  &  4.00  \\
\hline
  &  5  &  2.78\mbox{e-}4  &  --     &  2.43\mbox{e-}4  &  --    &  2.20\mbox{e-}4  &  --    \\
  & 10  &  8.87\mbox{e-}6  &  4.97   &  7.68\mbox{e-}6  &  4.98  &  6.94\mbox{e-}6  &  4.99 \\
 3& 20  &  2.78\mbox{e-}7  &  4.99   &  2.40\mbox{e-}7  &  5.00  &  2.17\mbox{e-}7  &  5.00  \\
  & 40  &  8.70\mbox{e-}9  &  5.00   &  7.50\mbox{e-}9  &  5.00  &  6.77\mbox{e-}9  &  5.00  \\
 \hline
 \end{array} $
}
\end{center}{$\phantom{|}$}
     \caption{History of convergence for $\|u - u_h^\star\|_{L^2(\mathcal{T}_h)}$ when $\epsilon = 1$.}
   \label{smooth_test_post}
\end{table}

\begin{table}[htbp]
\footnotesize
  \begin{center}
\scalebox{0.9}{%
    $\begin{array}{|c|c||c c|c c|c c||c c|c c|c c|}
    \hline
 \phantom{\big|} \mbox{degree} & \mbox{mesh} & \multicolumn{6}{|c||}{\epsilon=10^{-3}\phantom{\big|}} &\multicolumn{6}{|c|}{
\epsilon=10^{-9}\phantom{\big|}}
\\
\cline{3-14}
 \phantom{\big|}  k& h^{-1} & \mbox{error}  & \mbox{order} & \mbox{error}  & \mbox{order} & \mbox{error}  & \mbox{order} & \mbox{error}  & \mbox{order}
& \mbox{error}  & \mbox{order} & \mbox{error}  & \mbox{order} \\
\hline
& & \multicolumn{2}{|c|}{HDG1} &  \multicolumn{2}{|c}{HDG2}&
 \multicolumn{2}{|c||}{HDG3} &  \multicolumn{2}{|c}{HDG1}& \multicolumn{2}{|c}{HDG2} &  \multicolumn{2}{|c|}{HDG3}
\\
\hline
  &  5  &  3.16\mbox{e-}1  &  --     &  3.16\mbox{e-}1  &  --    &  3.14\mbox{e-}1  &  --   &  3.18\mbox{e-}1  &  --   &  3.18\mbox{e-}1  &  --   &  3.18\mbox{e-}1  &  --   \\
  & 10  &  1.71\mbox{e-}1  &  0.88   &  1.71\mbox{e-}1  &  0.88  &  1.69\mbox{e-}1  &  0.89 &  1.74\mbox{e-}1  &  0.87 &  1.74\mbox{e-}1  &  0.87 &  1.74\mbox{e-}1  &  0.87 \\
 0& 20  &  8.78\mbox{e-}2  &  0.96   &  8.78\mbox{e-}2  &  0.96  &  8.60\mbox{e-}2  &  0.98 &  9.06\mbox{e-}2  &  0.94 &  9.06\mbox{e-}2  &  0.94 &  9.06\mbox{e-}2  &  0.94 \\
  & 40  &  4.37\mbox{e-}2  &  1.00   &  4.38\mbox{e-}2  &  1.00  &  4.22\mbox{e-}2  &  1.03 &  4.63\mbox{e-}2  &  0.97 &  4.63\mbox{e-}2  &  0.97 &  4.63\mbox{e-}2  &  0.97 \\
\hline
  &  5  &  7.84\mbox{e-}2  &  --     &  7.84\mbox{e-}2  &  --  &  7.75\mbox{e-}2  &  --     &  7.96\mbox{e-}2  &  --   &  7.96\mbox{e-}2  &  --   &  7.96\mbox{e-}2  &  --   \\
  & 10  &  2.00\mbox{e-}2  &  1.97   &  2.00\mbox{e-}2  &  1.97 &  1.95\mbox{e-}2  &  1.99  &  2.04\mbox{e-}2  &  1.85 &  2.04\mbox{e-}2  &  1.85 &  2.04\mbox{e-}2  &  1.85 \\
 1& 20  &  4.95\mbox{e-}3  &  2.01   &  4.95\mbox{e-}3  &  2.01 &  4.73\mbox{e-}3  &  2.05  &  5.13\mbox{e-}3  &  1.96 &  5.13\mbox{e-}3  &  1.96 &  5.13\mbox{e-}3  &  1.96 \\
  & 40  &  1.21\mbox{e-}3  &  2.03   &  1.21\mbox{e-}3  &  2.03 &  1.11\mbox{e-}3  &  2.09  &  1.28\mbox{e-}3  &  1.99 &  1.28\mbox{e-}3  &  1.99 &  1.28\mbox{e-}3  &  1.99 \\
\hline
  &  5  &  1.32\mbox{e-}2  &  --     &  1.32\mbox{e-}2  &  --   &  1.31\mbox{e-}2  &  --   &  1.35\mbox{e-}2  &  --   &  1.35\mbox{e-}2  &  --  &  1.35\mbox{e-}2  &  --   \\
  & 10  &  1.72\mbox{e-}3  &  2.95   &  1.72\mbox{e-}3  &  2.95 &  1.68\mbox{e-}3  &  2.96 &  1.77\mbox{e-}3  &  2.93 &  1.77\mbox{e-}3  &  2.93 &  1.77\mbox{e-}3  &  2.93  \\
 2& 20  &  2.14\mbox{e-}4  &  3.00   &  2.14\mbox{e-}4  &  3.00 &  2.05\mbox{e-}4  &  3.03 &  2.24\mbox{e-}4  &  2.98 &  2.24\mbox{e-}4  &  2.98 &  2.24\mbox{e-}4  &  2.98  \\
  & 40  &  2.63\mbox{e-}5  &  3.02   &  2.63\mbox{e-}5  &  3.02 &  2.45\mbox{e-}5  &  3.07 &  2.80\mbox{e-}5  &  3.00 &  2.80\mbox{e-}5  &  3.00&  2.80\mbox{e-}5  &  3.00\\
\hline
  &  5  &  1.83\mbox{e-}3  &  --     &  1.83\mbox{e-}3  &  --   &  1.80\mbox{e-}3  &  --    &  1.87\mbox{e-}3  &  --   &  1.87\mbox{e-}3  &  --  &  1.87\mbox{e-}3  &  --   \\
  & 10  &  1.17\mbox{e-}4  &  3.97   &  1.17\mbox{e-}4  &  3.97 &  1.13\mbox{e-}4  &  3.99  &  1.20\mbox{e-}4  &  3.95 &  1.20\mbox{e-}4  &  3.95 &  1.20\mbox{e-}4  &  3.95  \\
 3& 20  &  7.23\mbox{e-}6  &  4.01   &  7.23\mbox{e-}6  &  4.01 &  6.82\mbox{e-}6  &  4.05  &  7.56\mbox{e-}6  &  3.99 &  7.56\mbox{e-}6  &  3.99&  7.56\mbox{e-}6  &  3.99\\
  & 40  &  4.43\mbox{e-}7  &  4.03   &  4.43\mbox{e-}7  &  4.03 &  4.01\mbox{e-}7  &  4.09  &  4.73\mbox{e-}7  &  4.00 &  4.73\mbox{e-}7  &  4.00 &  4.73\mbox{e-}7  &  4.00\\
\hline
 \end{array} $
}
\end{center}{$\phantom{|}$}
     \caption{History of convergence for $\|u - u_h\|_{L^2(\mathcal{T}_h)}$ when $\epsilon = 10^{-3}$ and $\epsilon = 10^{-9}$.}
   \label{smooth_test2}
\end{table}

\subsection{A rotating flow test}
We take $\epsilon = 10^{-6}$, $\boldsymbol{\beta} = [y - 1/2, 1/2 - x]^T$, and  $f = 0$. The solution $u$ is prescribed along the slip $1/2\times [0, 1/2]$, as follows:
\[
 u(1/2,y) = \sin^2(2\pi\,y)\qquad y \in [0, 1/2]. 
\]
See \cite{HughesScovazziBochevBuffa2006} for a detailed description of this test.

In Fig.~\ref{rotating1}, we plot $u_h$ obtained from the three HDG methods for various polynomial degrees in a structured triangular grid of 128 elements. 
To better compare the results, we plot in Fig.~\ref{rotating3} extracted data of $u_h$ 
along the horizontal center line $y = 1/2$. We also plot in Fig.~\ref{rotating_high} a comparison of $P_0$-$HDG1$ in 8192 elements and  $P_3$-$HDG1$ in 128 elements.
From Fig.~\ref{rotating1}, we find that all the HDG methods produce similar results. Moreover, it is clear that higher order methods
lead to better approximation results and 
are computationally cheaper than lower order methods for qualitively similar numerical results. 

\begin{figure}[htbp]
\centering
\includegraphics[width = 0.30\textwidth]{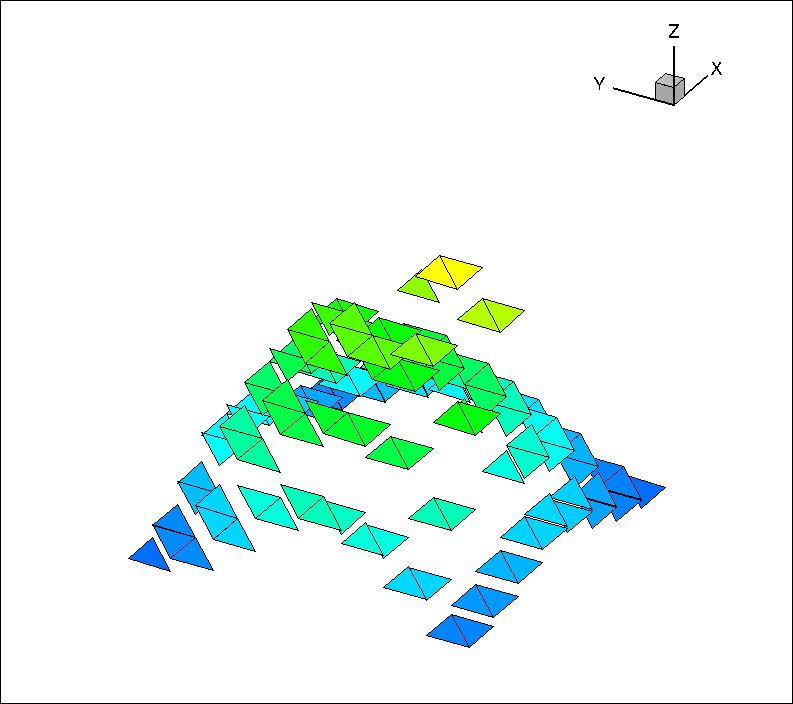}
\includegraphics[width = 0.30\textwidth]{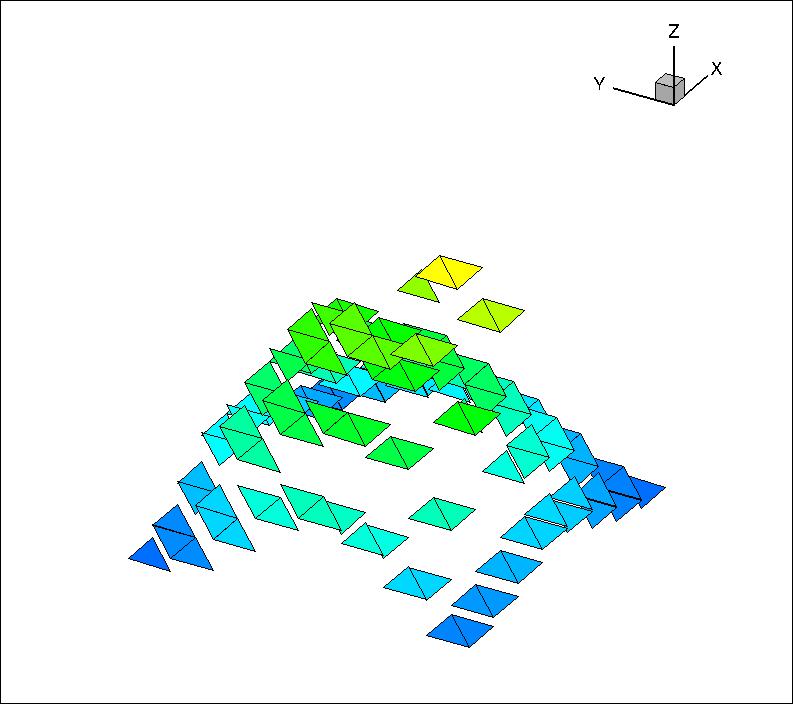}
\includegraphics[width = 0.30\textwidth]{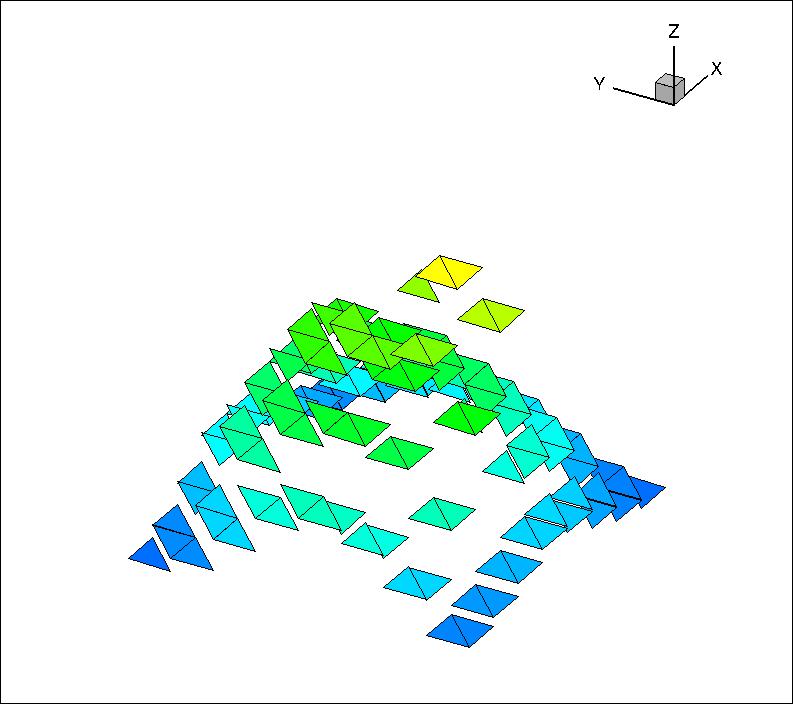}
\includegraphics[width = 0.30\textwidth]{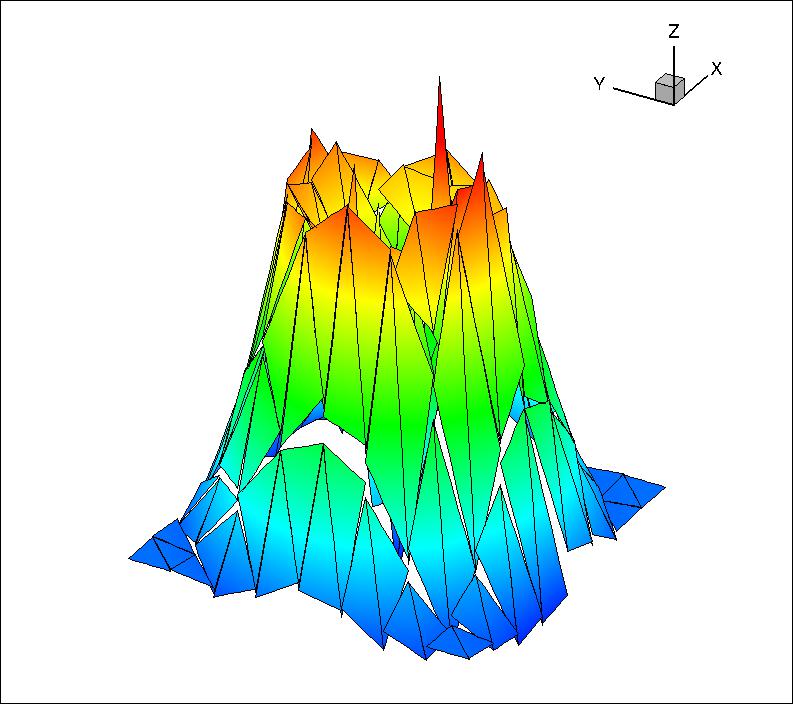}
\includegraphics[width = 0.30\textwidth]{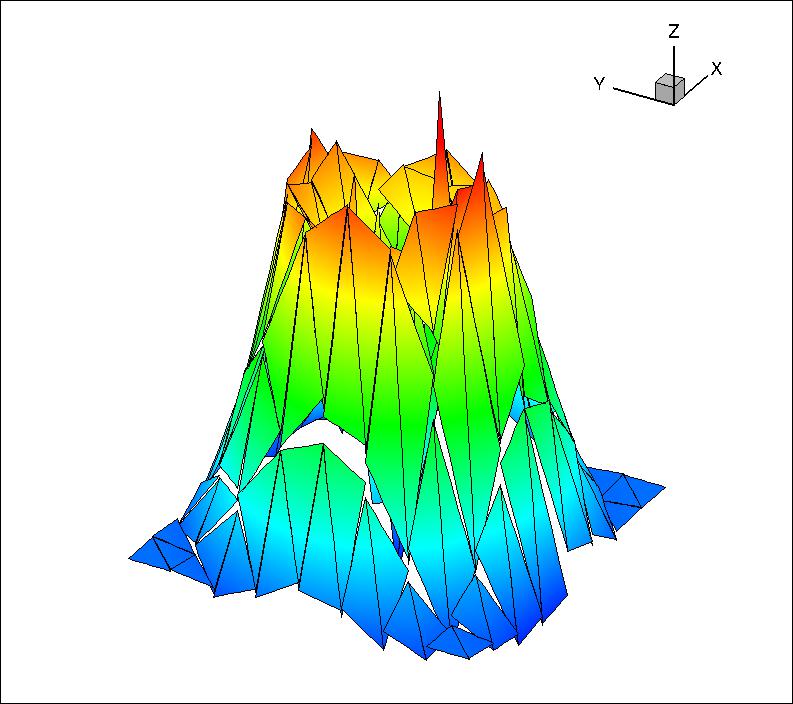}
\includegraphics[width = 0.30\textwidth]{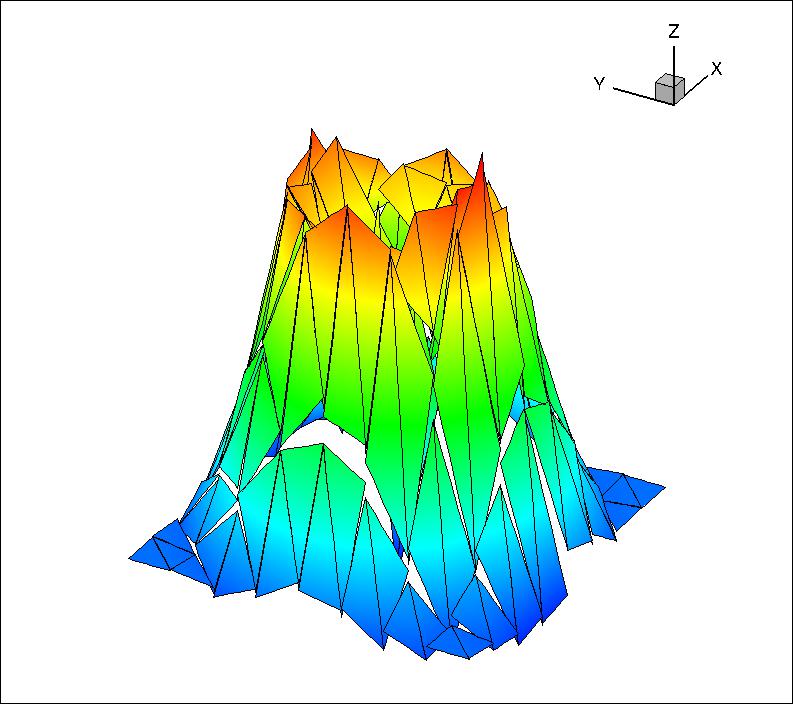}
\includegraphics[width = 0.30\textwidth]{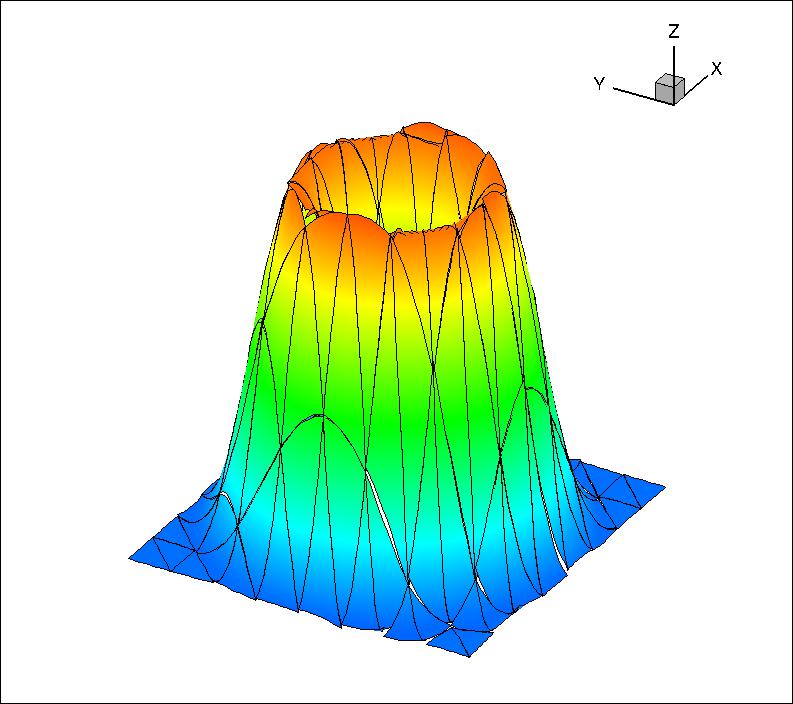}
\includegraphics[width = 0.30\textwidth]{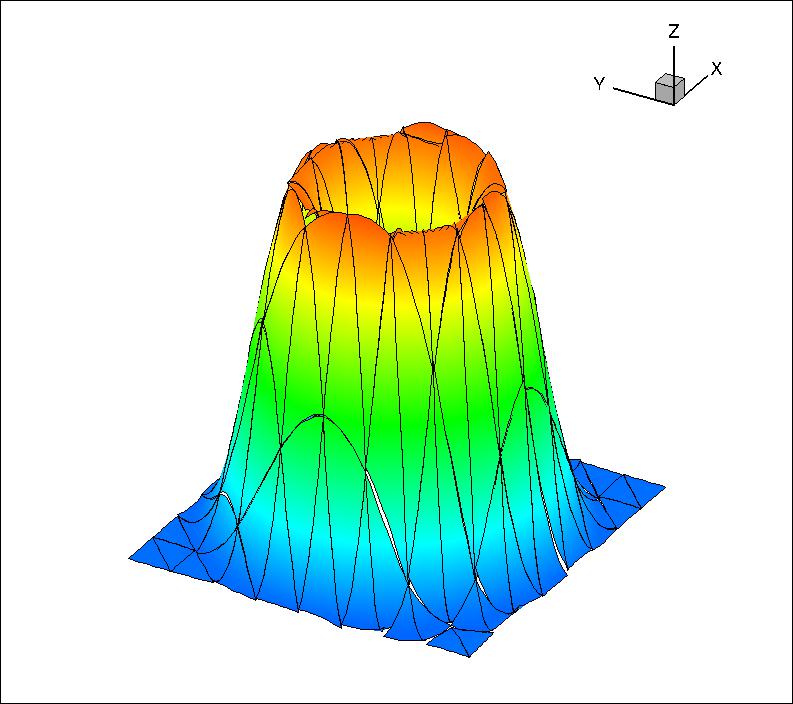}
\includegraphics[width = 0.30\textwidth]{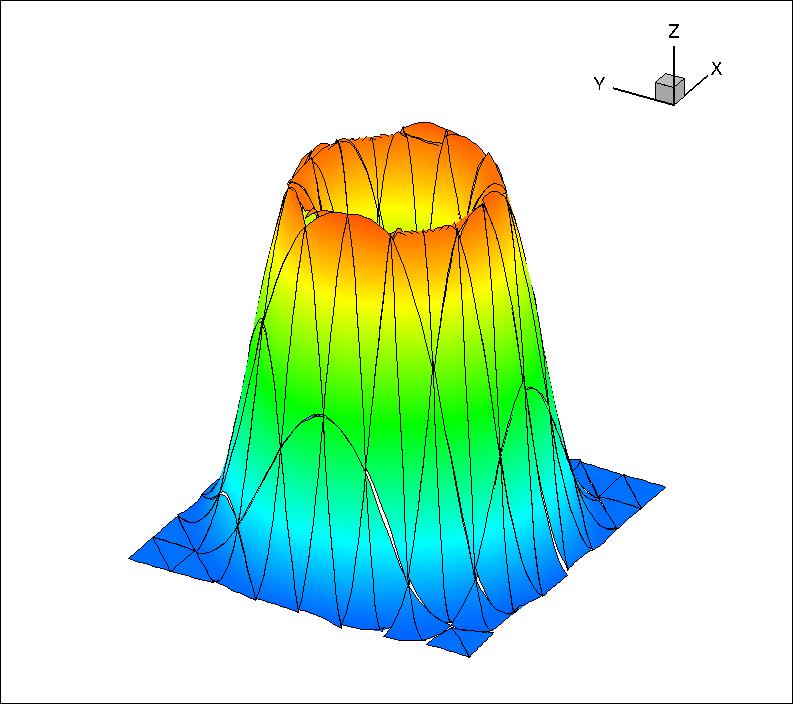}
\includegraphics[width = 0.30\textwidth]{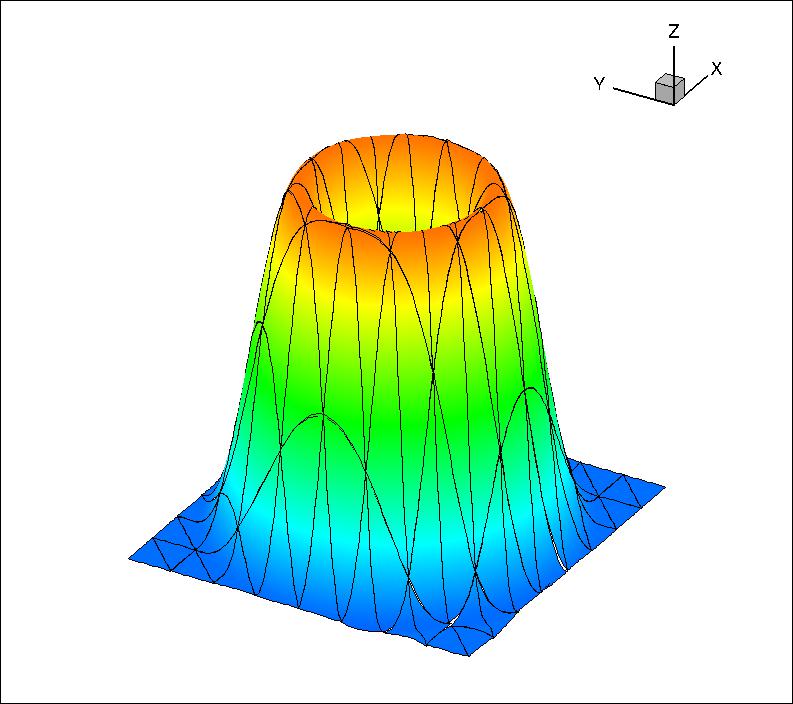}
\includegraphics[width = 0.30\textwidth]{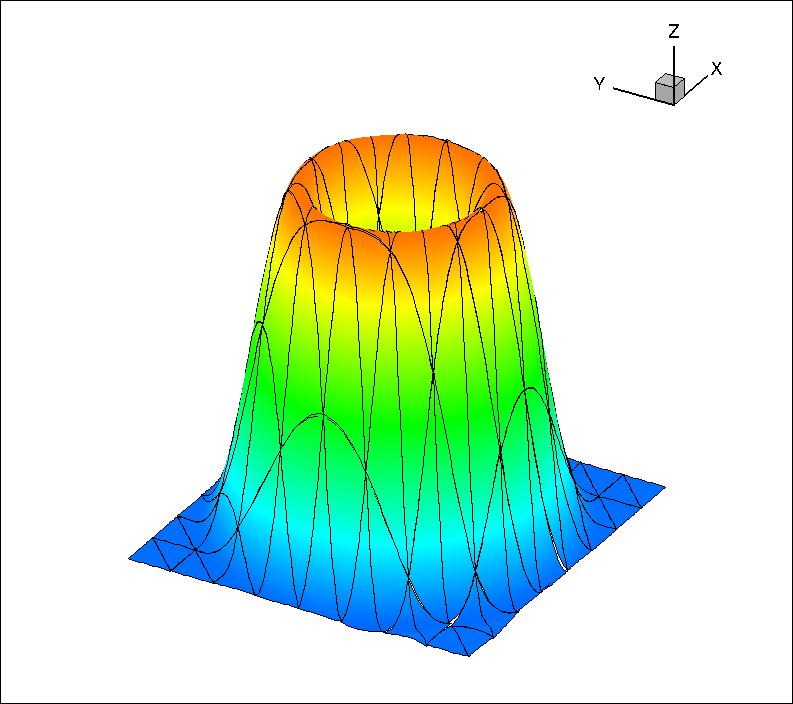}
\includegraphics[width = 0.30\textwidth]{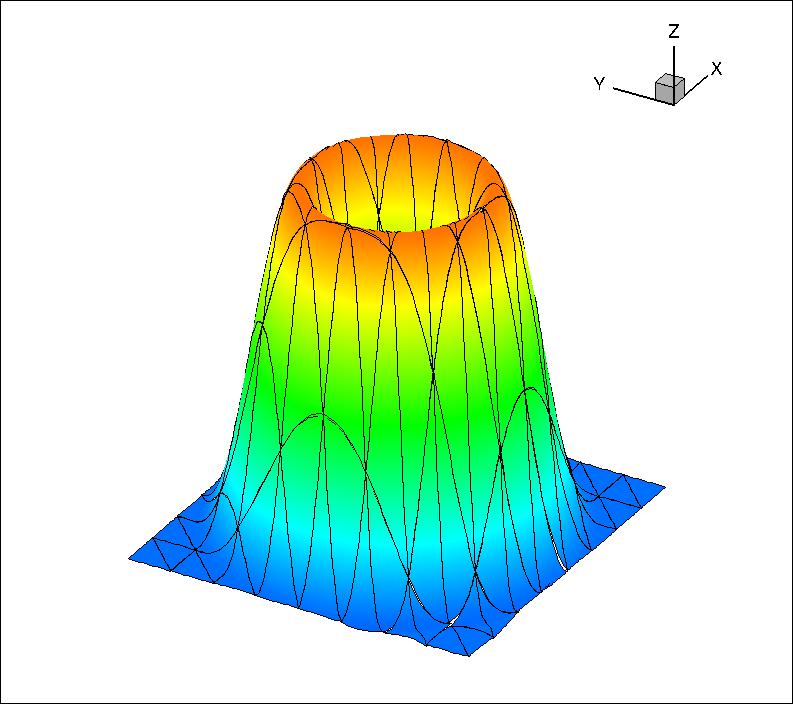}
\caption{3D plot of $u_h$ for rotating flow test with $\epsilon = 10^{-6}$ in 128 elements. Left--Right: $HDG1$, $HDG2$, $HDG3$.
Top--Bottom: $P_0$--$P_3$.}
\label{rotating1} 
\end{figure}

\begin{figure}[htbp]
\centering
\includegraphics[width = 0.75\textwidth]{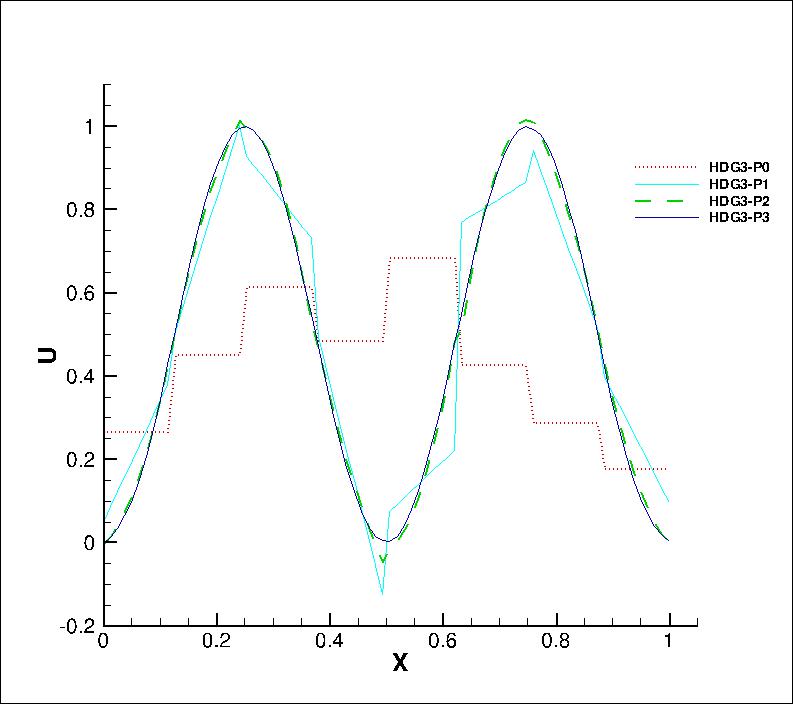}
\caption{$u_h$ along $y = 1/2$ for $P_k$-$HDG3$ with 128 elements.}
\label{rotating3} 
\end{figure}

\begin{figure}[htbp]
\centering
\includegraphics[width = 0.45\textwidth]{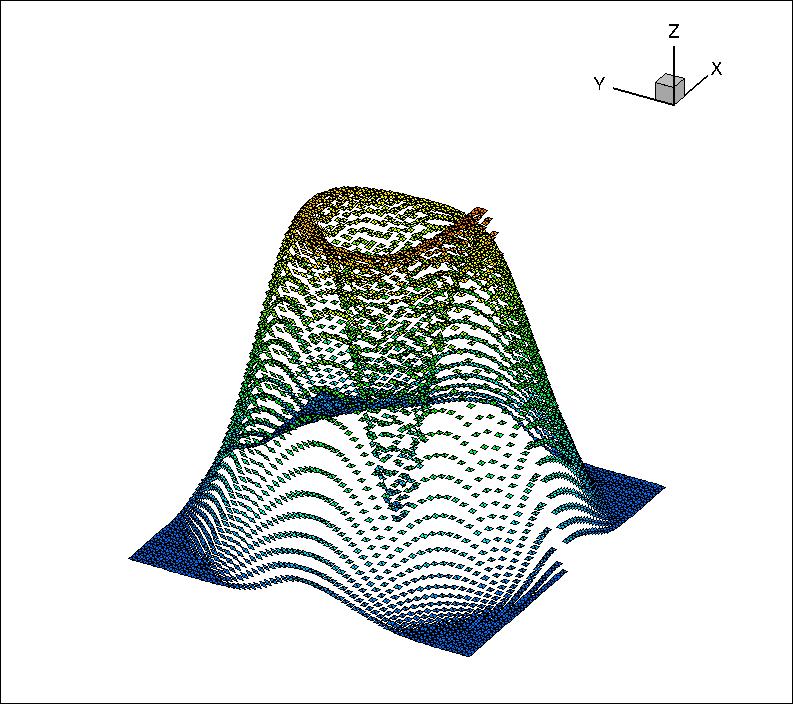}
\includegraphics[width = 0.45\textwidth]{./pic/rr/HDG1_P3.jpg}
\includegraphics[width = 0.45\textwidth]{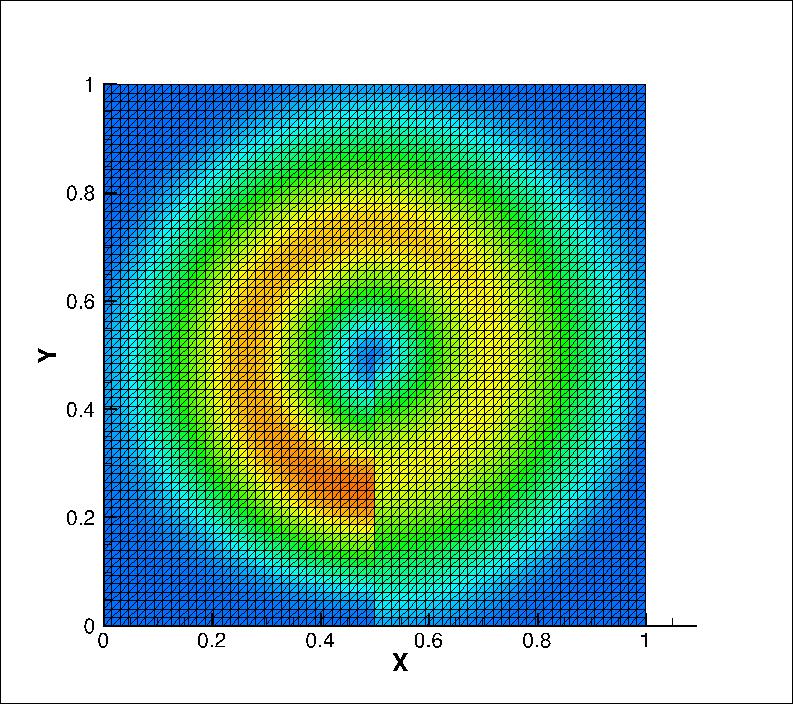}
\includegraphics[width = 0.45\textwidth]{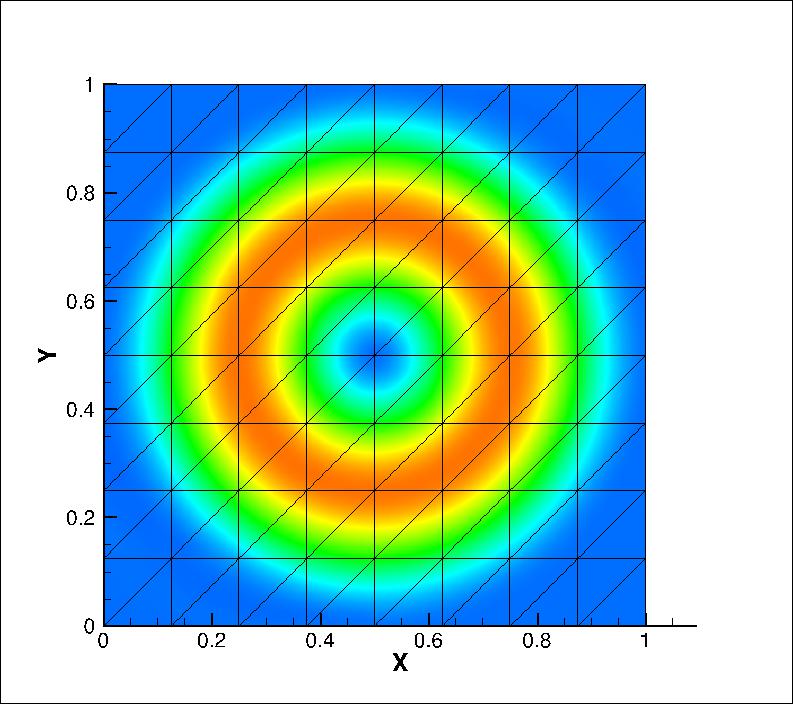}
\caption{A comparison of $P_0$-$HDG$ and $P_3$-$HDG$. Left: $P_0$-$HDG1$ with $8192$ elements; Right $P_3$-$HDG1$
with $128$ elements. Top: 3D plot; Bottom: 2D contour.}
\label{rotating_high} 
\end{figure}

\subsection{An interior layer test}
We take $\boldsymbol{\beta} = [1/2, \sqrt{3}/2]^T$, $f = 0$, and  the Dirichlet boundary condition as follows:
\begin{align*}
 u = \left\{\begin{tabular}{l l}
            $1$& { on }$\{y = 0, 0\le x\le 1\}$,\\
            $1$& { on }$\{x = 0, 0\le y\le 1/5\}$,\\
            $0$& { elsewhere. }
           \end{tabular}
\right.
\end{align*}
It is clear that for $\epsilon$ small, the exact solution produces an interior layer along $\bld \beta$ direction starting from $(0,1/5)$, and boundary 
layers on the right and top right boundary.

In Fig.~\ref{interior31} and Fig.~\ref{interior91},
we plot the computational results in a structured triangular grid of 128 elements for $\epsilon = 10^{-3}$ and $\epsilon = 10^{-9}$
respectively. 
 In order to better see the performance of the HDG method in capturing interior layers, 
in Fig.~\ref{interior_contour},
we plot the contour of $u_h$ using $P_k$-$HDG1$ with $ 0\le k \le 3$ for $\epsilon = 10^{-3}$ in three consecutive meshes with the coarsest one consists of 200 elements. 
Again,  all the HDG methods produce quite similar results. 
Note that, as expected, the piecewise-constant approximations are free of oscillations but extensively smear out the interior layer, while, on the 
other hand, higher order approximations capture the interior layer within a few elements but produce oscillations within the layer. 

%We have a huge improvement in the approximation of the layers
%from $P_0$ to $P_1$, while from $P_1$ to $P_3$, the numerical results look quite similar.

\begin{figure}[htbp]
\centering
\includegraphics[width = 0.3\textwidth]{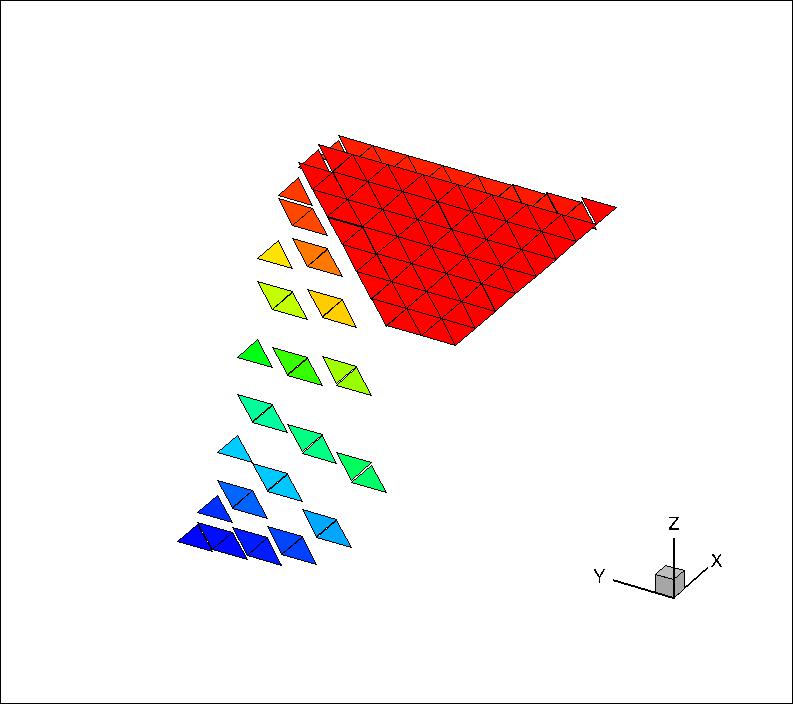}
\includegraphics[width = 0.3\textwidth]{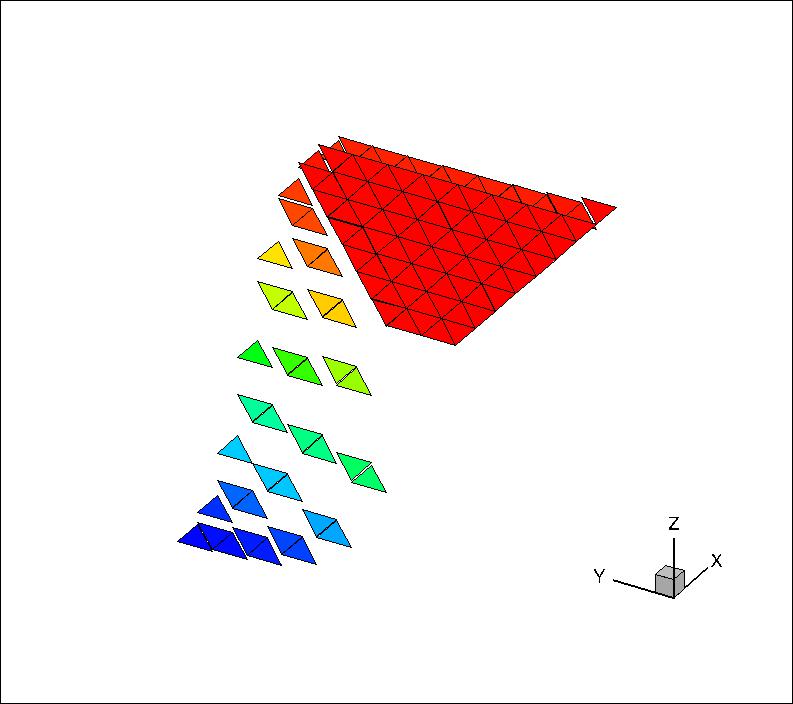}
\includegraphics[width = 0.3\textwidth]{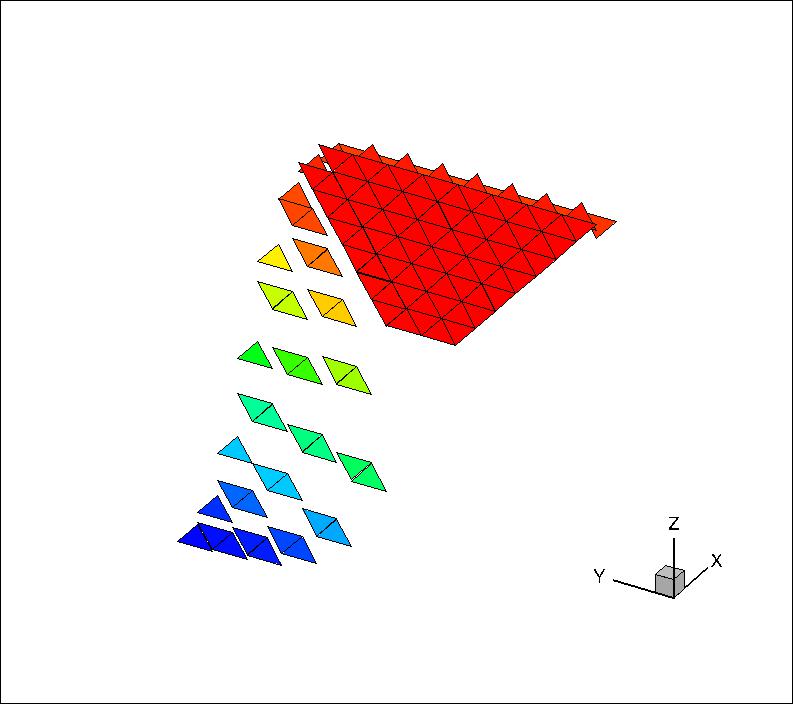}
\includegraphics[width = 0.3\textwidth]{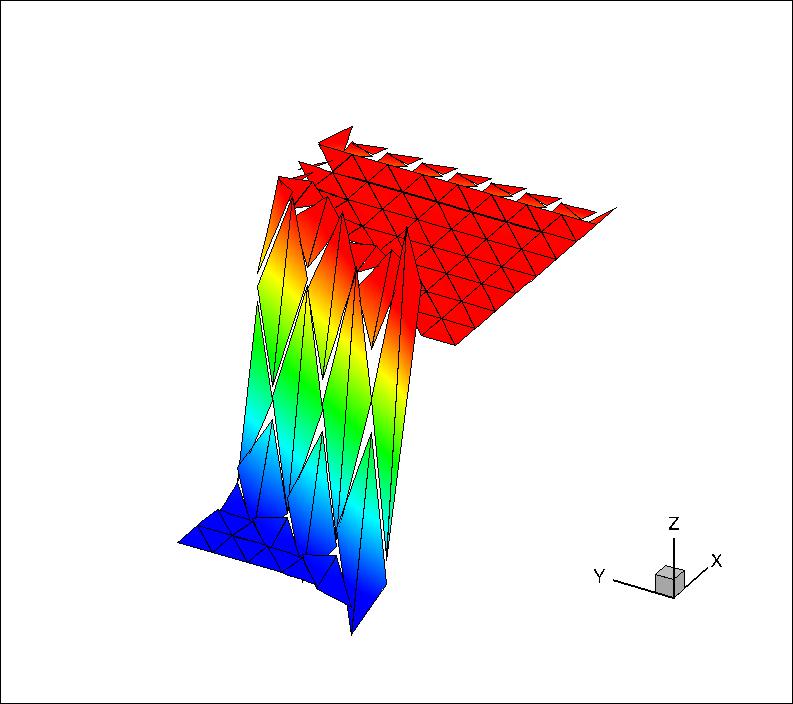}
\includegraphics[width = 0.3\textwidth]{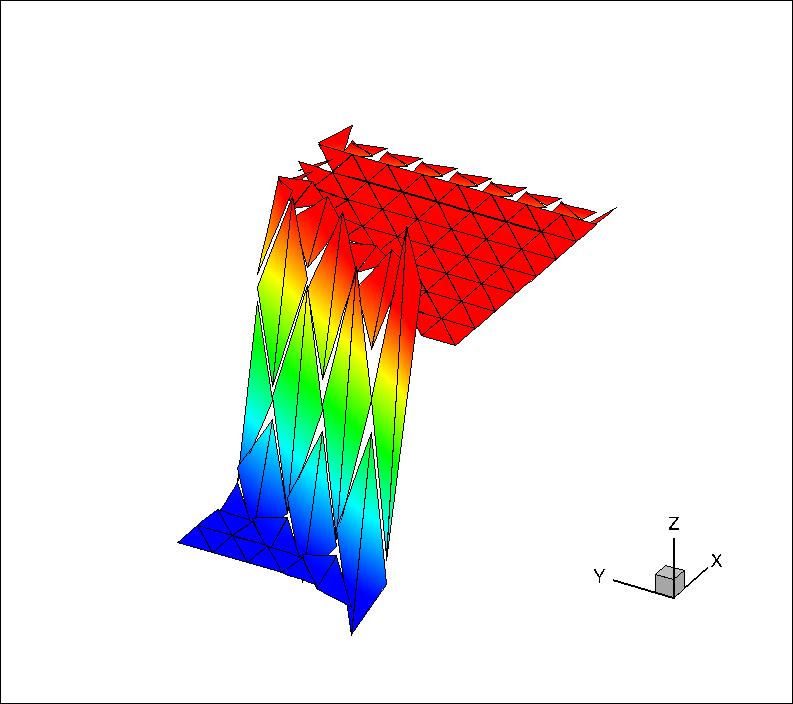}
\includegraphics[width = 0.3\textwidth]{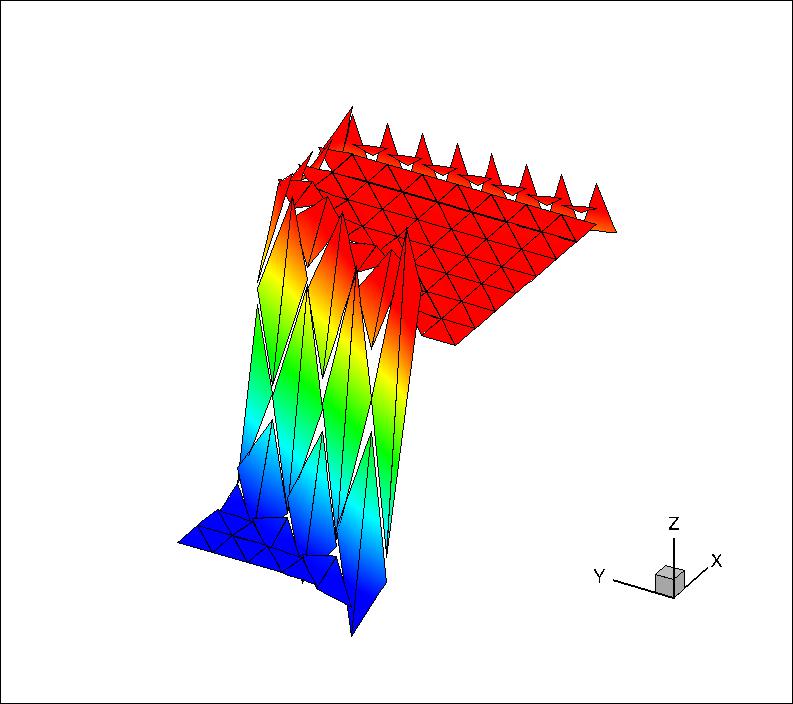}
\includegraphics[width = 0.3\textwidth]{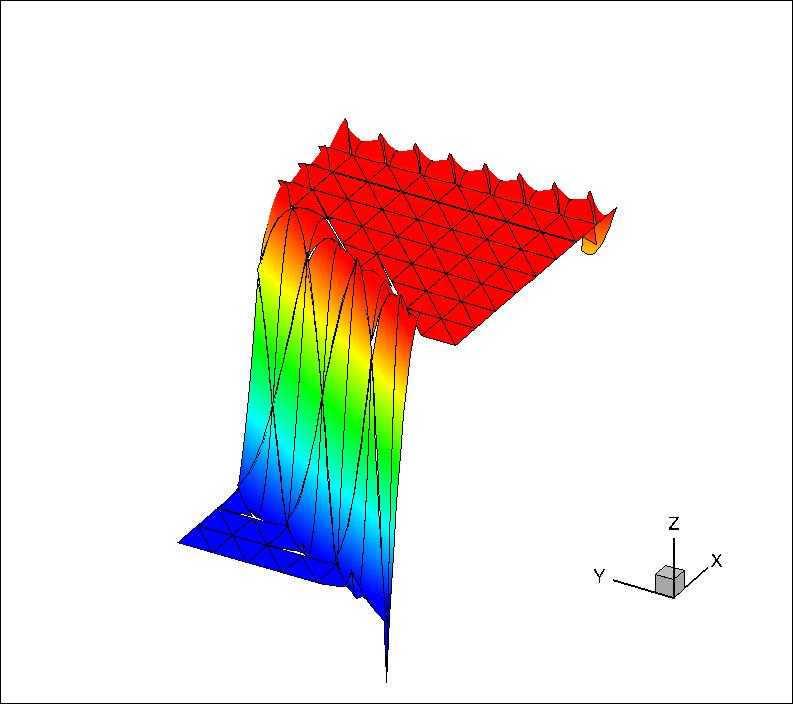}
\includegraphics[width = 0.3\textwidth]{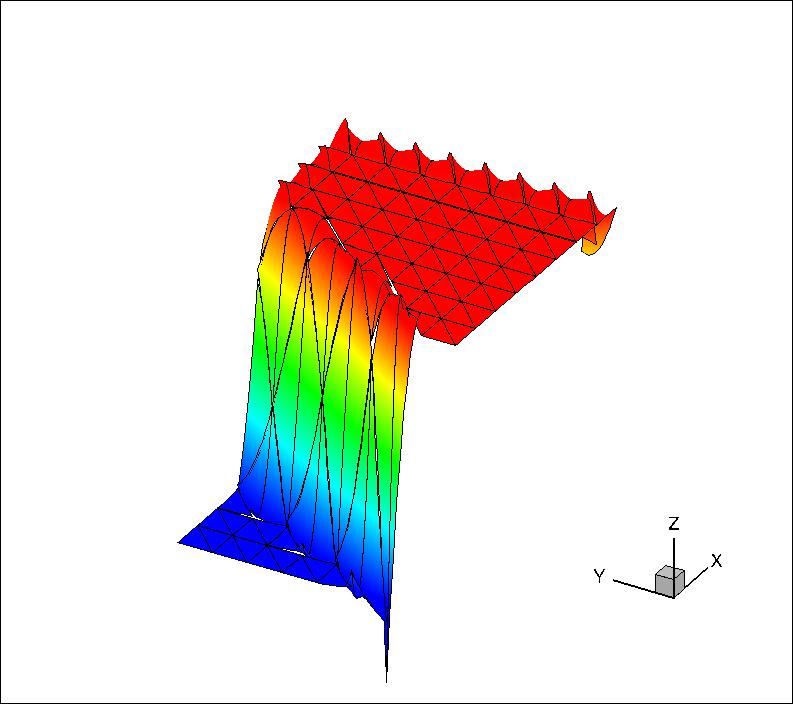}
\includegraphics[width = 0.3\textwidth]{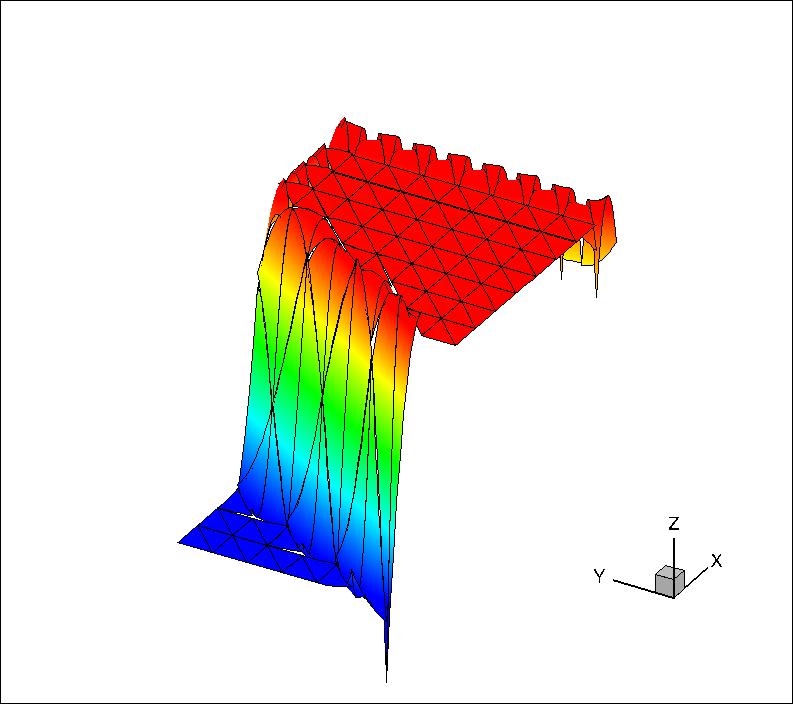}
\includegraphics[width = 0.3\textwidth]{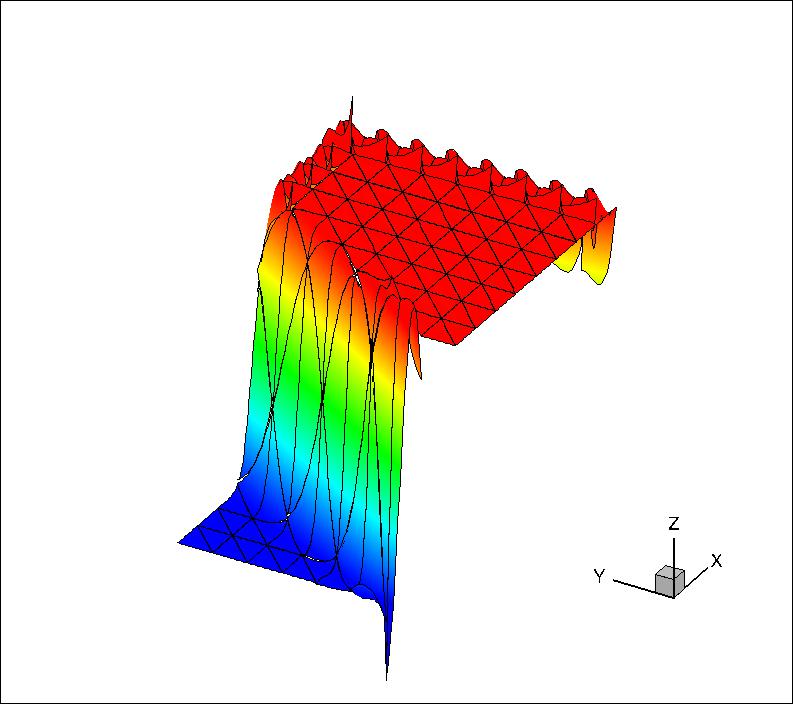}
\includegraphics[width = 0.3\textwidth]{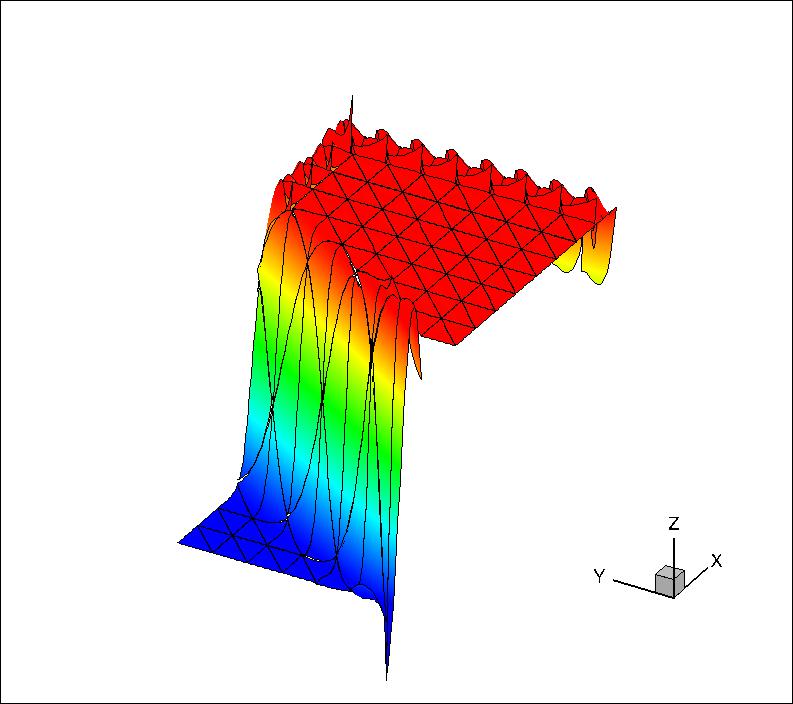}
\includegraphics[width = 0.3\textwidth]{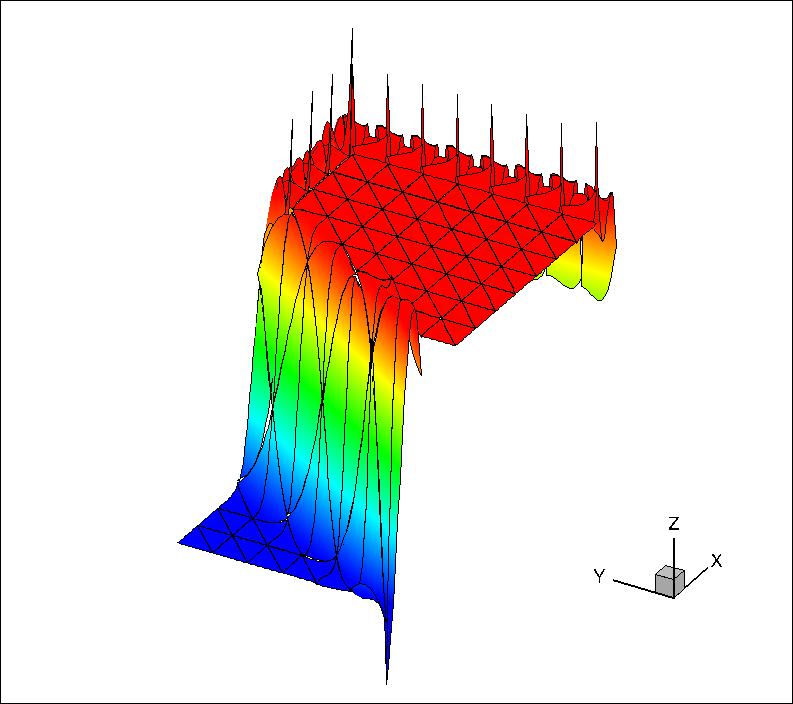}
\caption{3D plot of $u_h$ for the interior layer test with $\epsilon = 10^{-3}$ in 128 elements. Left--Right: $HDG1$, $HDG2$, $HDG3$.
Top--Bottom: $P_0$--$P_3$.}
\label{interior31} 
\end{figure}
\begin{figure}[htbp]
\centering
\includegraphics[width = 0.3\textwidth]{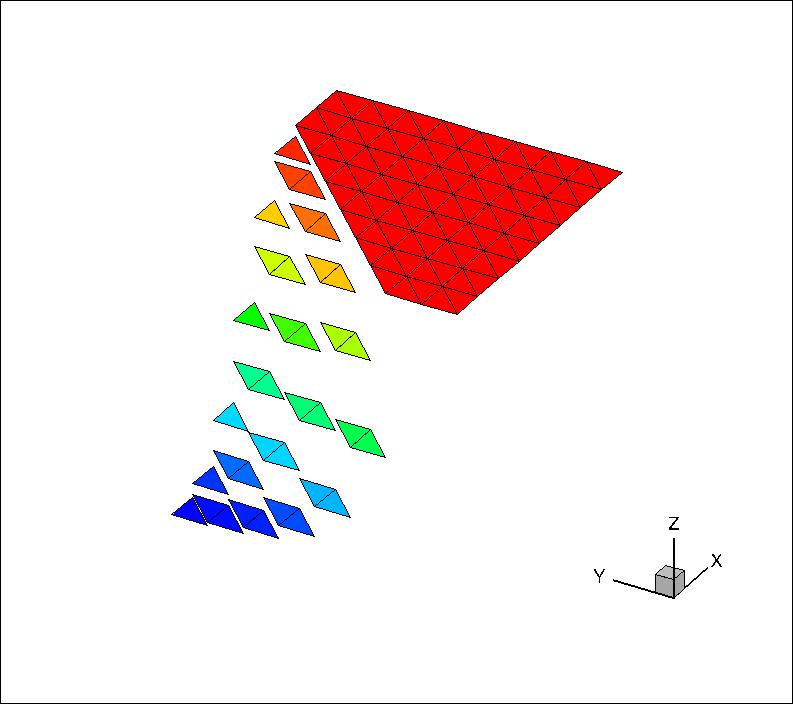}
\includegraphics[width = 0.3\textwidth]{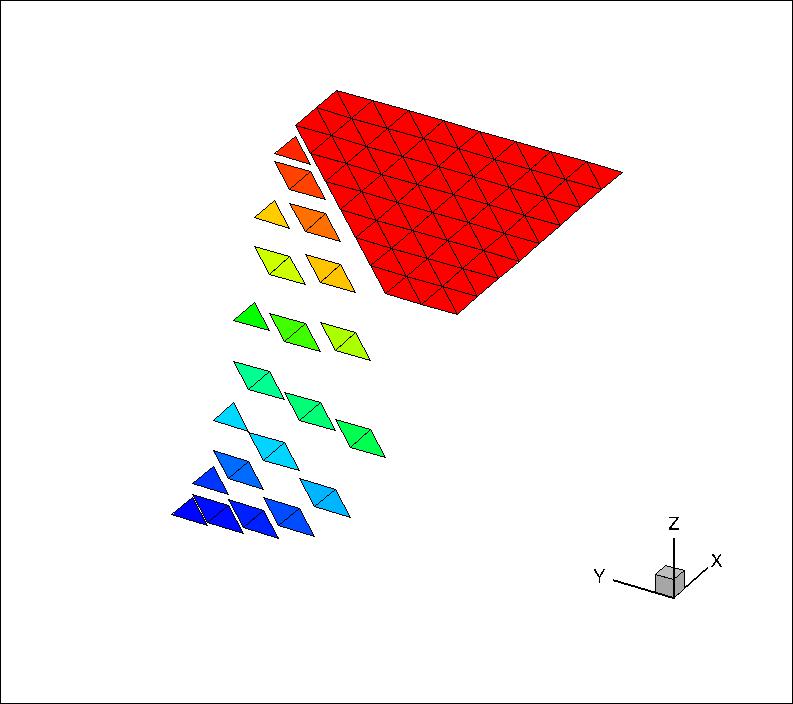}
\includegraphics[width = 0.3\textwidth]{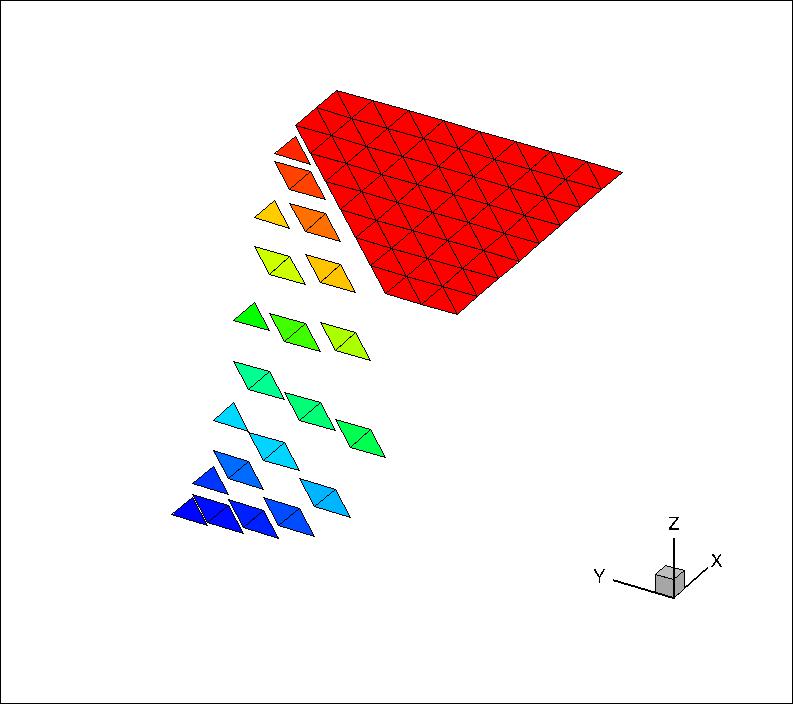}
\includegraphics[width = 0.3\textwidth]{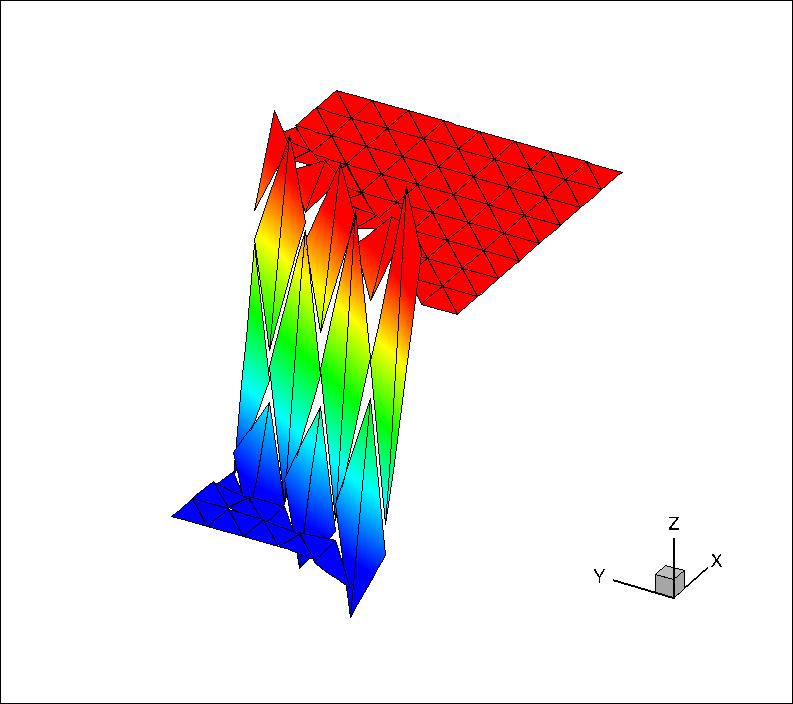}
\includegraphics[width = 0.3\textwidth]{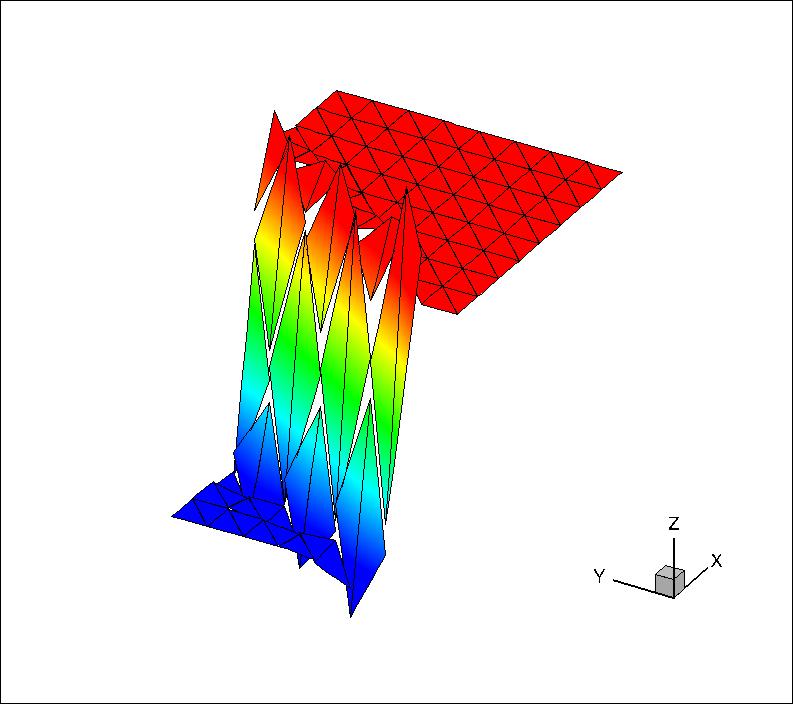}
\includegraphics[width = 0.3\textwidth]{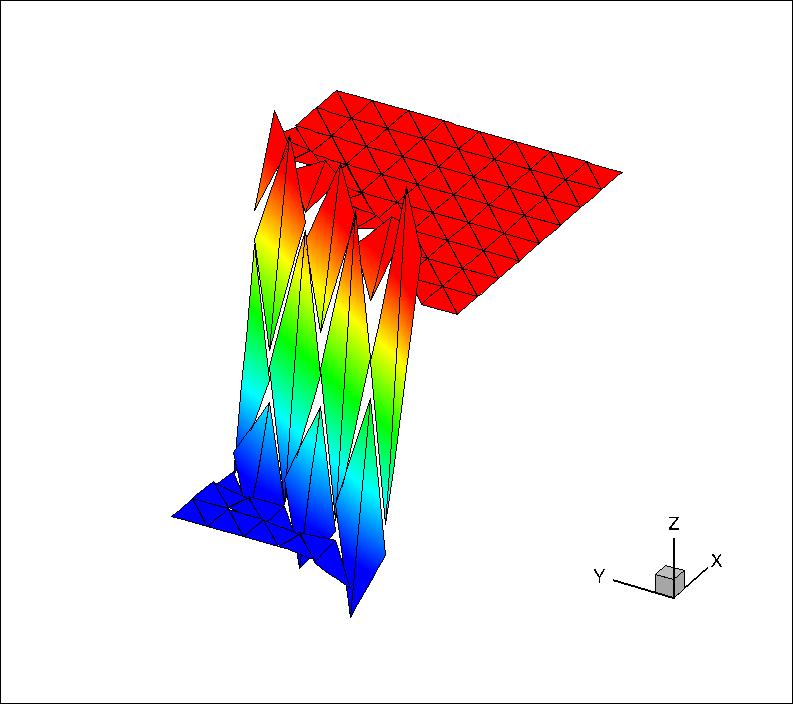}
\includegraphics[width = 0.3\textwidth]{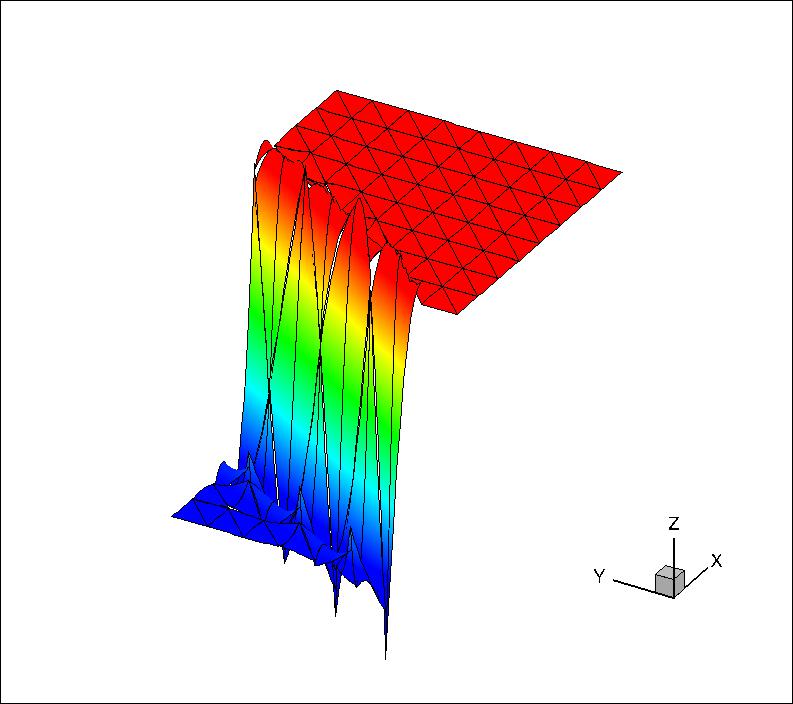}
\includegraphics[width = 0.3\textwidth]{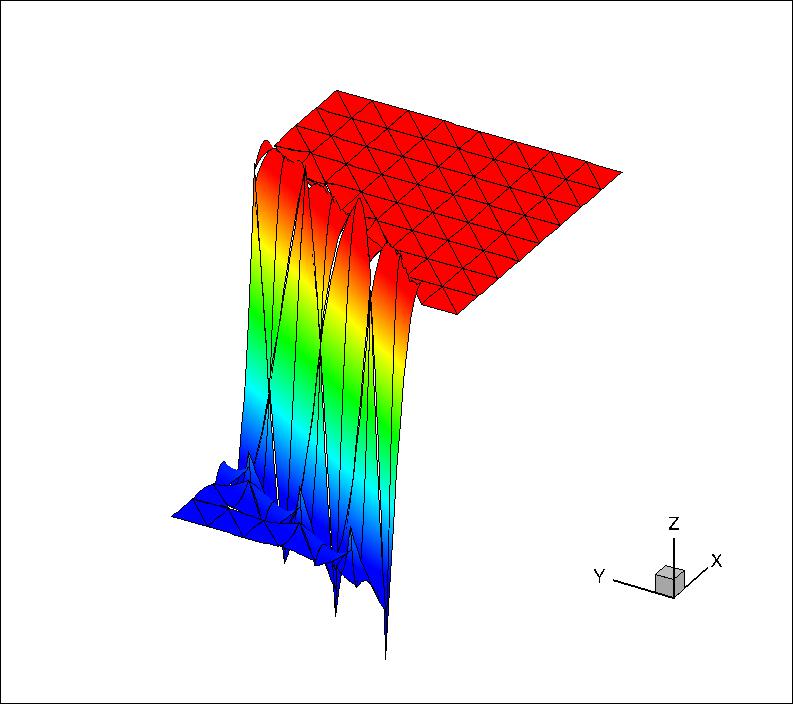}
\includegraphics[width = 0.3\textwidth]{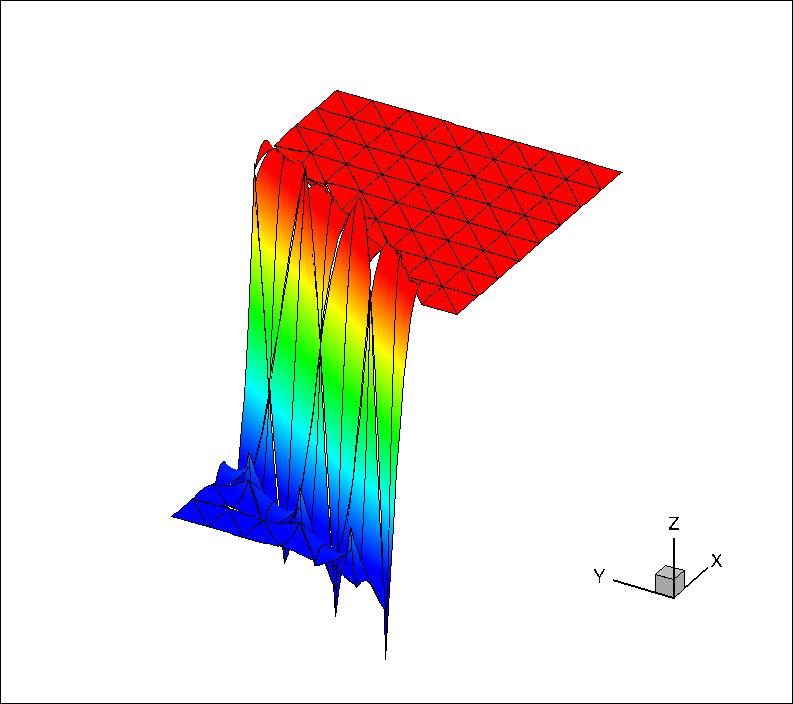}
\includegraphics[width = 0.3\textwidth]{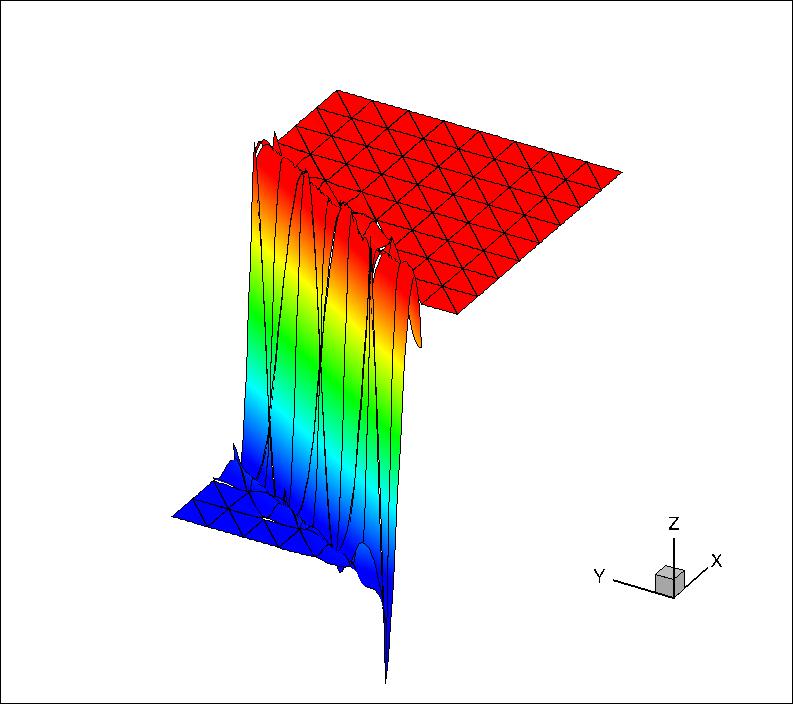}
\includegraphics[width = 0.3\textwidth]{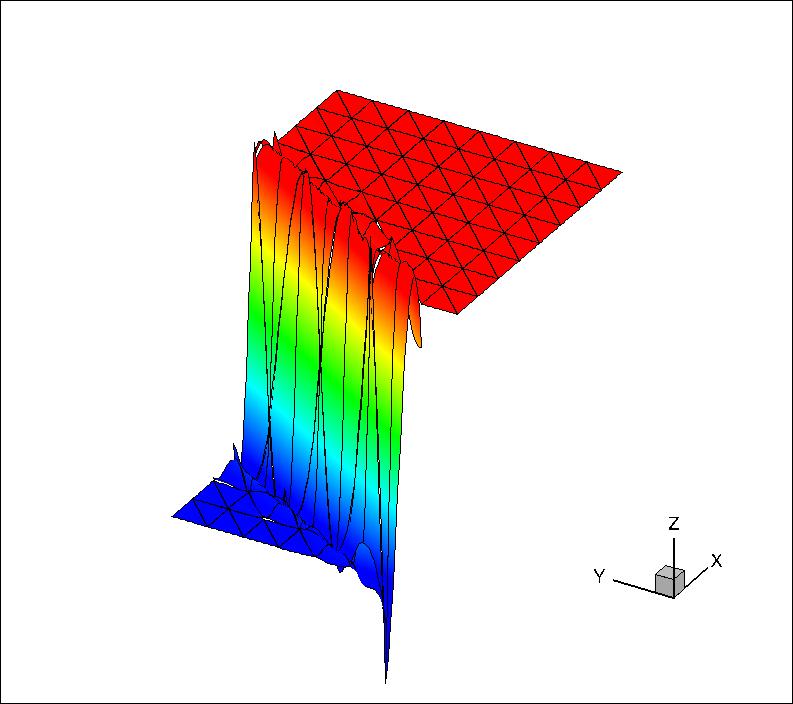}
\includegraphics[width = 0.3\textwidth]{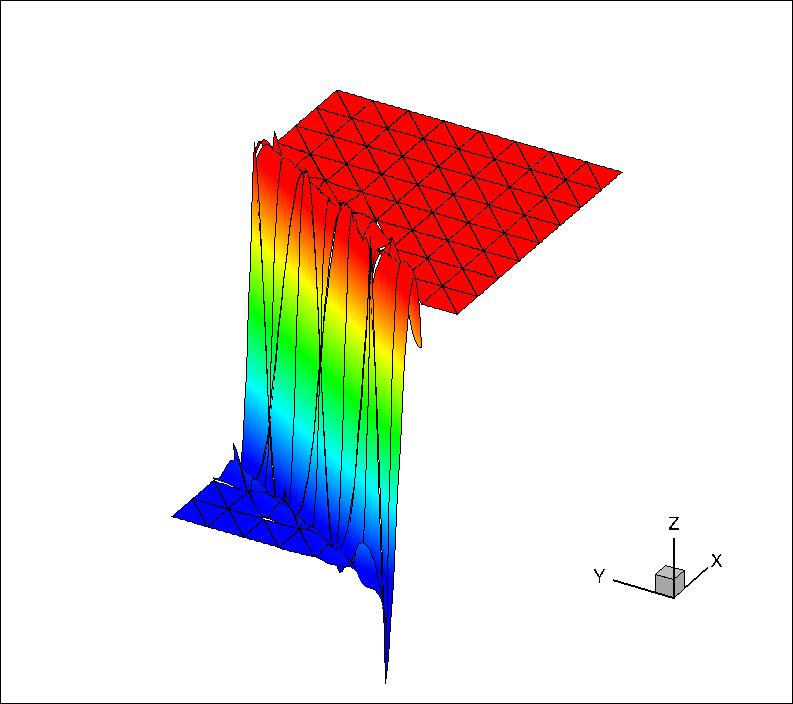}
\caption{3D plot of $u_h$ for the interior layer test with $\epsilon = 10^{-9}$ in 128 elements. Left--Right: $HDG1$, $HDG2$, $HDG3$.
Top--Bottom: $P_0$--$P_3$.}
\label{interior91} 
\end{figure}

\begin{figure}[htbp]
\centering
\includegraphics[width = 0.3\textwidth]{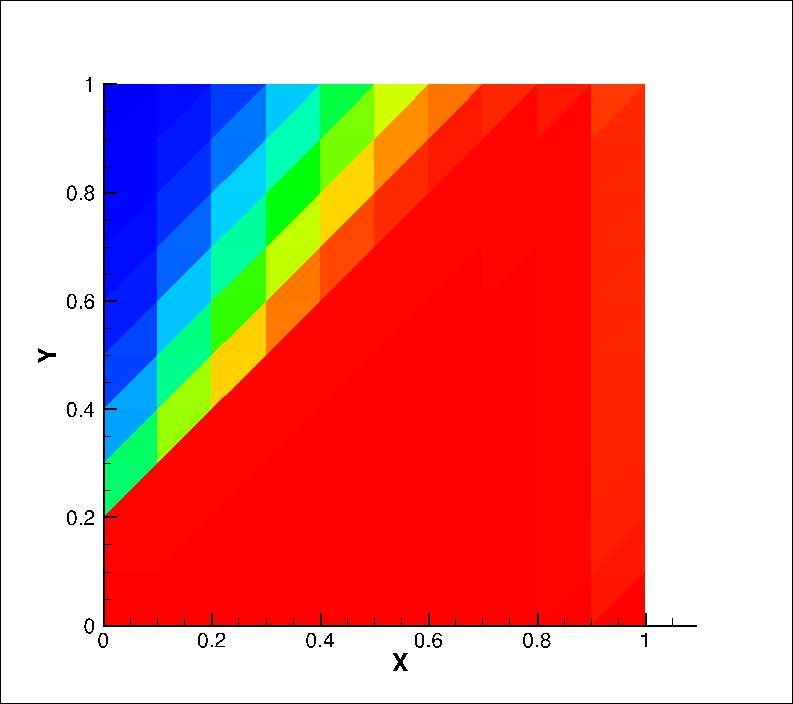}
\includegraphics[width = 0.3\textwidth]{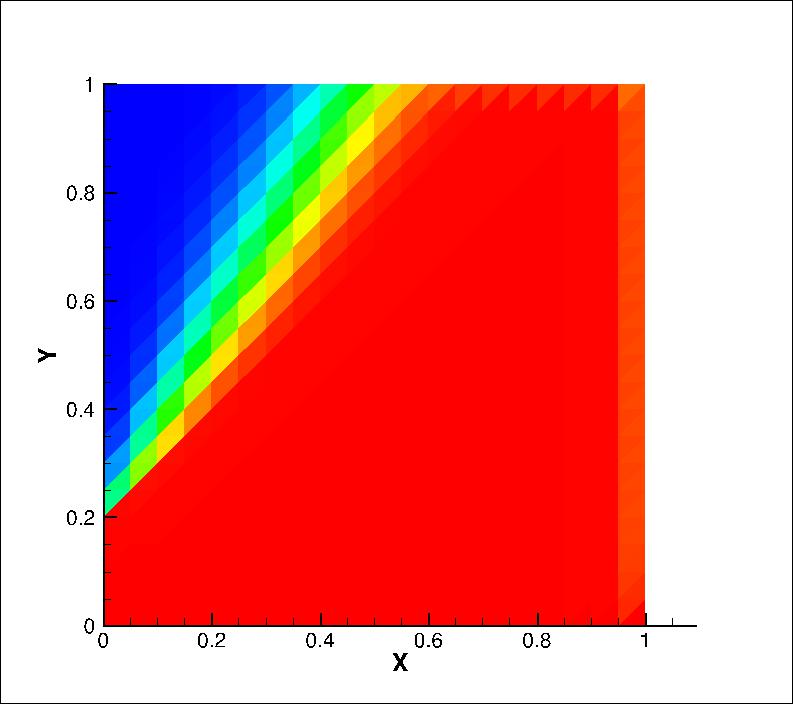}
\includegraphics[width = 0.3\textwidth]{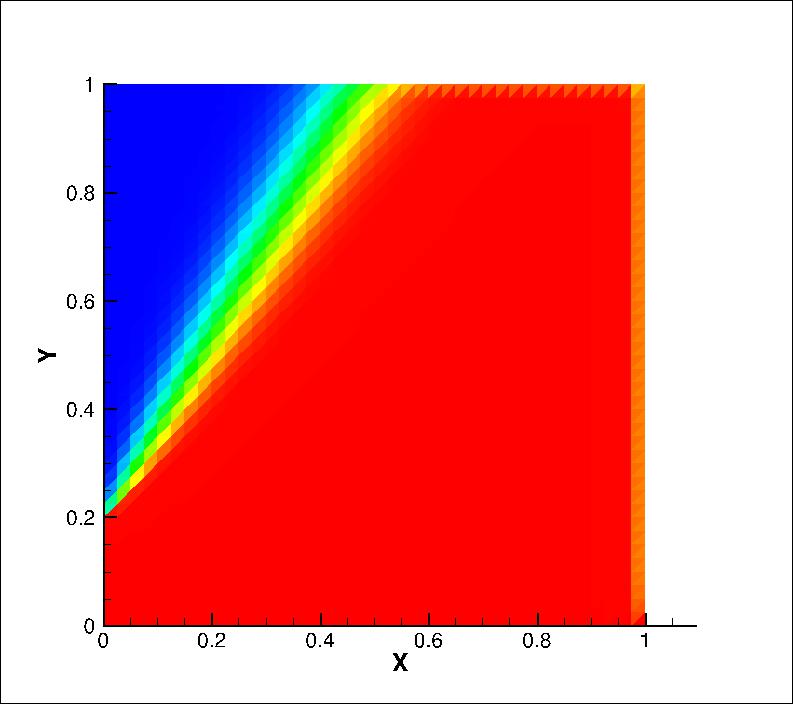}
\includegraphics[width = 0.3\textwidth]{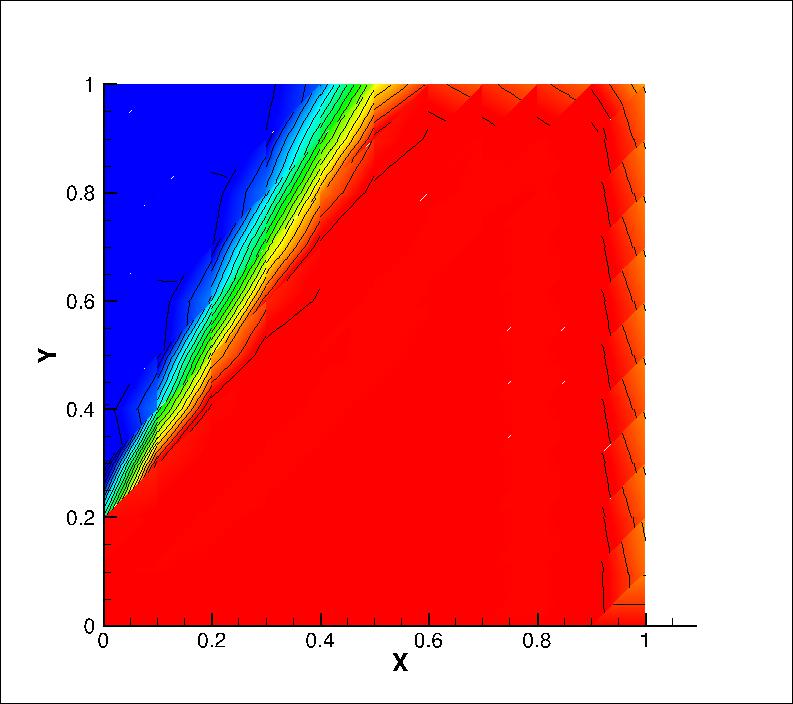}
\includegraphics[width = 0.3\textwidth]{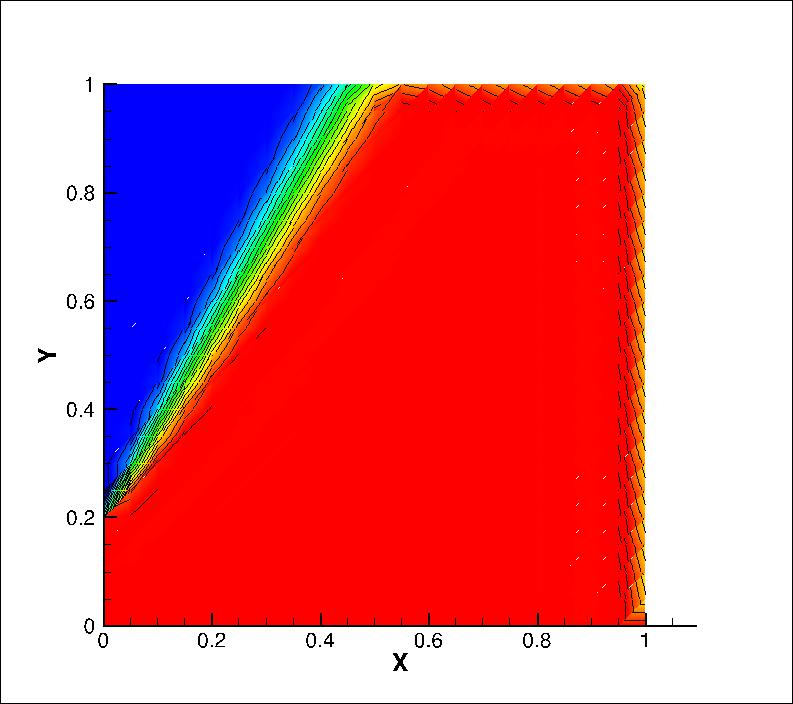}
\includegraphics[width = 0.3\textwidth]{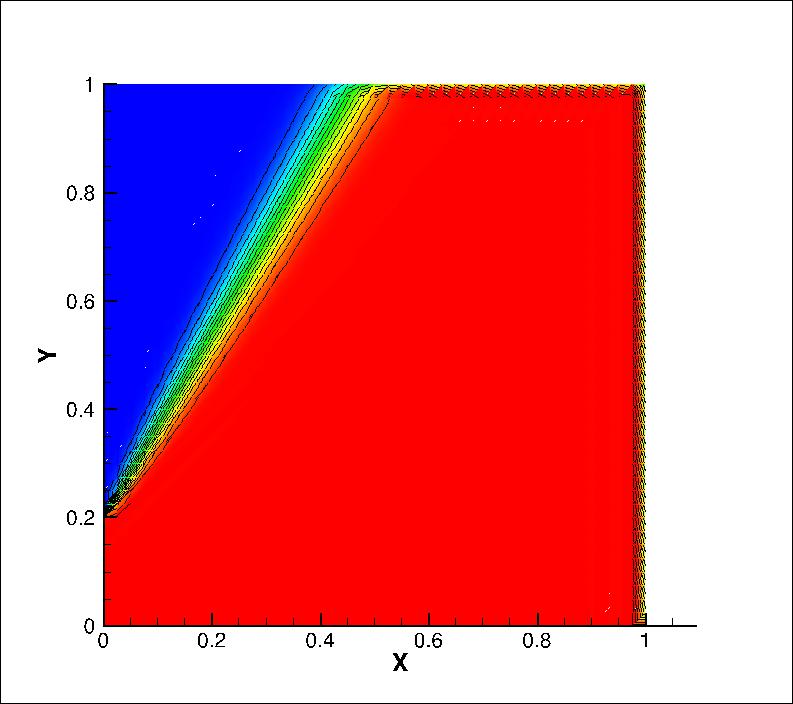}
\includegraphics[width = 0.3\textwidth]{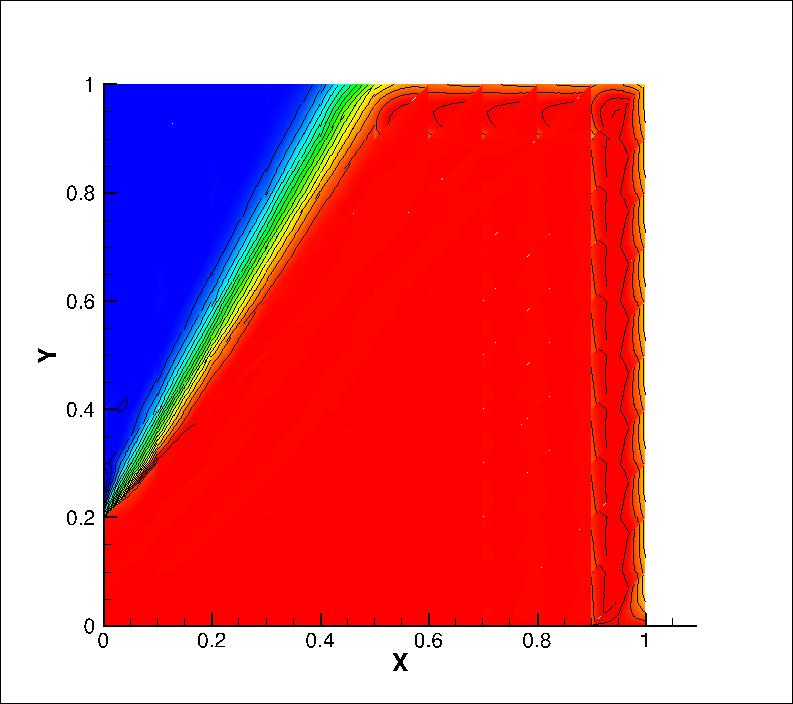}
\includegraphics[width = 0.3\textwidth]{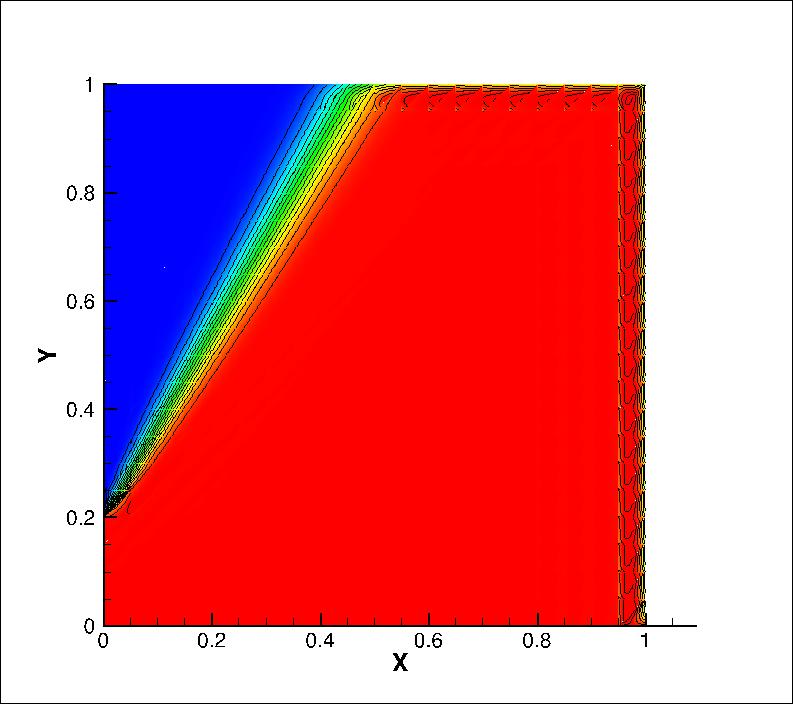}
\includegraphics[width = 0.3\textwidth]{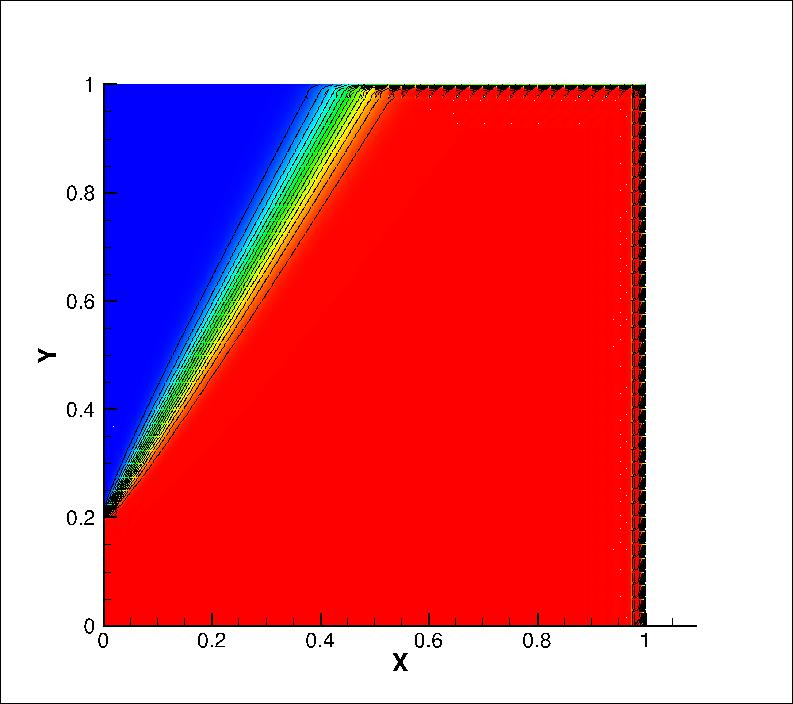}
\includegraphics[width = 0.3\textwidth]{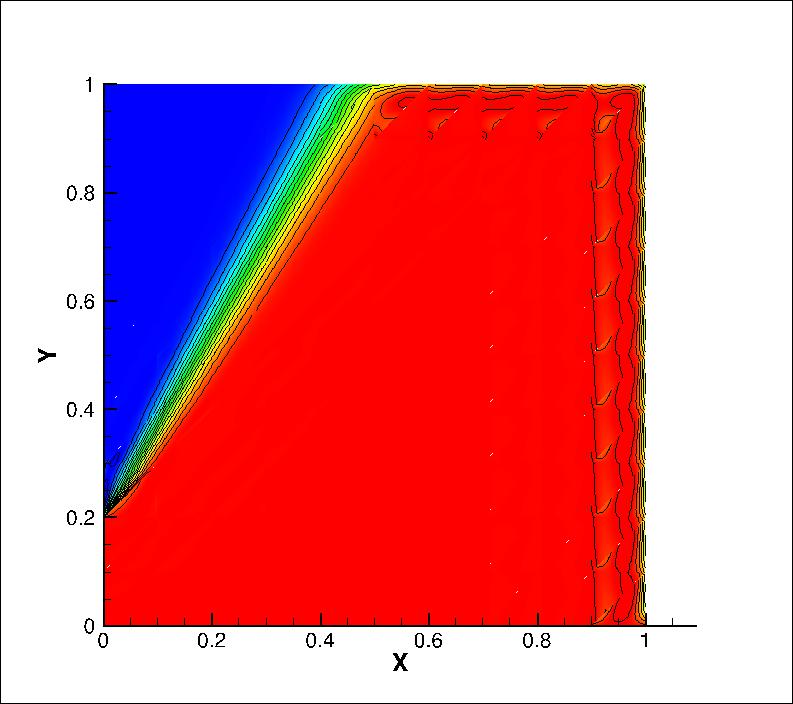}
\includegraphics[width = 0.3\textwidth]{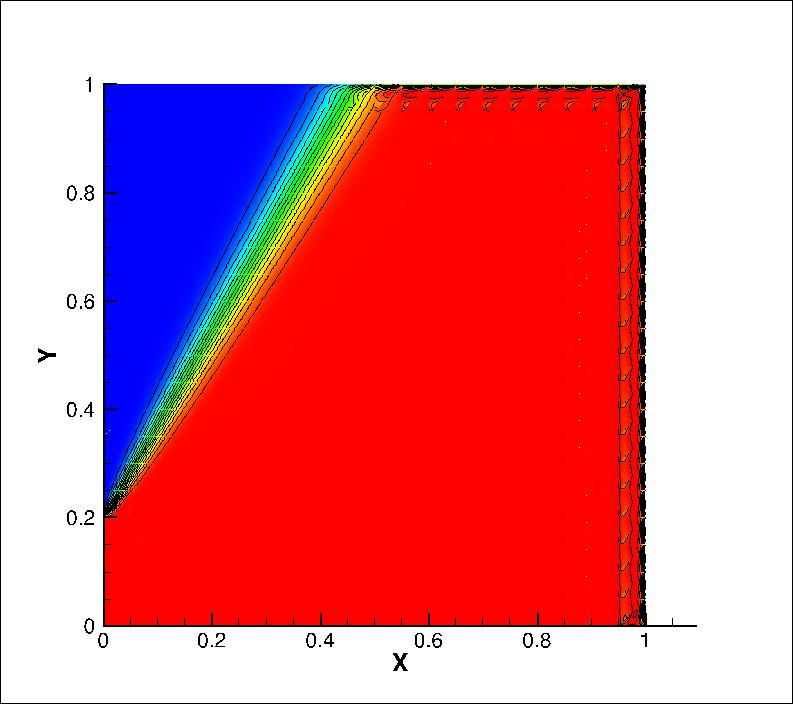}
\includegraphics[width = 0.3\textwidth]{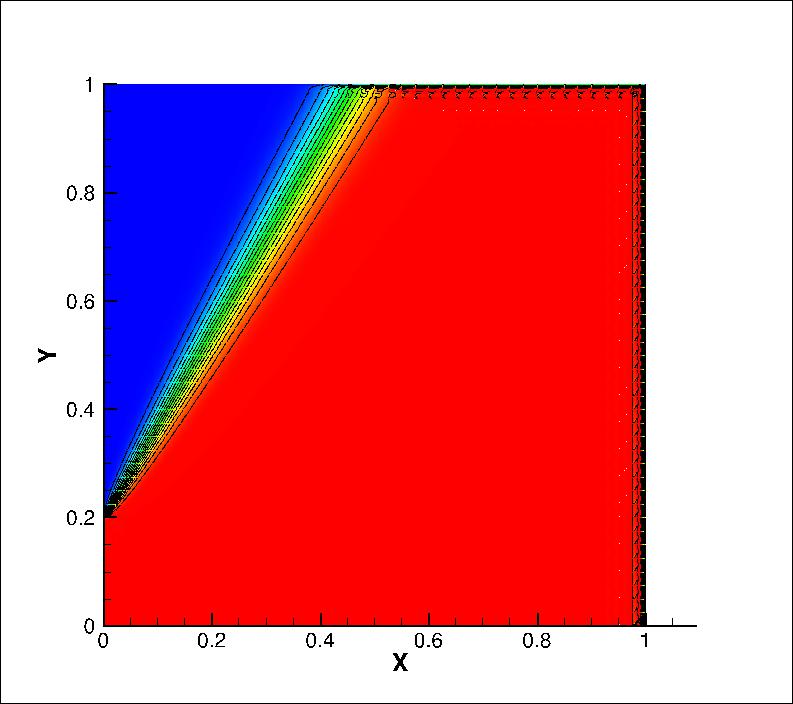}
\caption{Contour plot of $u_h$ using $HDG1$ for the interior layer test with $\epsilon = 10^{-3}$. Left to right: three consecutive meshes with the left one consists of
a uniform triangulation of 200 elements. 
Top--Bottom: $P_0$--$P_3$.}
\label{interior_contour} 
\end{figure}

\subsection{A boundary layer test}
Finally, we take $\boldsymbol{\beta} = [1, 1]^T$, and choose the source term $f$ so that the exact solution
\[
u(x,y) = \sin\frac{\pi\,x}{2} + \sin\frac{\pi\,y}{2}\left(1-\sin\frac{\pi\,x}{2}\right) + \frac{e^{-1/\epsilon} - e^{-(1-x)(1-y)/\epsilon}}{1 - e^{-1/\epsilon}}.
\]
The solution develops boundary layers along the top and right boundaries for small $\epsilon$ (see Fig.~\ref{bdry21} for $\epsilon = 10^{-2}$ and Fig.~\ref{bdry61}
for $\epsilon = 10^{-6}$). 
We take an exact solution which is a slight modification of that considered in \cite{AyusoMarini:cdf} so that, 
away from the boundary layers, our exact solution behaves not 
like a quadratic polynomial as in \cite{AyusoMarini:cdf}. 
This modification is useful for us to clearly see the orders of convergence for $k=2,3$.

In Fig.~\ref{bdry21} and Fig.~\ref{bdry61}, we plot the exact solution and computational
results for $\epsilon =  10^{-2}$ and $\epsilon = 10^{-6}$ in a structured 200 elements. 
We find that all the HDG methods produce similar results. The boundary layers are not resolved since 
the mesh is too coarse. 

In Table~\ref{bdry_order}, we show the convergence of $u_h$ in $L^2$--norm
for $\epsilon = 10^{-2},10^{-6}$ in the reduced domain $\widetilde{\Omega} = [0, 0.9]\times[0,0.9]\subset \Omega$ to exclude the unresolved boundary
layers.  Just as in the smooth case, the three HDG methods produce very similar convergence results. Hence we only show the computed results for $P_k$-$HDG1$ in Table~\ref{bdry_order}.
We observe optimal $L^2$--convergence rates for $u_h$.%, which analysis will be carried out in our forthcoming paper.

\begin{figure}[htbp]
\centering
\includegraphics[width = 0.3\textwidth]{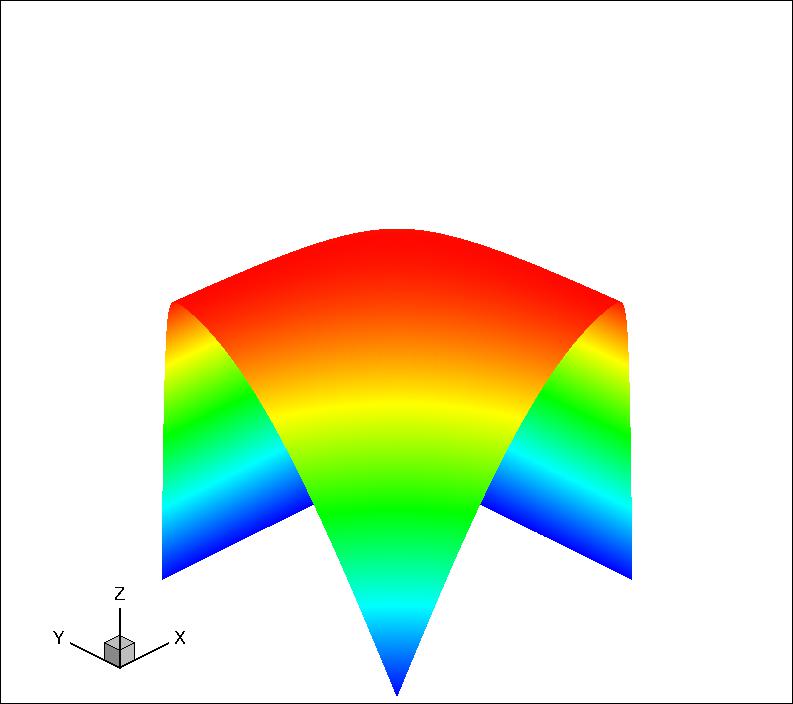}\\
\includegraphics[width = 0.3\textwidth]{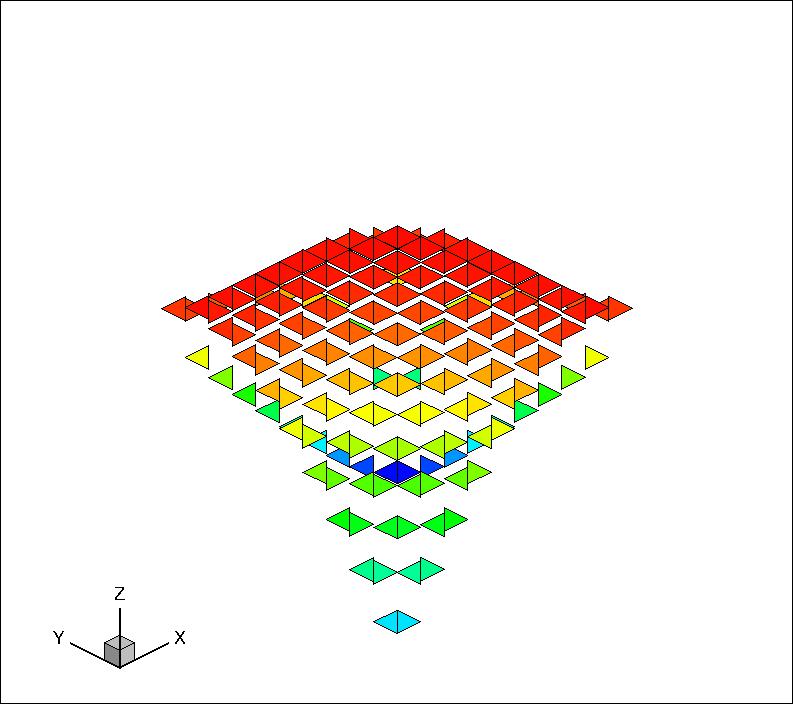}
\includegraphics[width = 0.3\textwidth]{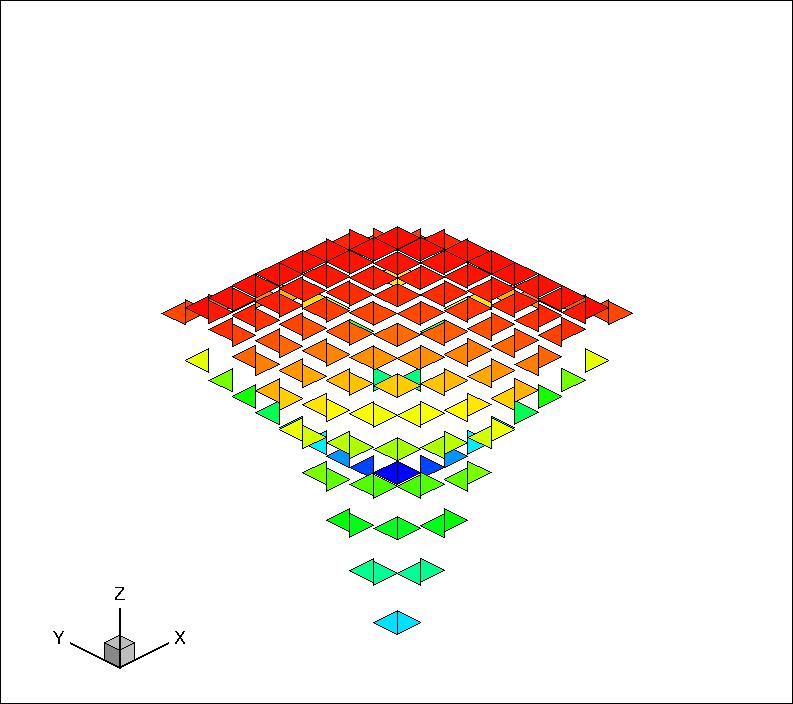}
\includegraphics[width = 0.3\textwidth]{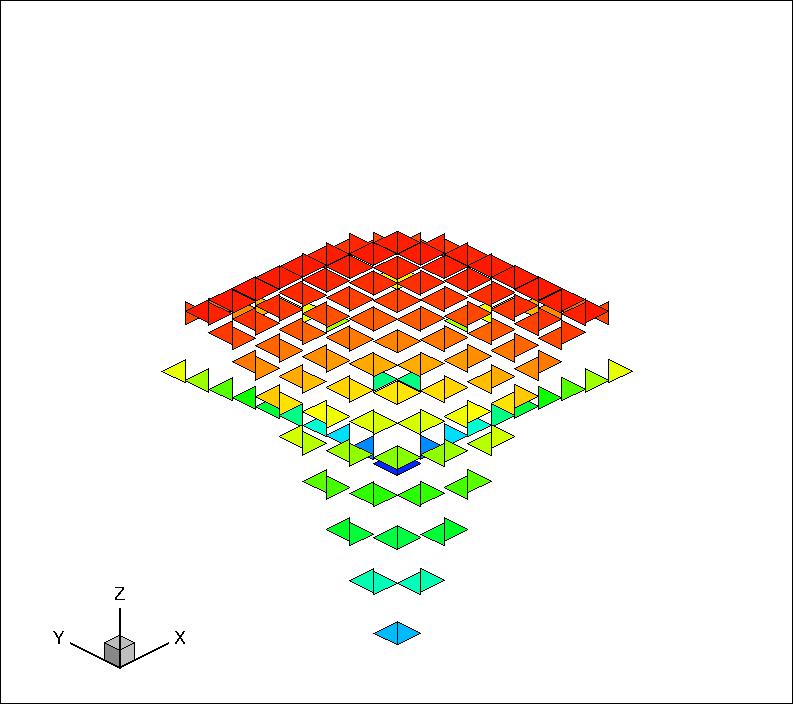}
\includegraphics[width = 0.3\textwidth]{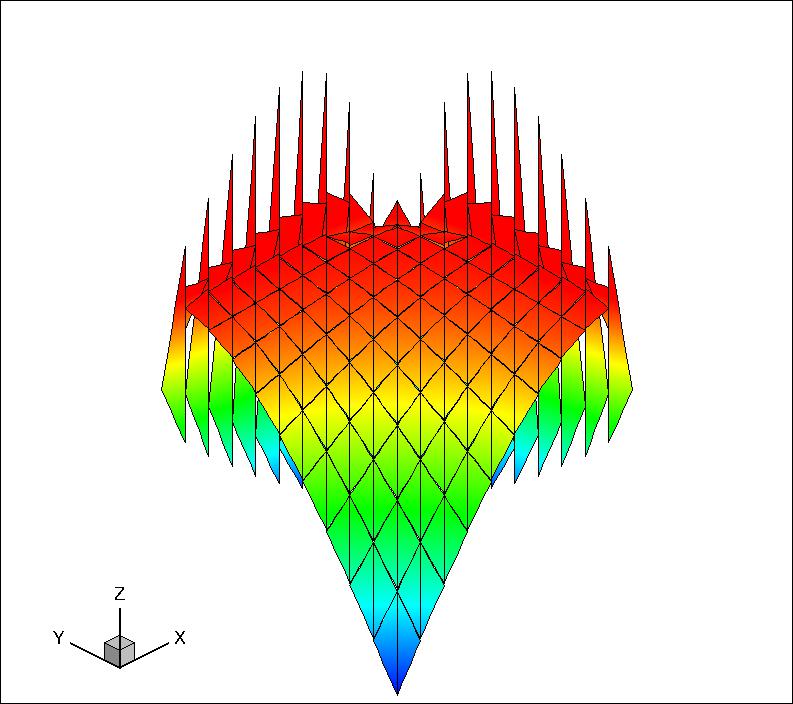}
\includegraphics[width = 0.3\textwidth]{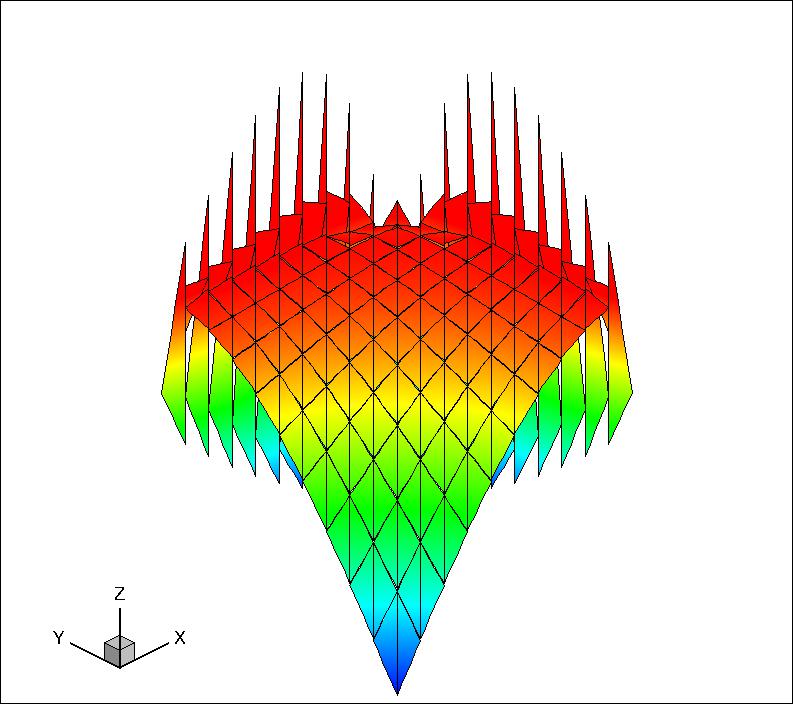}
\includegraphics[width = 0.3\textwidth]{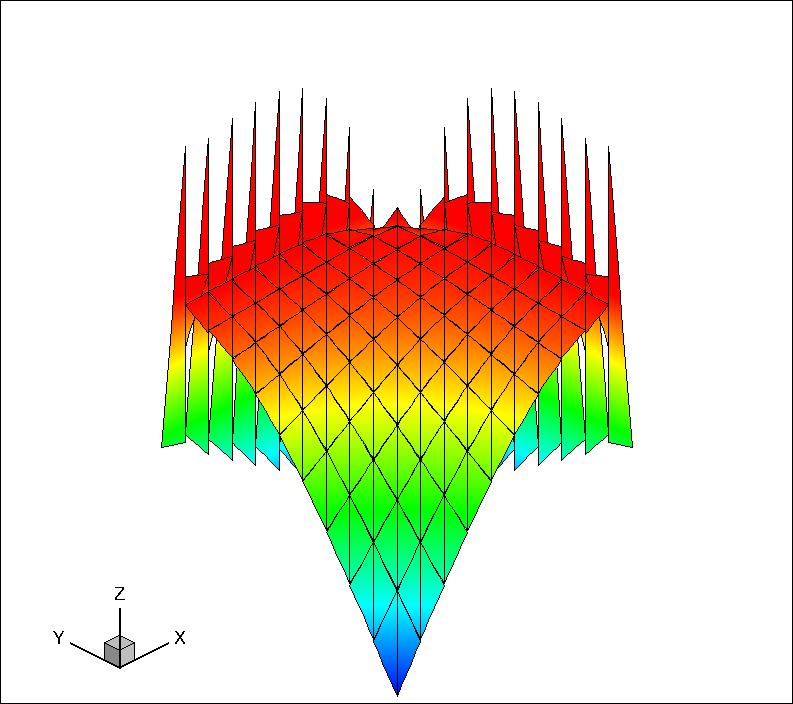}
\includegraphics[width = 0.3\textwidth]{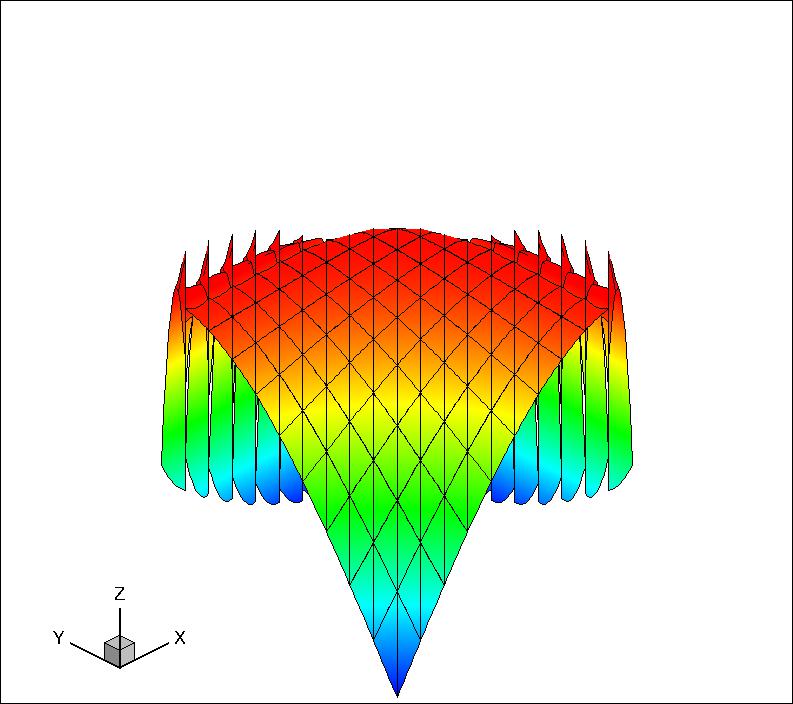}
\includegraphics[width = 0.3\textwidth]{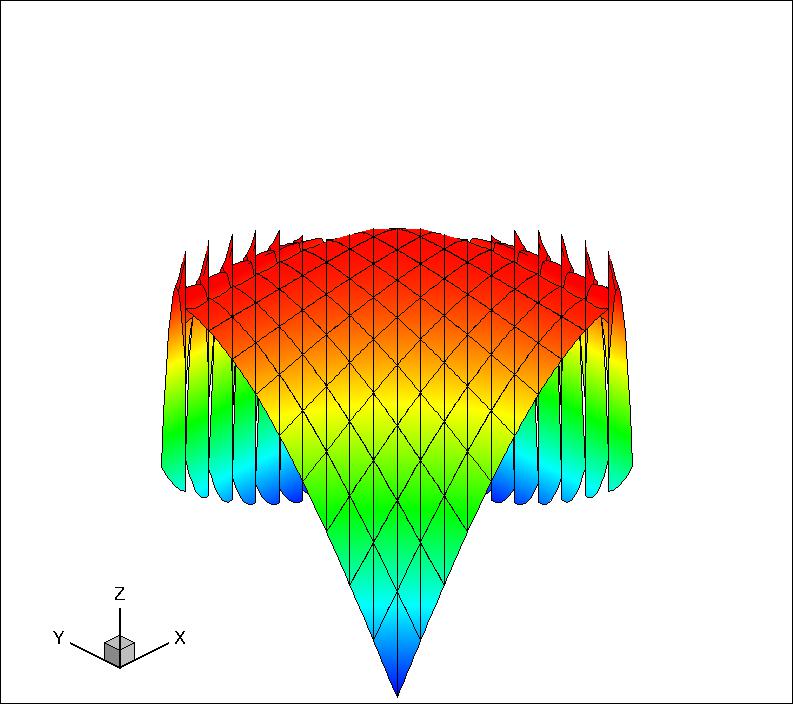}
\includegraphics[width = 0.3\textwidth]{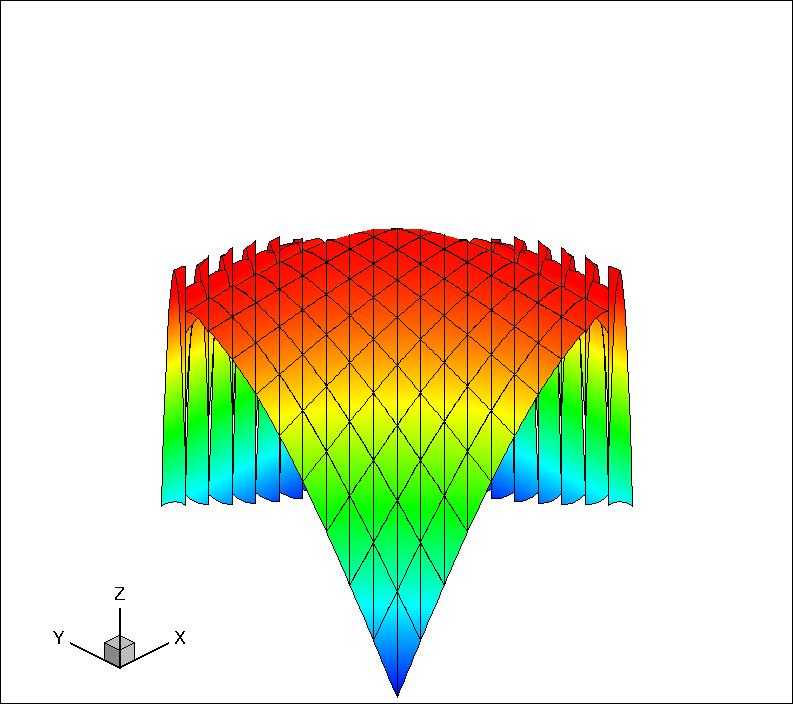}
\includegraphics[width = 0.3\textwidth]{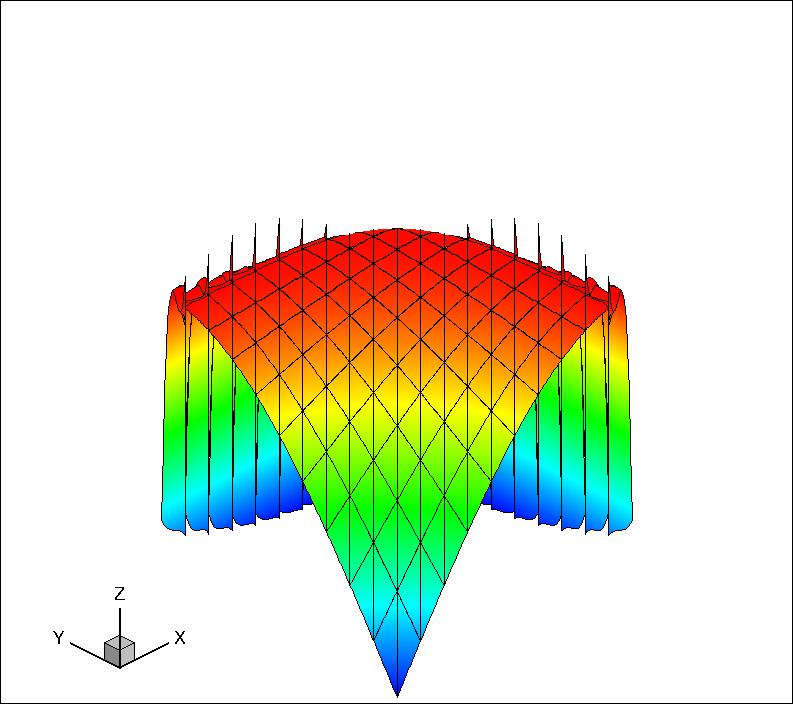}
\includegraphics[width = 0.3\textwidth]{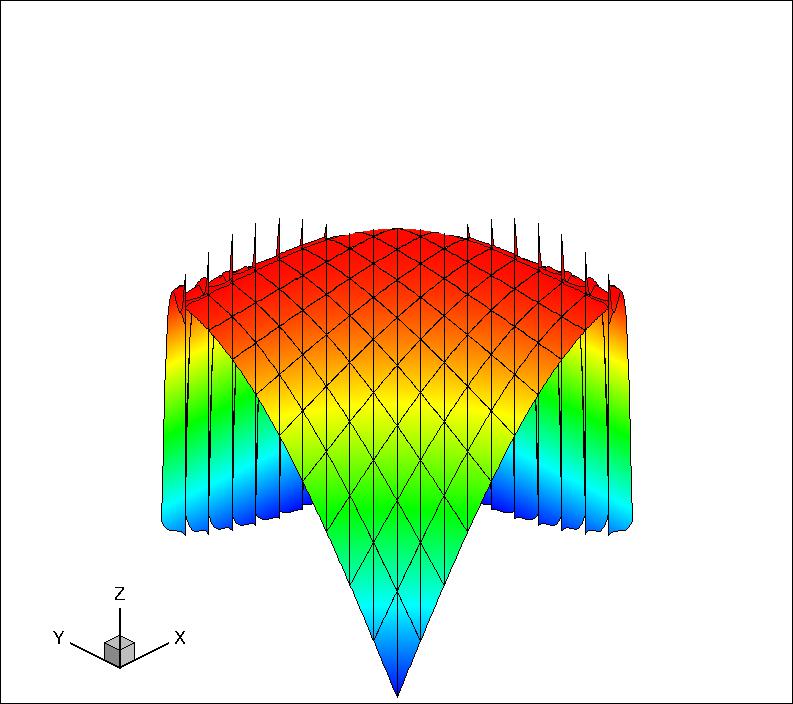}
\includegraphics[width = 0.3\textwidth]{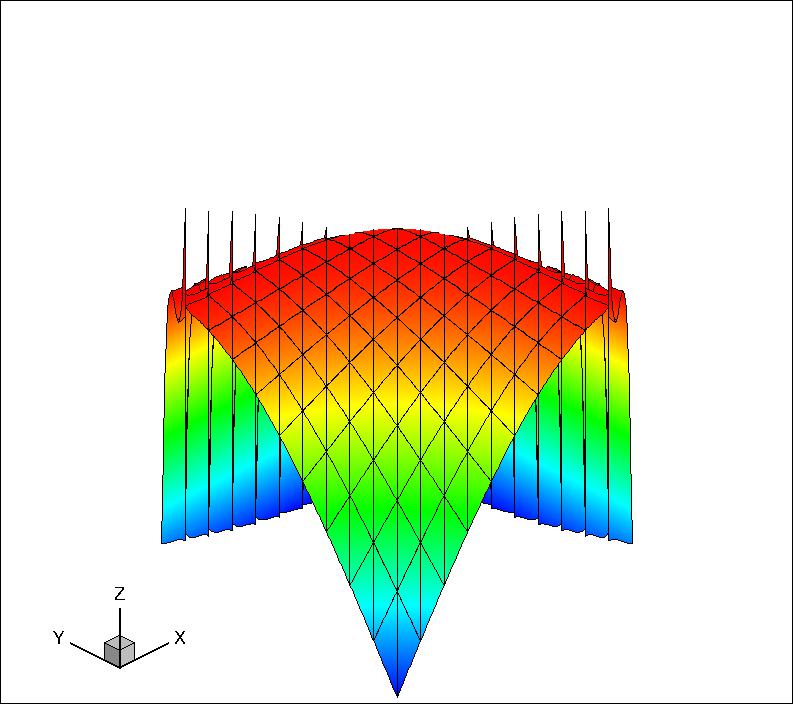}
\caption{3D plot of the exact solution and $u_h$ for the boundary layer test with $\epsilon = 10^{-2}$ in 128 elements. Top center: the exact solution.
 Left--Right: $HDG1$, $HDG2$, $HDG3$.
Top--Bottom: $P_0$--$P_3$.}
\label{bdry21} 
\end{figure}

\begin{figure}[htbp]
\centering
\includegraphics[width = 0.3\textwidth]{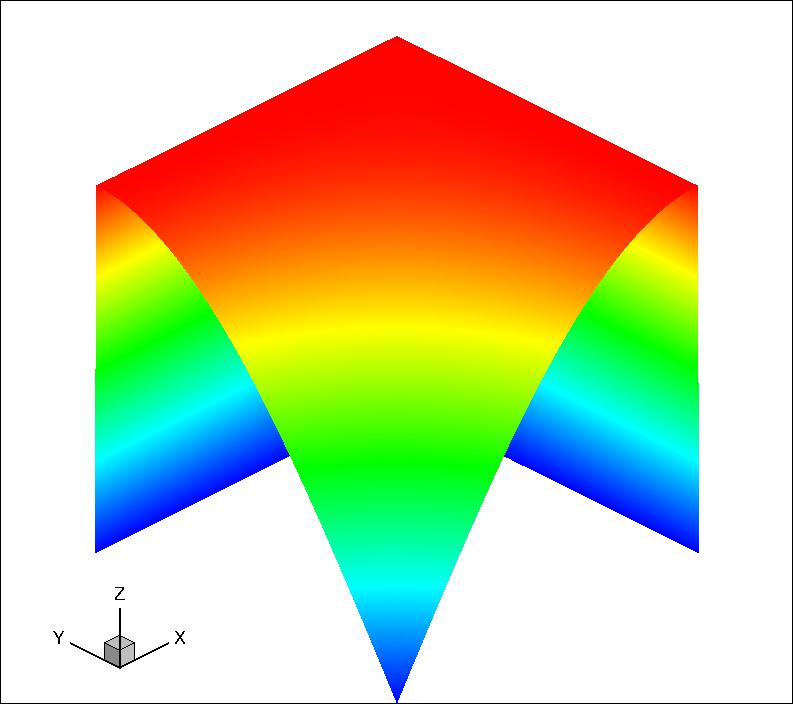}\\
\includegraphics[width = 0.3\textwidth]{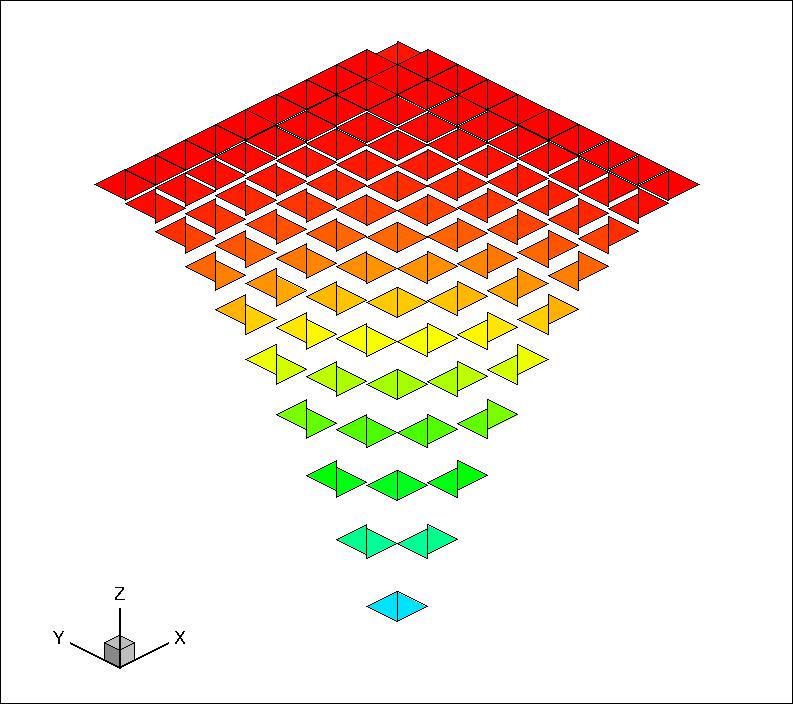}
\includegraphics[width = 0.3\textwidth]{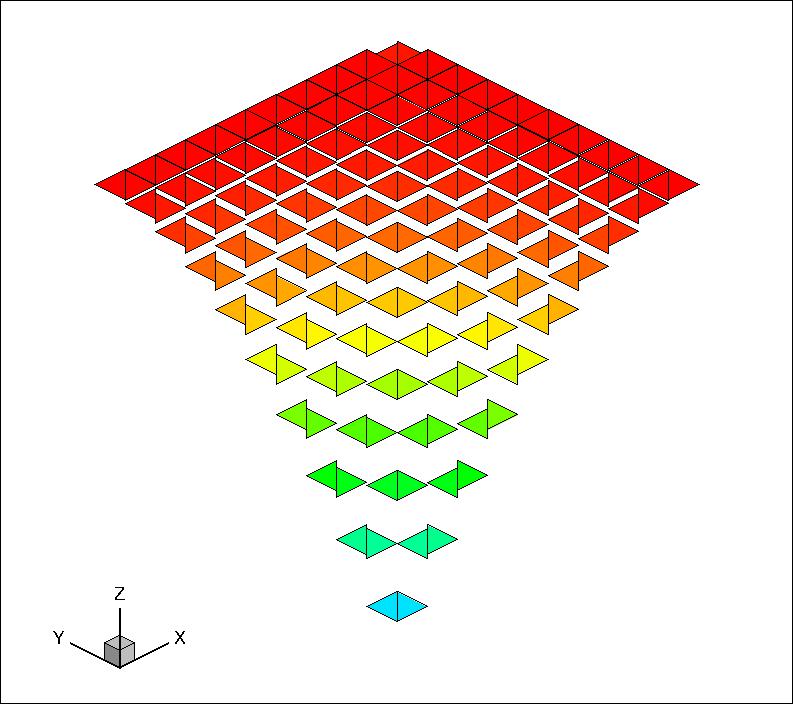}
\includegraphics[width = 0.3\textwidth]{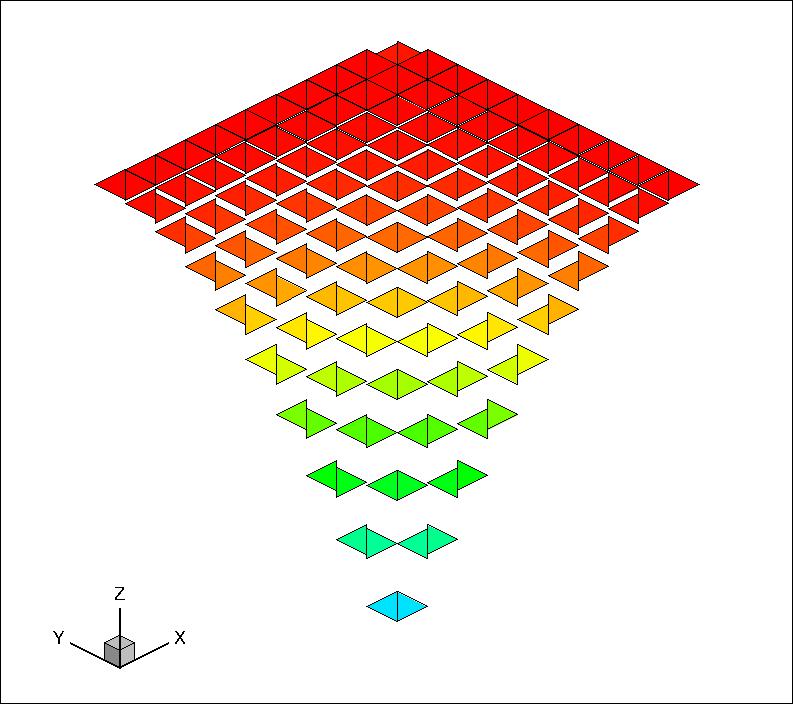}
\includegraphics[width = 0.3\textwidth]{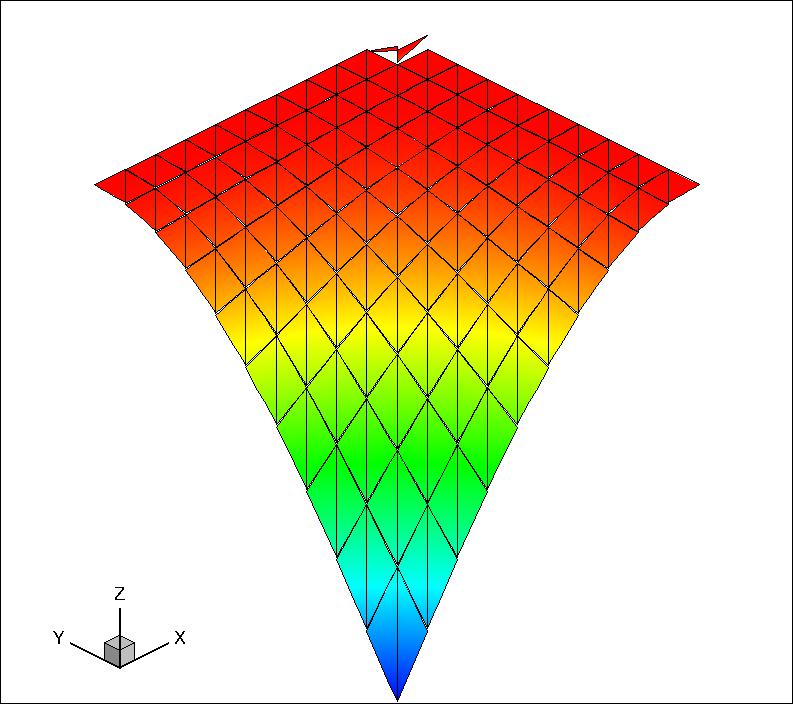}
\includegraphics[width = 0.3\textwidth]{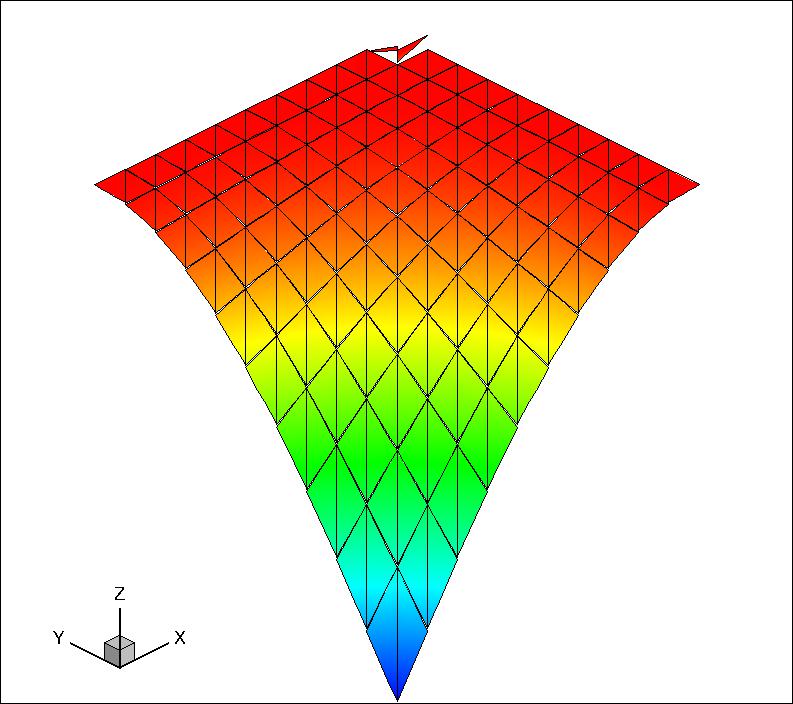}
\includegraphics[width = 0.3\textwidth]{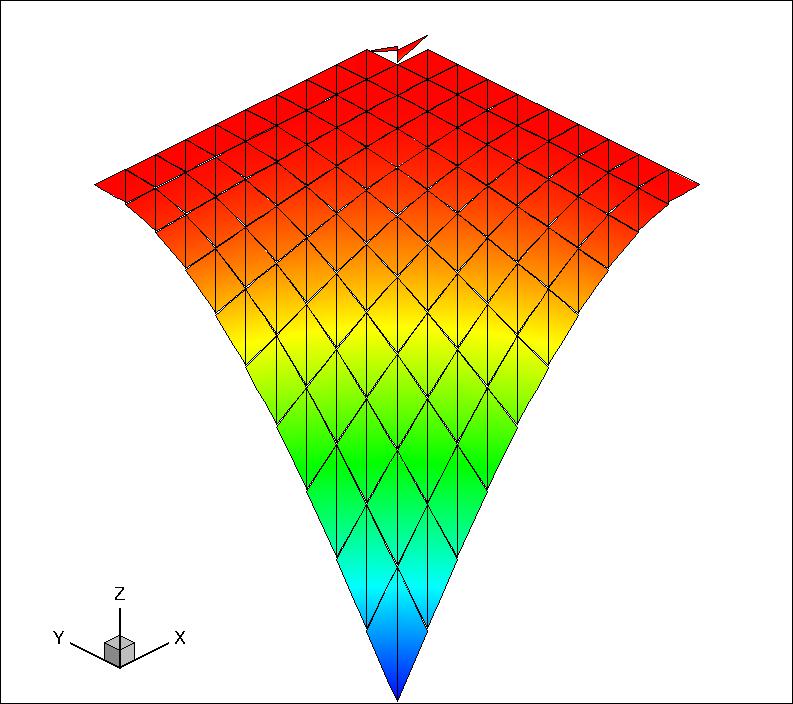}
\includegraphics[width = 0.3\textwidth]{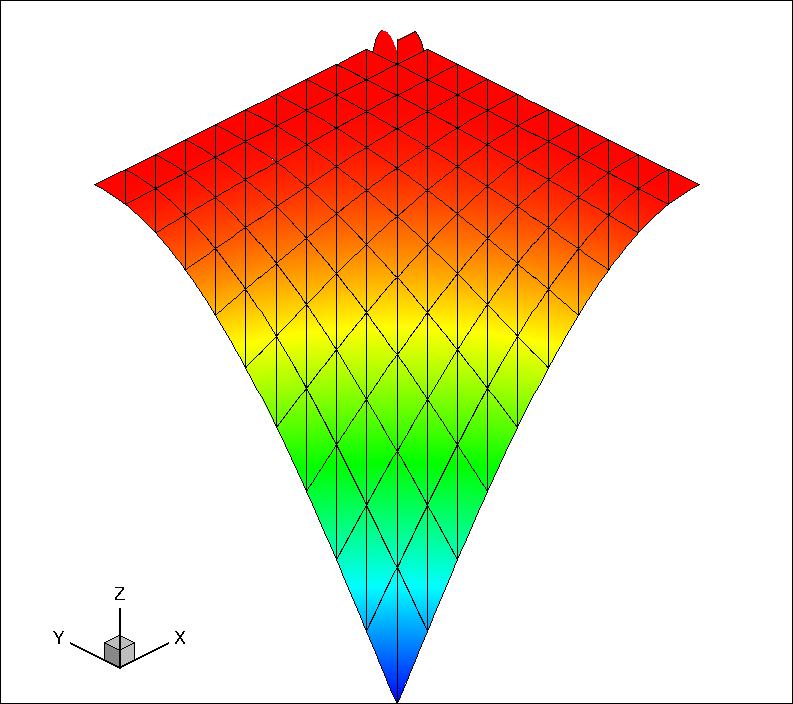}
\includegraphics[width = 0.3\textwidth]{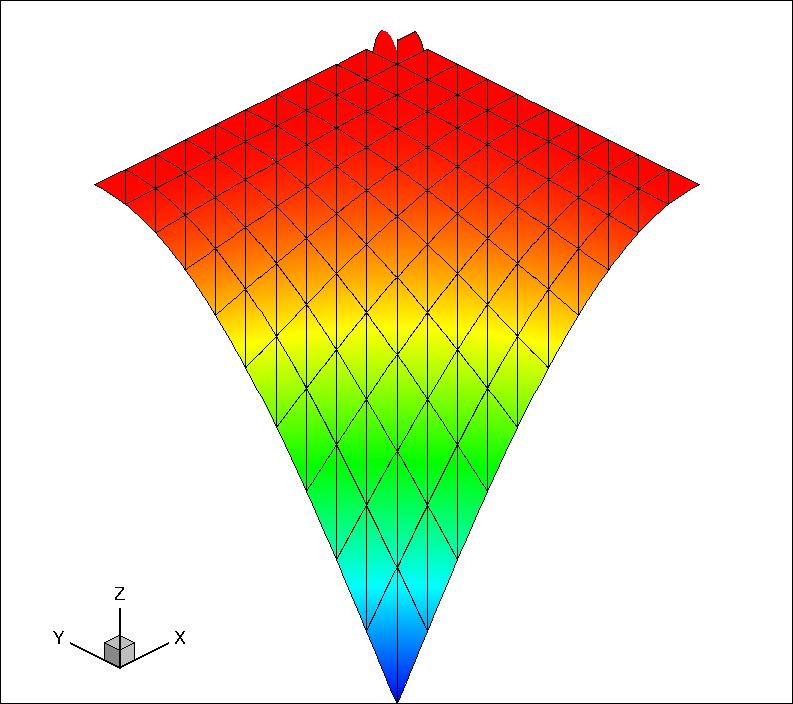}
\includegraphics[width = 0.3\textwidth]{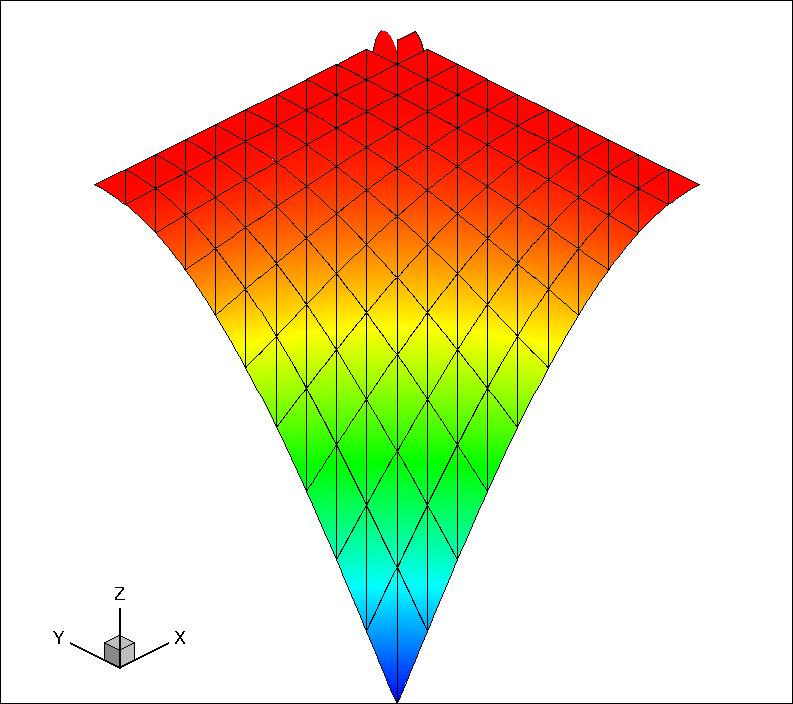}
\includegraphics[width = 0.3\textwidth]{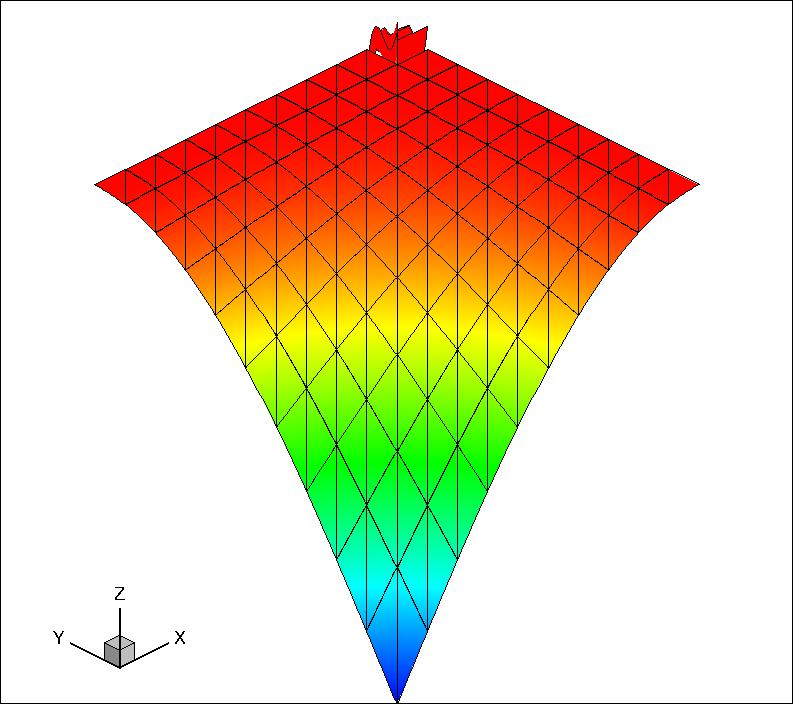}
\includegraphics[width = 0.3\textwidth]{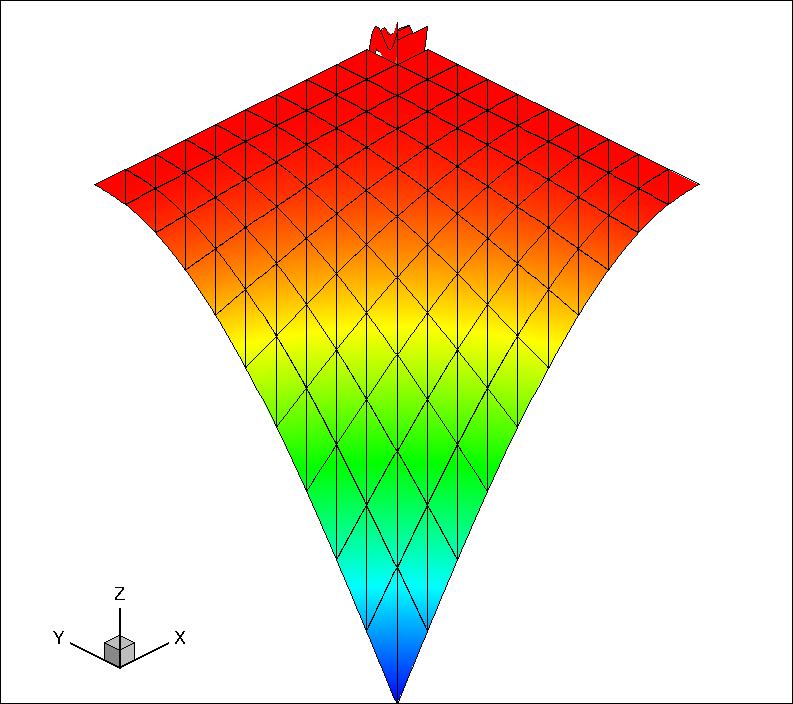}
\includegraphics[width = 0.3\textwidth]{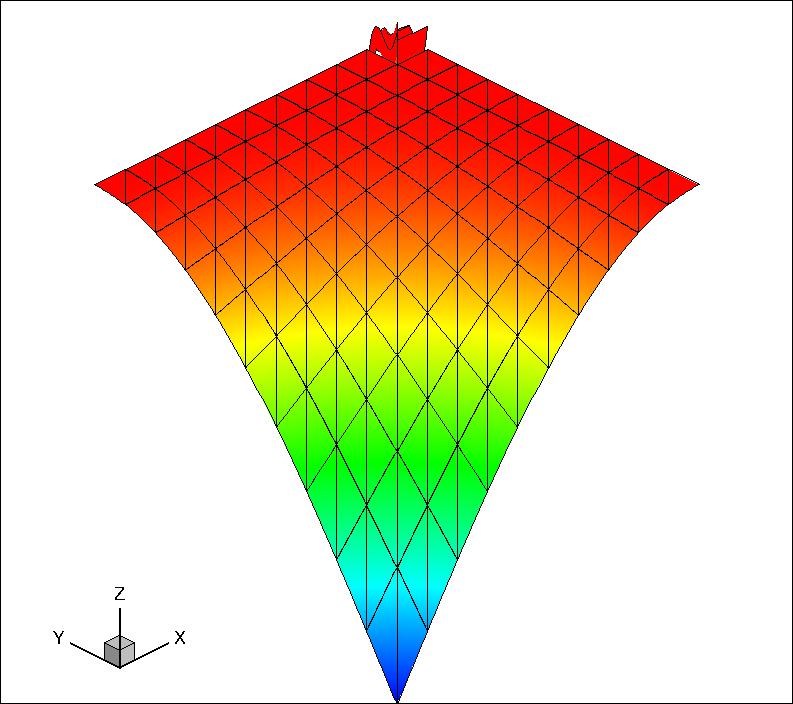}
\caption{3D plot of the exact solution and $u_h$ for the boundary layer test with $\epsilon = 10^{-6}$ in 128 elements. Top center: the exact solution.
 Left--Right: $HDG1$, $HDG2$, $HDG3$.
Top--Bottom: $P_0$--$P_3$.}
\label{bdry61} 
\end{figure}

\begin{table}[htbp]
\footnotesize
  \begin{center}
\scalebox{0.9}{%
    $\begin{array}{|c|c||c c||c c|}
    \hline
 \phantom{\big|} \mbox{degree} & \mbox{mesh} & \multicolumn{2}{|c||}{\epsilon=10^{-2}\phantom{\big|}} &\multicolumn{2}{|c|}{
\epsilon=10^{-6}\phantom{\big|}}
\\
\cline{3-6}
 \phantom{\big|}  k& h^{-1} & \mbox{error}  & \mbox{order} & \mbox{error}  & \mbox{order}\\
\hline
%& & \multicolumn{2}{|c|}{HDG-1} &  \multicolumn{2}{|c|}{HDG-1}
%\\
%\hline
  & 10  &  3.61\mbox{e-}2  &  --     &  3.32\mbox{e-}2  &  --      \\
  & 20  &  1.81\mbox{e-}2  &  0.99   &  1.67\mbox{e-}2  &  1.00   \\
 0& 40  &  9.06\mbox{e-}3  &  1.00   &  8.34\mbox{e-}3  &  1.00   \\
  & 80  &  4.52\mbox{e-}3  &  1.00   &  4.17\mbox{e-}3  &  1.00  \\
\hline
  & 10  &  4.22\mbox{e-}3  &  --     &  1.20\mbox{e-}3  &  --    \\
  & 20  &  8.54\mbox{e-}4  &  2.30   &  3.00\mbox{e-}4  &  2.00  \\
 1& 40  &  2.13\mbox{e-}4  &  2.00   &  7.51\mbox{e-}5  &  2.00  \\
  & 80  &  5.30\mbox{e-}5  &  2.01   &  1.88\mbox{e-}5  &  2.00  \\
\hline
  & 10  &  1.48\mbox{e-}3  &  --     &  1.90\mbox{e-}5  &  --     \\
  & 20  &  6.66\mbox{e-}5  &  4.47   &  2.37\mbox{e-}6  &  3.00  \\
 2& 40  &  8.19\mbox{e-}6  &  3.02   &  2.96\mbox{e-}7  &  3.00   \\
  & 80  &  1.03\mbox{e-}6  &  3.00   &  3.70\mbox{e-}8  &  3.00 \\
\hline
  & 10  &  4.10\mbox{e-}4  &  --     &  3.17\mbox{e-}7  &  --    \\
  & 20  &  5.35\mbox{e-}6  &  6.26   &  1.99\mbox{e-}8  &  3.99   \\
 3& 40  &  3.56\mbox{e-}7  &  3.91   &  1.25\mbox{e-}9  &  4.00 \\
  & 80  &  2.27\mbox{e-}8  &  3.97   &  7.79\mbox{e-}11  &  4.00 \\
\hline
 \end{array} $
}
\end{center}{$\phantom{|}$}
     \caption{History of convergence of $HDG1$ for $\|u - u_h\|_{L^2(\widetilde{\Omega})}$ when $\epsilon = 10^{-2}$ and $\epsilon = 10^{-6}$.}
   \label{bdry_order}
\end{table}

\subsection{The condition number}
Now, we present the condition number of the matrix generated by the original bilinear form $a_h$ in \eqref{hybrid-ee} and 
the scaled bilinear form $\widetilde{a}_h$ in \eqref{reduced_matrix2}. We use the same setup as that of the smooth test with two choices of 
$\bld \beta$. The first choice of $\bld \beta$ is $\bld\beta = [1,2]^T$. For this choice, assumption
\eqref{assump_cond} is satisfied by both example of $\tau$ in \eqref{tau1} and \eqref{tau2}. 
Since the condition numbers of all three HDG methods are very similar in our tests, we only present
that for $P_k$-$HDG1$ in Table~\ref{cond_1}. Notice that $\mathcal{O}(h^{-2})$ is observed for $\epsilon = 1,10^{-3},10^{-9}$ for different polynomial degrees, and that the condition number of 
the scaled system is similar to that of unscaled system.
The second choice of $\bld \beta$ is $\bld\beta = [1,1]^T$. For this choice, assumption 
\eqref{assump_cond} is satisfied by the second choice of $\tau$ in \eqref{tau2}, but not for the choice \eqref{tau1} since the mesh is aligned with $\bld\beta$. However, one can easily modify 
$\tau$ in $P_k$-$HDG1$ and $P_k$-$HDG3$ on the aligned faces so that \eqref{assump_cond} holds. We only present the condition numbers for $P_k$-$HDG2$ in Table~\ref{cond_2}. 
The dependence of the condition number on $\epsilon$ seems to be of order $\mathcal{O}({\epsilon^{-1}})$ for the original system, 
and we observe a huge improvement
of the condition number after scaling for $\epsilon = 10^{-9}$. Also, we find that the condition number for the scaled system is of order $\mathcal{O}(h^{-2})$ for $\epsilon = 1$, and 
of order $\mathcal{O}(h^{-1})$ for $\epsilon = 10^{-3},10^{-9}$, which is not predicted by our theory.
\begin{table}[htbp]
\footnotesize
  \begin{center}
\scalebox{0.9}{%
    $\begin{array}{|c|c||c|c|c|c|| c|c|c|c|}
    \hline
 \phantom{\big|} \mbox{$\epsilon$} & \mbox{mesh} & \multicolumn{4}{|c||}{\text{Condition numbers
for $a_h(\cdot,\cdot)$} \phantom{\big|}} & \multicolumn{4}{|c|}{\text{Condition numbers
for $\widetilde{a}_h(\cdot,\cdot)$} \phantom{\big|}} 
\\
\cline{3-10}
 \phantom{\big|}  & h^{-1} & \mbox{k=}0  & \mbox{k=}1 & \mbox{k=}2  & \mbox{k=}3& \mbox{k=}0  & \mbox{k=}1 & \mbox{k=}2  & \mbox{k=}3\\
\hline
%& & \multicolumn{2}{|c|}{HDG-1} &  \multicolumn{2}{|c|}{HDG-1}
%\\
%\hline
               & 5   &  1.11\mbox{e+}2   &  1.42\mbox{e+}2    &  2.82\mbox{e+}2 &  3.31\mbox{e+}2 &  2.49\mbox{e+}2   &  2.65\mbox{e+}2    &  5.30\mbox{e+}2 &  6.29\mbox{e+}2\\
               & 10   &  3.37\mbox{e+}2  &  5.28\mbox{e+}2    &  1.06\mbox{e+}3 &  1.24\mbox{e+}2 &  6.29\mbox{e+}2   &  9.83\mbox{e+}2    &  1.98\mbox{e+}3 &  2.34\mbox{e+}2 \\
 1\mbox{e-}0   & 20  &  1.37\mbox{e+}3   &  2.09\mbox{e+}3    &  4.17\mbox{e+}3 &  4.89\mbox{e+}2 &  2.55\mbox{e+}3   &  3.89\mbox{e+}3    &  7.78\mbox{e+}3 &  9.21\mbox{e+}2 \\
               & 40   &  5.50\mbox{e+}3  &  8.32\mbox{e+}3    &  1.66\mbox{e+}4 &  1.95\mbox{e+}2 &  1.03\mbox{e+}4   &  1.55\mbox{e+}4    &  3.10\mbox{e+}4 &  3.67\mbox{e+}2 \\
\hline
               & 5   &  4.86\mbox{e+}1   &  1.60\mbox{e+}2    &  3.64\mbox{e+}2  &  5.00\mbox{e+}2  &  3.71\mbox{e+}1   &  1.12\mbox{e+}2    &  2.44\mbox{e+}2 &  3.40\mbox{e+}2 \\
               & 10   &  1.74\mbox{e+}2   &  5.23\mbox{e+}2    &  8.79\mbox{e+}2 &  1.27\mbox{e+}3  &  1.36\mbox{e+}2   &  3.75\mbox{e+}2    &  6.51\mbox{e+}2 &  8.48\mbox{e+}2 \\
 1\mbox{e-}3   & 20   &  6.48\mbox{e+}2   &  1.82\mbox{e+}3    &  2.35\mbox{e+}3 &  4.06\mbox{e+}3  &  5.06\mbox{e+}2   &  1.25\mbox{e+}3    &  1.17\mbox{e+}3 &  2.50\mbox{e+}3 \\
               & 40   &  2.48\mbox{e+}3   &  6.94\mbox{e+}3    &  8.76\mbox{e+}3 &  1.49\mbox{e+}4  &  1.94\mbox{e+}3   &  4.44\mbox{e+}3    &  5.98\mbox{e+}3 &  8.96\mbox{e+}3 \\
\hline
                & 5   &  4.90\mbox{e+}1   &  1.69\mbox{e+}2    &  4.21\mbox{e+}2 &  5.76\mbox{e+}2  &  3.73\mbox{e+}1   &  9.57\mbox{e+}1    &  2.97\mbox{e+}2 &  3.80\mbox{e+}2\\
               & 10   &  1.74\mbox{e+}2   &  5.38\mbox{e+}2    &  1.17\mbox{e+}3 &  1.52\mbox{e+}3  &  1.36\mbox{e+}2   &  4.04\mbox{e+}2    &  9.11\mbox{e+}2 &  1.05\mbox{e+}3 \\
 1\mbox{e-}9   & 20   &  6.49\mbox{e+}2   &  1.94\mbox{e+}3    &  3.76\mbox{e+}3 &  4.73\mbox{e+}3  &  5.05\mbox{e+}2   &  1.46\mbox{e+}3    &  2.94\mbox{e+}3 &  3.40\mbox{e+}3 \\
               & 40   &  2.50\mbox{e+}3   &  7.06\mbox{e+}3    &  1.34\mbox{e+}4 &  1.65\mbox{e+}4  &  1.93\mbox{e+}3   &  5.37\mbox{e+}3    &  9.45\mbox{e+}3 &  1.25\mbox{e+}4\\
\hline
 \end{array} $
}
\end{center}{$\phantom{|}$}
     \caption{Condition numbers for $HDG1$ when $\bld\beta =[1,2]^T$.}
   \label{cond_1}
\end{table}

\begin{table}[htbp]
\footnotesize
  \begin{center}
\scalebox{0.9}{%
    $\begin{array}{|c|c||c|c|c|c|| c|c|c|c|}
    \hline
 \phantom{\big|} \mbox{$\epsilon$} & \mbox{mesh} & \multicolumn{4}{|c||}{\text{Condition numbers
for $a_h(\cdot,\cdot)$} \phantom{\big|}} & \multicolumn{4}{|c|}{\text{Condition numbers
for $\widetilde{a}_h(\cdot,\cdot)$} \phantom{\big|}} 
\\
\cline{3-10}
 \phantom{\big|}  & h^{-1} & \mbox{k=}0  & \mbox{k=}1 & \mbox{k=}2  & \mbox{k=}3& \mbox{k=}0  & \mbox{k=}1 & \mbox{k=}2  & \mbox{k=}3\\
\hline
%& & \multicolumn{2}{|c|}{HDG-1} &  \multicolumn{2}{|c|}{HDG-1}
%\\
%\hline
               & 5    &  9.99\mbox{e+}1   &  1.41\mbox{e+}2    &  2.82\mbox{e+}2 &  3.31\mbox{e+}2 &  1.16\mbox{e+}2   &  1.96\mbox{e+}2    &  3.92\mbox{e+}2 &  4.64\mbox{e+}2\\
               & 10   &  3.40\mbox{e+}2   &  5.34\mbox{e+}2    &  1.07\mbox{e+}3 &  1.25\mbox{e+}2 &  4.35\mbox{e+}2   &  6.82\mbox{e+}2    &  1.37\mbox{e+}3 &  1.61\mbox{e+}2 \\
 1\mbox{e-}0   & 20   &  1.38\mbox{e+}3   &  2.12\mbox{e+}3    &  4.23\mbox{e+}3 &  4.96\mbox{e+}2 &  1.77\mbox{e+}3   &  2.71\mbox{e+}3    &  5.41\mbox{e+}3 &  6.38\mbox{e+}2 \\
               & 40   &  5.57\mbox{e+}3   &  8.44\mbox{e+}3    &  1.68\mbox{e+}4 &  1.98\mbox{e+}2 &  7.12\mbox{e+}3   &  1.08\mbox{e+}4    &  2.15\mbox{e+}4 &  2.54\mbox{e+}2 \\
\hline
               & 5    &  2.01\mbox{e+}2   &  6.63\mbox{e+}2    &  9.06\mbox{e+}2 &  1.55\mbox{e+}3  &  2.48\mbox{e+}2   &  1.81\mbox{e+}3    &  7.02\mbox{e+}3 &  9.83\mbox{e+}3 \\
               & 10   &  4.00\mbox{e+}2   &  1.41\mbox{e+}3    &  1.62\mbox{e+}3 &  3.27\mbox{e+}3  &  5.63\mbox{e+}2   &  2.34\mbox{e+}3    &  7.45\mbox{e+}3 &  1.16\mbox{e+}4 \\
 1\mbox{e-}3   & 20   &  1.25\mbox{e+}3   &  4.42\mbox{e+}3    &  4.98\mbox{e+}3 &  9.57\mbox{e+}3  &  1.20\mbox{e+}3   &  3.99\mbox{e+}3    &  9.01\mbox{e+}3 &  1.66\mbox{e+}4 \\
               & 40   &  4.69\mbox{e+}3   &  1.63\mbox{e+}4    &  1.86\mbox{e+}4 &  3.12\mbox{e+}4  &  2.60\mbox{e+}3   &  7.29\mbox{e+}3    &  1.32\mbox{e+}4 &  2.49\mbox{e+}4 \\
\hline
                & 5   &  1.43\mbox{e+}8   &  4.02\mbox{e+}8    &  6.04\mbox{e+}8 &  7.34\mbox{e+}8  &  2.38\mbox{e+}2   &  2.07\mbox{e+}3    &  1.23\mbox{e+}4 &  2.49\mbox{e+}4\\
               & 10   &  1.31\mbox{e+}8   &  3.84\mbox{e+}8    &  5.68\mbox{e+}8 &  7.03\mbox{e+}8  &  5.35\mbox{e+}2   &  4.66\mbox{e+}3    &  2.77\mbox{e+}4 &  5.60\mbox{e+}4 \\
 1\mbox{e-}9   & 20   &  1.25\mbox{e+}8   &  3.75\mbox{e+}8    &  5.51\mbox{e+}8 &  6.87\mbox{e+}8  &  1.13\mbox{e+}3   &  9.84\mbox{e+}3    &  5.85\mbox{e+}4 &  1.12\mbox{e+}5 \\
               & 40   &  1.22\mbox{e+}8   &  3.70\mbox{e+}8    &  5.42\mbox{e+}8 &  6.79\mbox{e+}8  &  2.32\mbox{e+}3   &  2.02\mbox{e+}4    &  1.20\mbox{e+}5 &  2.30\mbox{e+}5\\
\hline
 \end{array} $
}
\end{center}{$\phantom{|}$}
     \caption{Condition numbers for $HDG2$ when $\bld\beta =[1,1]^T$.}
   \label{cond_2}
\end{table}

%{\bf Acknowledgements}. The work of Weifeng Qiu was supported by City University of Hong Kong under start-up grant (No. 7200324).

{\bf Acknowledgements}. The authors would like to thank Professor Bernardo Cockburn 
for constructive criticism leading to a better presentation of the material in this paper. The authors would also like to 
thank one of the referees for pointing out the paper by Egger and Sch\"oberl \cite{Egger2010} and 
for suggesting to comment on the relation of the HDG method to the MH--DG method in \cite{Egger2010}, as was done in Appendix~\ref{appendix-1}.
The work of Weifeng Qiu was supported by City University of Hong Kong under start-up grant (No. 7200324).

\appendix
\section{The relation between the HDG method %\eqref{cd_hdg_eqs} 
and the MH--DG method in \cite{Egger2010}}
\label{appendix-1}
In this section, we establish the relation between the HDG method \eqref{cd_hdg_eqs} and  the MH--DG considered in \cite{Egger2010}. We first present the  MH--DG method for equations \eqref{cd_eqs} under the condition that $g= 0$ and $\bld \beta \in H(\mathrm{div};\Omega)$ is
 constant in each element. Then, we show that this method coincides with the HDG method \eqref{cd_hdg_eqs} when using the same approximation spaces and
choosing the stability function $\tau$ to be \eqref{tau1}.
%\subsection{The MH--DG method}

The MH--DG method seeks an approximation $(\bld q_h,u_h,\lambda_h)\in \widetilde{\bld V}_h\times
W_h\times M_h(0)$ so that 
\begin{align}
\label{mhdg}
B_h((\bld q_h,u_h,\lambda_h), (\bld r, w,\mu)) = (f,w)_{\Oh},
\end{align}
for all $(\bld r, w,\mu)\in  \widetilde{\bld V}_h\times
W_h\times M_h(0)$,
where $W_h$ and $M_h(0)$ is defined in \eqref{fem_spaces} and $\widetilde{\bld V}_h$ is the so called Raviart--Thomas space, slightly larger than $\bld V_h$,
 defined as follows
\begin{align*}
\widetilde{\bld V}_{h}&=\{\boldsymbol{r}\in L^{2}(\Omega;\mathbb{R}^{d}):\boldsymbol{r}|_{K}\in P_{k}(K;\mathbb{R}^{d}) + \bld x P_k(K)
\quad \forall K\in \mathcal{T}_{h}\},
\end{align*} 
and 
\begin{align*}
B_h((\boldsymbol{q},u,\lambda),(\boldsymbol{r},w,\mu))%\\
= & \;(\epsilon^{-1}\boldsymbol{q},\boldsymbol{r})_{\mathcal{T}_{h}}-(u,\nabla\cdot\boldsymbol{r})_{\mathcal{T}_{h}}
+\langle \lambda, \boldsymbol{r}\cdot \boldsymbol{n}\rangle_{\partial\mathcal{T}_{h}}\\
& -(\boldsymbol{q}+\boldsymbol{\beta}u,\nabla w)_{\mathcal{T}_{h}}
+\langle \bld q\cdot \boldsymbol{n}
+ \bld\beta\cdot\bld n \{\lambda/u\},w - \mu\rangle_{\partial\mathcal{T}_{h}},
\end{align*}
where
\[
 \{\lambda / u\} : = \left\{\begin{tabular}{c c}
                             $\lambda$, &\quad\text{ if } $\bld \beta\cdot\bld n <0$,\\
                             $u$, &\quad\text{ if } $\bld \beta\cdot\bld n \ge 0$.\\
                            \end{tabular}
\right.
\]

Comparing the bilinear form for the HDG method \eqref{bilinear_form} with approximation spaces 
$ \widetilde{\bld V}_h\times
W_h\times M_h(0)$ and the stability function $\tau$ in \eqref{tau1} and that for the MH--DG method
in \eqref{mhdg}, we notice that the only difference lies in the definition of the numerical flux 
$(\widehat{\bld q}_h +\widehat{\bld\beta u_h})\cdot \bld n$ and $\bld q_h \cdot \boldsymbol{n}
+ \bld\beta\cdot\bld n \{\lambda/u\}$.
However, by the following simple calculation, we observe that the two numerical flux are actually the same:
\begin{align*}
(\widehat{\bld q}_h +\widehat{\bld\beta u_h} )\cdot \bld n& = \;{\bld q}_h\cdot \bld n +{\bld\beta\cdot \bld n \lambda_h} + \tau (u_h-\lambda_h)\\
& = \;{\bld q}_h\cdot \bld n +{\bld\beta\cdot \bld n \lambda_h} + \max(\bld \beta\cdot \bld n,0) ( u_h-\lambda_h) \\
& = \; \left\{\begin{tabular}{c c}
                             $\bld q_h\cdot \bld n+ \bld\beta\cdot\bld n \lambda_h$, &\quad\text{ if } $\bld \beta\cdot\bld n <0$\\
                             $\bld q_h\cdot \bld n+ \bld\beta\cdot\bld n u_h$, &\quad\text{ if } $\bld \beta\cdot\bld n \ge 0$
                            \end{tabular}\right.\\
&=\; \bld q_h\cdot\bld n+\bld\beta\cdot\bld n \{\lambda_h/u_h\}.
\end{align*}
Hence, these two methods coincide.

\section{Conditioning of the HDG methods}
\label{appendix}
In this section, we give a proof of Theorem~\ref{Thm_conditioning}. Again, the key idea is to 
recover an estimate of the $L^2$--norm of $u_h$, see the estimate \eqref{low_bound1} below.
By similar argument in the proof of Lemma~\ref{lemma_practical_infsup}, 
%for HDG methods with the second example of stabilization function,
we have the following local energy estimate:

\begin{lemma}
\label{lemma_local_energy_estimates}
If $\epsilon \leq \mathcal{O}(h)$, then there is $h_{0}>0$, which is independent of $\epsilon$ and $h$, 
such that for any $\lambda\in M_{h}(0)$ and 
$K\in\mathcal{T}_{h}$,
\begin{align}
\label{local_energy_estimates}
\epsilon^{-1/2}\Vert \boldsymbol{q}_{h}^{\lambda}\Vert_{K}+ \Vert u_{h}^{\lambda}\Vert_{K}
+\Vert \vert\tau-\frac{1}{2}\boldsymbol{\beta}\cdot\boldsymbol{n}\vert^{1/2}u_{h}^{\lambda} \Vert_{\partial K}
\leq C \Vert \vert\tau-\frac{1}{2}\boldsymbol{\beta}\cdot\boldsymbol{n}\vert^{1/2}\lambda\Vert_{\partial K},
\end{align}
if $h< h_{0}$. 
\end{lemma}

For all $(\boldsymbol{\sigma},v,\lambda), (\boldsymbol{r},w,\mu)\in H^{1}(\mathcal{T}_{h};\mathbb{R}^{d})\times H^{1}(\mathcal{T}_{h})\times L^{2}(\mathcal{E}_{h})$, we define 
\begin{align*}
%\label{reduced_bilinear _form}
b_{h}((\boldsymbol{\sigma},v,\lambda), (\boldsymbol{r},w,\mu)) %\\
%\nonumber
= &\; (\epsilon^{-1}\boldsymbol{\sigma},\boldsymbol{r})_{\mathcal{T}_{h}}
+\langle \tau (v-\lambda),w - \mu \rangle_{\partial\mathcal{T}_{h}}\\
%\nonumber
&-(\boldsymbol{\beta}v,\nabla w)_{\mathcal{T}_{h}} 
+ \langle (\boldsymbol{\beta}\cdot\boldsymbol{n})\lambda, w - \mu \rangle_{\partial\mathcal{T}_{h}}
-((\nabla\cdot \boldsymbol{\beta})v, w)_{\mathcal{T}_{h}},
\end{align*}
The next result is similar to Lemma~\ref{lemma_practical_infsup}.
%, we have Lemma~\ref{low_bound_reduced_system} in the following.
\begin{lemma}
\label{low_bound_reduced_system}
If $\epsilon \leq \mathcal{O}(h)$, then there is $h_{0}>0$, which is independent of $\epsilon$ and $h$, 
such that for any $\lambda\in M_{h}(0)$,
\begin{align*}
%\label{low_bound_reduced_system_eq}
&\epsilon^{-1}\Vert \boldsymbol{q}_{h}^\lambda\Vert_{\mathcal{T}_{h}}^{2}+ \Vert u_{h}^\lambda\Vert_{\mathcal{T}_{h}}^{2}
+\Vert \vert \tau -\frac{1}{2}\boldsymbol{\beta}\cdot\boldsymbol{n}\vert^{1/2}
(u_{h}^\lambda -\lambda)\Vert_{\partial\mathcal{T}_{h}}^{2} \\
%\nonumber 
\leq & C b_{h}((\boldsymbol{q}_{h}^{\lambda},u_{h}^{\lambda},\lambda), ( \boldsymbol{q}_{h}^{(P_{0,M}\varphi)\lambda},
u_{h}^{(P_{0,M}\varphi)\lambda}, (P_{0,M}\varphi)\lambda)),
\end{align*}
if $h< h_{0}$. Here, the weight function $\varphi = e^{-\psi}+\chi$ is introduced in (\ref{weight_function}). 
$P_{0,M}$ is the $L^{2}$--orthogonal projection onto $P_{0}(\mathcal{E}_{h})$. 
\end{lemma}

\begin{remark}
Notice that in general, the space $\{(\boldsymbol{q}_{h}^{m}, u_{h}^{m},m):m\in M_{h}(0)\}$ is a non-trivial subspace of 
$\boldsymbol{V}_{h}\times W_{h}\times M_{h}(0)$.
Then, given $m\in M_{h}$, $(\boldsymbol{\Pi}_{h}(\varphi \boldsymbol{q}_{h}^{m}),P_{h}(\varphi  u_{h}^{m}), P_{M}(\varphi m))$ is {\em not} necessarily contained in 
 $\{(\boldsymbol{q}_{h}^{m}, u_{h}^{m},m):m\in M_{h}\}$. So, the proof of Lemma~\ref{low_bound_reduced_system} can {\em not} be derived from the stability 
 of HDG methods in Lemma~\ref{lemma_practical_infsup}.
\end{remark}

\begin{proof}
We accomplish the proof in the following steps.

(\textbf{I}) By the same argument in the proof of Lemma~\ref{lemma_ideal_infsup}, if $h$ is small enough (independent of $\epsilon$), then 
\begin{align*}
%\label{reduced_low_bound_ineq1}
&\epsilon^{-1}\chi \Vert \boldsymbol{q}_{h}^{\lambda}\Vert_{\mathcal{T}_{h}}^{2}+ \Vert u_{h}^{\lambda}\Vert_{\mathcal{T}_{h}}^{2}
+\chi \Vert \vert \tau -\frac{1}{2}\boldsymbol{\beta}\cdot\boldsymbol{n}\vert^{1/2}
(u_{h}^{\lambda} -\lambda)\Vert_{\partial\mathcal{T}_{h}}^{2} \\
%\nonumber
 \leq & C b_{h}((\boldsymbol{q}_{h}^{\lambda}, u_{h}^{\lambda},\lambda), (\varphi\boldsymbol{q}_{h}^{\lambda},
\varphi u_{h}^{\lambda}, \varphi\lambda)).
\end{align*}

(\textbf{II}) By similar argument in the proof of Lemma~\ref{lemma_practical_infsup}, if we choose $\chi$ big enough and $h$ small enough (both are independent of $\epsilon$), then 
\begin{align*}
%\label{reduced_low_bound_ineq2}
&\epsilon^{-1}\chi\Vert \boldsymbol{q}_{h}^{\lambda}\Vert_{\mathcal{T}_{h}}^{2}+ \Vert u_{h}^{\lambda}\Vert_{\mathcal{T}_{h}}^{2}
+\chi\Vert \vert \tau -\frac{1}{2}\boldsymbol{\beta}\cdot\boldsymbol{n}\vert^{1/2}
(\mathcal{U}\lambda -\lambda)\Vert_{\partial\mathcal{T}_{h}}^{2} \\
%\nonumber
 \leq & C b_{h}((\boldsymbol{q}_{h}^{\lambda}, u_{h}^{\lambda},\lambda), ( (P_{0,h} \varphi) \boldsymbol{q}_{h}^{\lambda},
P_{h}(\varphi u_{h}^{\lambda}), (P_{0,M}\varphi)\lambda)).
\end{align*}
Here, we define $P_{0,h}\varphi$ to be the average of $\varphi$ on every element $K\in\mathcal{T}_{h}$.

(\textbf{III}) Now, we want to bound $b_{h}((\boldsymbol{q}_{h}^{\lambda},u_{h}^{\lambda},\lambda), 
( (P_{0,h} \varphi)\boldsymbol{q}_{h}^{\lambda},
(P_{0,h} \varphi)u_{h}^{\lambda}, (P_{0,M}\varphi)\lambda))$ from below. Notice that 
\begin{align*}
& b_{h}((\boldsymbol{q}_{h}^{\lambda},u_{h}^{\lambda},\lambda), ( (P_{0,h} \varphi)\boldsymbol{q}_{h}^{\lambda},
(P_{0,h} \varphi) u_{h}^{\lambda}, (P_{0,M}\varphi)\lambda))\\
= & b_{h}((\boldsymbol{q}_{h}^{\lambda},u_{h}^{\lambda},\lambda), ( (P_{0,h} \varphi)\boldsymbol{q}_{h}^{\lambda},
P_{h}(\varphi u_{h}^{\lambda}), (P_{0,M}\varphi)\lambda)) \\
& + (\boldsymbol{\beta}\cdot\nabla u_{h}^{\lambda},P_{h}(\varphi u_{h}^{\mu})-(P_{0,h} \varphi)u_{h}^{\mu})_{\mathcal{T}_{h}}\\
& +\langle (\tau-\boldsymbol{\beta}\cdot\boldsymbol{n})(u_{h}^{\lambda}-\lambda),
P_{h}(\varphi u_{h}^{\mu})-(P_{0,h} \varphi)u_{h}^{\mu}\rangle_{\partial\mathcal{T}_{h}}.
\end{align*}
By \eqref{local2}, we have
\begin{align}
\label{reduce_system_key1}
& b_{h}((\boldsymbol{q}_{h}^{\lambda},u_{h}^{\lambda},\lambda), ( (P_{0,h} \varphi)\boldsymbol{q}_{h}^{\lambda},
(P_{0,h} \varphi) u_{h}^{\lambda}, (P_{0,M}\varphi)\lambda))\\
\nonumber
= & b_{h}((\boldsymbol{q}_{h}^{\lambda},u_{h}^{\lambda},\lambda), ( (P_{0,h} \varphi)\boldsymbol{q}_{h}^{\lambda},
P_{h}(\varphi u_{h}^{\lambda}), (P_{0,M}\varphi)\lambda)) \\
\nonumber
& - (\nabla\cdot\boldsymbol{q}_{h}^{\lambda},  P_{h}(\varphi u_{h}^{\mu})-(P_{0,h} \varphi)u_{h}^{\mu})_{\mathcal{T}_{h}}.
\end{align}
By inverse inequality to $\nabla\cdot\boldsymbol{q}_{h}^{\lambda}$ and assumption $\epsilon \leq \mathcal{O}(h)$, we have
\begin{equation*}
\Vert \nabla\cdot\boldsymbol{q}_{h}^{\lambda}\Vert_{\mathcal{T}_{h}}\leq C \epsilon^{-1/2}h^{-1/2}
\Vert \boldsymbol{q}_{h}^{\lambda}\Vert_{\mathcal{T}_{h}}.
\end{equation*}
In addition, we have
\begin{equation*}
\Vert P_{h}(\varphi u_{h}^{\mu})-(P_{0,h} \varphi)u_{h}^{\mu}\Vert_{\mathcal{T}_{h}}\leq C h \Vert u_{h}^{\mu}\Vert_{\mathcal{T}_{h}}.
\end{equation*}
So, if $h$ is small enough (independent of $\epsilon,\chi$), we have
\begin{align}
\label{reduced_low_bound_ineq3}
&\chi\epsilon^{-1}\Vert \boldsymbol{q}_{h}^{\lambda}\Vert_{\mathcal{T}_{h}}^{2}+ \Vert u_{h}^{\lambda}\Vert_{\mathcal{T}_{h}}^{2}
+\chi\Vert \vert \tau -\frac{1}{2}\boldsymbol{\beta}\cdot\boldsymbol{n}\vert^{1/2}
(u_{h}^{\lambda} -\lambda)\Vert_{\partial\mathcal{T}_{h}}^{2} \\
\nonumber
 \leq & C b_{h}((\boldsymbol{q}_{h}^{\lambda},u_{h}^{\lambda},\lambda), ( (P_{0,h} \varphi)\boldsymbol{q}_{h}^{\lambda},
(P_{0,h} \varphi)u_{h}^{\lambda}, (P_{0,M}\varphi)\lambda)).
\end{align}
%The constant $C$ in (\ref{reduced_low_bound_ineq3}) is independent of $\epsilon,h,\chi$.

(\textbf{IV}) Now, we want to bound $b_{h}((\boldsymbol{q}_{h}^{\lambda},u_{h}^{\lambda},\lambda), 
( (P_{0,h} \varphi)\boldsymbol{q}_{h}^{\lambda},
u_{h}^{(P_{0,M}\varphi)\lambda}, (P_{0,M}\varphi)\lambda))$ from below. Similar to (\ref{reduce_system_key1}), we have 
\begin{align}
\label{reduce_system_key2}
& b_{h}((\boldsymbol{q}_{h}^{\lambda},u_{h}^{\lambda},\lambda), ( (P_{0,h} \varphi)\boldsymbol{q}_{h}^{\lambda},
u_{h}^{(P_{0,M}\varphi)\lambda}, (P_{0,M}\varphi)\lambda))\\
\nonumber
= & b_{h}((\boldsymbol{q}_{h}^{\lambda}, u_{h}^{\lambda},\lambda), ( (P_{0,h} \varphi)\boldsymbol{q}_{h}^{\lambda},
(P_{0,h} \varphi)u_{h}^{\lambda}, (P_{0,M}\varphi)\lambda)) \\
\nonumber
& - (\nabla\cdot\boldsymbol{q}_{h}^{\lambda},  (P_{0,h} \varphi)u_{h}^{\lambda}-u_{h}^{(P_{0,M}\varphi)\lambda})_{\mathcal{T}_{h}}.
\end{align}
Since $\lambda\rightarrow u_{h}^{\lambda}$ is linear and \eqref{local_energy_estimates}, then for any $K\in\mathcal{T}_{h}$,
\begin{align*}
& \Vert (P_{0,h} \varphi)u_{h}^{\lambda}-u_{h}^{(P_{0,M}\varphi)\lambda}\Vert_{K} \\
= &\Vert u_{h}^{(P_{0,h} \varphi)\lambda}-u_{h}^{(P_{0,M}\varphi)\lambda}\Vert_{K}\\
\leq & C \Vert \vert\tau -\frac{1}{2}\boldsymbol{\beta}\cdot\boldsymbol{n}\vert^{1/2}
((P_{0,h} \varphi)\lambda-(P_{0,M}\varphi)\lambda)\Vert_{\partial K}\\
\leq & C h\Vert \vert\tau -\frac{1}{2}\boldsymbol{\beta}\cdot\boldsymbol{n}\vert^{1/2}\lambda \Vert_{\partial K}\\
\leq & C h^{1/2} \left( \Vert u_{h}^{\lambda}\Vert_{K}^{2}
+\Sigma_{F\in\mathcal{E}(K)}\Vert \vert \tau -\frac{1}{2}\boldsymbol{\beta}\cdot\boldsymbol{n}\vert^{1/2}
(u_{h}^{\lambda} -\lambda)\Vert_{F}^{2} \right)^{1/2}.
\end{align*}
Recall that by an inverse inequality to $\nabla\cdot\boldsymbol{q}_{h}^{\lambda}$ and assumption $\epsilon \leq \mathcal{O}(h)$,
\begin{align*}
\Vert \nabla\cdot\boldsymbol{q}_{h}^{\lambda}\Vert_{\mathcal{T}_{h}}\leq C \epsilon^{-1/2}h^{-1/2}
\Vert \boldsymbol{q}_{h}^{\lambda}\Vert_{\mathcal{T}_{h}}.
\end{align*}
So, if $\chi$ is big enough (independent of $\epsilon,h$), we have 
\begin{align*}
%\label{reduced_low_bound_ineq4}
&\epsilon^{-1}\Vert \boldsymbol{q}_{h}^{\lambda}\Vert_{\mathcal{T}_{h}}^{2}+ \Vert u_{h}^{\lambda}\Vert_{\mathcal{T}_{h}}^{2}
+\Vert \vert \tau -\frac{1}{2}\boldsymbol{\beta}\cdot\boldsymbol{n}\vert^{1/2}
(u_{h}^{\lambda} -\lambda)\Vert_{\partial\mathcal{T}_{h}}^{2} \\
%\nonumber
 \leq & C b_{h}((\boldsymbol{q}_{h}^{\lambda},u_{h}^{\lambda},\lambda), ( (P_{0,h} \varphi)\boldsymbol{q}_{h}^{\lambda},
u_{h}^{(P_{0,M}\varphi)\lambda}, (P_{0,M}\varphi)\lambda)).
\end{align*}
%The constant $C$ in (\ref{reduced_low_bound_ineq4}) is independent of $\epsilon,h$.

(\textbf{V}) By \eqref{local_energy_estimates}, \eqref{reduced_low_bound_ineq3}, \eqref{assump_cond} and the fact that 
$\lambda\rightarrow\boldsymbol{q}_{h}^{\lambda}$ is linear, we have
\begin{align*}
%\label{reduced_low_bound_ineq5}
&\epsilon^{-1}\Vert \boldsymbol{q}_{h}^{\lambda}\Vert_{\mathcal{T}_{h}}^{2}+ \Vert u_{h}^{\lambda}\Vert_{\mathcal{T}_{h}}^{2}
+\Vert \vert \tau -\frac{1}{2}\boldsymbol{\beta}\cdot\boldsymbol{n}\vert^{1/2}
(u_{h}^{\lambda} -\lambda)\Vert_{\partial\mathcal{T}_{h}}^{2} \\
%\nonumber 
\leq & C b_{h}((\boldsymbol{q}_{h}^{\lambda},u_{h}^{\lambda},\lambda), ( \boldsymbol{q}_{h}^{(P_{0,M}\varphi)\lambda},
u_{h}^{(P_{0,M}\varphi)\lambda}, (P_{0,M}\varphi)\lambda))
\end{align*}
if $h$ is small enough (independent of $\epsilon$).

So, we can conclude the proof is complete.
\end{proof}

Now, we are ready to prove Theorem~\ref{Thm_conditioning}.
%\begin{proof}
By the definition of $\tilde{a}_{h}$ in (\ref{reduced_matrix2}), 
\begin{align*}
& \tilde{a}_{h}(\tilde{\lambda},\tilde{\mu}) = a_{h}(\Lambda_{\epsilon}^{-1}\tilde{\lambda},\Lambda_{\epsilon}^{-1}\tilde{\mu}) \\
= & b_{h}((\boldsymbol{q}_{h}^{\Lambda_{\epsilon}^{-1}\tilde{\lambda}}, u_{h}^{\Lambda_{\epsilon}^{-1}\tilde{\lambda}},
\Lambda_{\epsilon}^{-1}\tilde{\lambda}), 
(\boldsymbol{q}_{h}^{\Lambda_{\epsilon}^{-1}\tilde{\mu}}, u_{h}^{\Lambda_{\epsilon}^{-1}\tilde{\mu}},\Lambda_{\epsilon}^{-1}\tilde{\mu})),
\end{align*}
for all $\tilde{\lambda},\tilde{\mu}\in M_{h}(0)$.

We recall that $\Lambda_{\epsilon}|_{F} = \left(\sup_{x\in F}\vert \boldsymbol{\beta}\cdot\boldsymbol{n}(x)\vert
+\min(\dfrac{\epsilon}{h_{F}},1) \right)^{1/2},\quad \forall F\in\mathcal{E}_{h}$ in (\ref{change_variables}).
By assumption \eqref{assump_cond} 
%the second example of stabilization function in (\ref{tau2}) 
and Lemma~\ref{low_bound_reduced_system}, we have that for any $\tilde{\lambda}\in M_{h}(0)$,
\begin{align}
\label{low_bound1}
& \tilde{a}_{h}(\tilde{\lambda},(P_{0,M}\varphi)\tilde{\lambda})\\
\nonumber
\geq & C\left(\Vert u_{h}^{\Lambda_{\epsilon}^{-1}\tilde{\lambda}}\Vert_{\mathcal{T}_{h}}^{2}+
\Vert \vert \tau -\frac{1}{2}\boldsymbol{\beta}\cdot\boldsymbol{n}\vert^{1/2}
(u_{h}^{\Lambda_{\epsilon}^{-1}\tilde{\lambda}} -\Lambda_{\epsilon}^{-1}\tilde{\lambda})\Vert_{\partial\mathcal{T}_{h}}^{2}\right)
\\
\nonumber
\geq & C h \Vert \vert \tau -\frac{1}{2}\boldsymbol{\beta}\cdot\boldsymbol{n}\vert^{1/2}
\Lambda_{\epsilon}^{-1}\tilde{\lambda}\Vert_{\partial\mathcal{T}_{h}}^{2}\quad\quad\text{ (by trace inequality and triangle inequality)}\\
\nonumber
\geq & C \Vert \tilde{\lambda}\Vert_{h}^{2}\geq C \Vert \tilde{\lambda}\Vert_{h}\cdot \Vert (P_{0,M}\varphi)\tilde{\lambda}\Vert_{h},
\end{align}
where
\begin{equation*}
\Vert \tilde{\lambda}\Vert_{h} = h^{1/2}\Vert \tilde{\lambda}\Vert_{\mathcal{E}_{h}},\quad\forall \tilde{\lambda}\in L^{2}(\mathcal{E}_{h}).
\end{equation*}

According to (\ref{local_energy_estimates}) and the definition of $\Lambda_{\epsilon}$ in (\ref{change_variables}), we have
\begin{equation}
\label{upp_bound1}
\tilde{a}_{h}(\tilde{\lambda},\tilde{\mu})\leq Ch^{-2}\Vert \tilde{\lambda}\Vert_{h}\cdot \Vert \tilde{\mu}\Vert_{h},
\quad \forall \tilde{\lambda}, \tilde{\mu}\in M_{h}(0).
\end{equation}

Using (\ref{low_bound1}) and (\ref{upp_bound1}), we can conclude the proof of Theorem~\ref{Thm_conditioning}.
%\end{proof} 

\section{Generating special meshes}
\label{section_appb}

As in \cite{CockburnDongGuzmanQian2010}, we do not intend to provide a detailed description how to generate 
meshes satisfying assumption (\ref{mesh_assumps}). We just give the main idea to generate the triangulation 
in the following, which is similar to the idea in \cite{CockburnDongGuzmanQian2010}.

\begin{itemize}
\item[(i)] Given a positive value $h$, we triangulate the outflow boundary $\Gamma^{+}=
\{x\in\partial\Omega:\boldsymbol{\beta}\cdot\boldsymbol{n}(x)>0\}$ in segments of size no bigger than $h$.

\item[(ii)] For each node $x_{0}$ on $\Gamma^{+}$, we apply the forward Euler time-marching method to the problem
\begin{equation*}
\frac{d}{dt}x(t) = -\boldsymbol{\beta}(x(t))\quad t>0, x(0)=x_{0},
\end{equation*}
to obtain the set of nodes $\{x_i\}^{N(x_{0})}_{i=1}$ such that the distance between $x_i$ and
$x_{i-1}$ is of order $h$ and $x_{N (x_0)}$ is the point on $\partial\Omega\setminus \Gamma^{+}$.

\item[(iii)] We add the vertices of $\partial\Omega\setminus \Gamma^{+}$ to the set of nodes. Then we
generate a triangulation.

\item[(iv)] We numerically check assumption (\ref{mesh_assumps}) and modify the simplexes which
violate the assumption by using an algorithm similar to that in \cite{Iliescu99}.

\end{itemize}


\begin{thebibliography}{10}

%\bibitem{ArnoldBrezzi85}
%D.~N. Arnold and F.~Brezzi, \emph{Mixed and nonconforming finite element
%  methods: implementation, postprocessing and error estimates}, RAIRO Mod\'el.
%  Math. Anal. Num\'er. \textbf{19} (1985), 7--32.

%\bibitem{ArnoldBrezziCockburnMarini02}
%D.~N. Arnold, F.~Brezzi, B.~Cockburn, and L.~D. Marini, \emph{Unified analysis
%  of discontinuous {G}alerkin methods for elliptic problems}, SIAM J. Numer.
%  Anal. \textbf{39} (2002), 1749--1779.
  
\bibitem{AyusoMarini:cdf}
{B.~Ayuso and D.~Marini}, {\em Discontinuous Galerkin Methods for Advection-Diffusion-Reaction Problems}, 
SIAM J. Numer. Anal., \textbf{47(2)} (2009), 1391--1420.  

%\bibitem{BassiRebayMariottiPedinottiSavini97}
%F.~Bassi, S.~Rebay, G.~Mariotti, S.~Pedinotti, and M.~Savini, \emph{A
%  high-order accurate discontinuous finite element method for inviscid and
%  viscous turbomachinery flows}, 2nd European Conference on Turbomachinery
%  Fluid Dynamics and Thermodynamics (Antwerpen, Belgium) (R.~Decuypere and
%  G.~Dibelius, eds.), Technologisch Instituut, March 5--7 1997, pp.~99--108.

%\bibitem{BeckerHansbo00}
%R.~Becker and P.~Hansbo, \emph{Discontinuous galerkin methods for convection-diffusion problems with arbitrary p\'eclet number}, 'w
%  In P. Neittaanm\"{a}ki, T. Tiihonen, and P. Tarvainen, editors,
%  Numerical Mathematics and Advanced Applications, Proceedings of ENUMATH 99, pp~100--109,
%  River Edge, NJ, 2000. Technologisch Instituut, World Scientific.

\bibitem{BaumannOden99}
{C.E.~Baumann and J.T.~Oden}, {\em A discontinuous hp finite element method for convection--diffusion
problems}, 
Comput. Methods Appl. Mech. Engrg., \textbf{175} (1999), 311--341.

\bibitem{BrezziHughesMariniRussoSuli1999}
{F.~Brezzi, T.J.R.~Hughes, L.D.~Marini, A.~Russo, and E.~S\"{u}li}, {\em A Priori Error Analysis of 
Residual-freeBubbles for Advection-Diffusion Problems}, 
SIAM J. Numer. Anal., \textbf{36(6)} (1999), 1933--1948.

\bibitem{BrezziMariniSuli2000}
{F.~Brezzi, L.D.~Marini, and E.~S\"{u}li}, {\em Residual-free bubbles for 
advection-diffusion problems: the general error analysis}, 
Numer. Math., \textbf{85} (2000), 31--47.

\bibitem{BrooksHughes}
{A. N.~Brooks and T.J.R.~Hughes}, {\em Streamline upwind/{P}etrov-{G}alerkin formulations for convection dominated flows 
with particular emphasis on the incompressible {N}avier-{S}tokes equations}, 
Comput. Methods Appl. Mech. Engrg., \textbf{32} (1982), 199--259.


\bibitem{BuffaHughesSangalli2006}
{A.~Buffa, T.J.R.~Hughes, and G.~Sangalli}, {\em Analysis of a multiscale discontinuous Galerkin
method for convection-diffusion problems}, 
SIAM J. Numer. Anal., \textbf{44(4)} (2006), 1420--1440.

\bibitem{BurmanErn09}
{Erik~Burman and Alexandre~Ern}, {\em Stabilized Galerkin Approximation 
of Convection-Diffusion-Reaction equations: 
Discrete Maximum Principle and Convergence}, Math. Comp. \textbf{74} (2005), 1637--1652.

\bibitem{CastilloCockburnSchotzauSchwab02}
{P.~Castillo,  B.~Cockburn, D.~Sch\"otzau, and C.~Schwab}, {\em An optimal a priori error
estimate for the hp--version of the local discontinuous Galerkin method for convection--diffusion
problems}, Math. Comput., \textbf{71} (2002), 455--478.

\bibitem{ChenCockburnHDGI}
{Y.~Chen and B.~Cockburn}, {\em Analysis of variable-degree {HDG} methods for
  convection-diffusion equations. {P}art I: General nonconforming meshes}, IMA
  J. Num. Anal., \textbf{32(4)} (2012), 1267--1293.

\bibitem{ChenCockburnHDGII}
{Y.~Chen and B.~Cockburn}, {\em Analysis of variable-degree {HDG} methods for
  convection-diffusion equations. {P}art II: Semimatching nonconforming
  meshes}, Math. Comp., to appear.

\bibitem{Cockburn99}
{B.~Cockburn}, {\em A Discontinuous Galerkin methods for convection-dominated problems}, 
  In High-order methods for computational physics,
  volume 9 of Lect. Notes Comput. Sci. Eng.,
  Springer, Berlin, 1999, 69--224.

\bibitem{CockburnDawson00}
{B.~Cockburn and C.~Dawson}, {\em Some extensions of the local discontinuous Galerkin method for
  convection-diffusion equations in multidimensions}, 
  In The mathematics of finite elements and applications,
  X, MAFELAP 1999 (Uxbridge),
  Elsevier, Oxford, 2000, 225--238.
\bibitem{CockburnDong07}
{B.~Cockburn and B.~Dong}, {\em An analysis of the minimal dissipation local discontinuous Galerkin 
method for convection--diffusion problems},
  J. Sci. Comput. \textbf{32} (2007), 233--262.

\bibitem{CockburnDongGuzmanSFH08}
{B.~Cockburn, B.~Dong, and J.~Guzm\'an}, {\em A superconvergent
  {LDG}-hybridizable {G}alerkin method for second-order elliptic problems},
  Math. Comp. \textbf{77} (2008), 1887--1916.

\bibitem{CockburnDongGuzman2008}
{B.~Cockburn, B.~Dong, and J.~Guzm\'an}, {\em Optimal convergence of the original 
DG method for the transport-reaction equation on special meshes},
  SIAM J. Numer. Anal. \textbf{48(1)} (2008), 1250--1265.
\bibitem{CockburnDongGuzmanMarcoRiccardo09}
{B.~Cockburn, B.~Dong, J.~Guzm\'an, M.~Restelli, and R.~Sacco}, {\em A
hybridizable discontinuous Galerkin method for steady--state convection--diffusion--reaction problems},
  J. Sci. Comput. \textbf{31} (2009), 3827--3846.

\bibitem{CockburnDongGuzmanQian2010}
{B.~Cockburn, B.~Dong, J.~Guzm\'an, and J.~Qian}, {\em Optimal Convergence of 
the Original DG Method on Special Meshes for Variable Transport Velocity},
  SIAM J. Numer. Anal. \textbf{46(3)} (2010), 133--146.

\bibitem{CockburnDuboisGopalakrishnanTan2013}
{B.~Cockburn, O.~Dubois, J.~Gopalakrishnan, and S.~Tan}, {\em Multigrid for an HDG method},
  IMA J. Numer. Anal., to appear.

\bibitem{CockburnGopalakrishnanLazarov09}
{B.~Cockburn, J.~Gopalakrishnan, and R.~Lazarov}, {\em Unified hybridization of
  discontinuous {Galerkin}, mixed and continuous {Galerkin} methods for second
  order elliptic problems}, SIAM J. Numer. Anal. \textbf{47} (2009),
  1319--1365.

\bibitem{CockburnGopalakrishnanSayas09}
{B.~Cockburn, J.~Gopalakrishnan, and F.-J.~Sayas}, {\em A projection-based error
  analysis of {HDG} methods}, Math. Comp. \textbf{79} (2010), 1351--1367.
  
\bibitem{CockburnQiuShi12}
{B.~Cockburn, W.~Qiu, and K.~Shi} {\em Conditions for superconvergence of HDG Methods for second-order elliptic problems},
 Math. Comp., \textbf{81} (2012), 1327--1353.  

\bibitem{CockburnShu98}
{B.~Cockburn and  C.W.~Shu} {\em The local discontinuous Galerkin method for time--dependent 
convection--diffusion systems},
 SIAM J. Numer. Anal., \textbf{35} (1998), 2440--2463.  

\bibitem{CockburnShu01}
{B.~Cockburn and  C.W.~Shu} {\em Runge--Kutta discontinuous Galerkin method for convection--dominated problems},
J. Sci. Comput., \textbf{16} (2001), 173--261.  

\bibitem{DevinatzEllisFriedman}
{A.~Devinatz, R.~Ellis, and A.~Friedman}, {\em The asymptotic behavior of the first real eigenvalue of second order elliptic operators with
 a small parameter in the highest derivatives. II.}, 
 Indiana Univ. Math. J. \textbf{23} (1973-1974), 991--1011.

\bibitem{Eckhaus72}
{W.~Eckhaus}, {\em Boundary layers in linear elliptic singular perturbation problems}, 
SIAM Rev. \textbf{44(2)} (1972).

\bibitem{Egger2010}
{H.~Egger and J.~Sch{\"{o}}berl}, {\em A hybrid mixed discontinuous Galerkin finite-element method for convection–diffusion problems}, 
IMA  J. Num. Anal., \textbf{30} (2010), 1206--1234.

\bibitem{Goering83}
{H.~Goering, A.~Felgenhauer, G.~Lube, H.-G.~Roos, and L.~Tobiska}, 
{\em Singularly perturbed differential equations},  Akademie-Verlag, Berlin 1983.

\bibitem{GopaKanschat2003}
{J.~Gopalakrishnan and G.~Kanschat}, {\em A multilevel discontinuous Galerkin method}, 
Numer. Math., \textbf{95(3)} (2003), 527--550.

\bibitem{HoustonSchwabSuli2000}
{P.~Houston, C.~Schwab, and E.~S\"{u}li}, {\em Stabilized hp--finite element methods for 
first order hyperbolic problems}, SIAM  J. Num. Anal., \textbf{37} (2000), 1618--1643.

\bibitem{HoustonSchwabSuli2002}
{P.~Houston, C.~Schwab, and E.~S\"{u}li}, {\em Discontinuous hp-finite element methods for 
advection-diffusion-reaction problems}, SIAM  J. Num. Anal., \textbf{39(6)} (2002), 2133--2163.

\bibitem{HughesScovazziBochevBuffa2006}
{T.J.R.~Hughes, G.~Scovazzi, P.~Bochev, and A.~Buffa}, {\em A multiscale discontinuous Galerkin
method with the computational structure of a continuous Galerkin method}, 
Comput. Methods Appl. Mech. Engrg., \textbf{195(19-22)} (2006), 2761--2787.

\bibitem{Johnson86}
C.~Johnson and J.~Pitk\"{a}ranta {\em An analysis of the discontinuous Galerkin method for a 
scalar hyperbolic equation},
Math. Comp., \textbf{46} (1986), 1--26.


\bibitem{Iliescu99}
T.~Iliescu, {\em A flow-aligning algorithm for convection-dominated problems},
Internat. J. Numer. Methods Engrg., \textbf{46} (1999), 993--1000.

\bibitem{Peterson91}
{T.E.~Peterson}, {\em A note on the convergence of the discontinuous Galerkin method for 
a scalar hyperbolic equation}, 
SIAM J. Numer. Anal., \textbf{28}(1991), pp. 133--140.

\bibitem{ReedHill73}
{W.~Reed and T.~Hill}, {\em Triangular mesh methods for the neutron transport equation}, 
Technical Report LA- UR-73-479, Los Alamos Scientific Laboratory, 1973.

\bibitem{NguyenPeraireCockburnHDGLCD09}
{N.C. Nguyen, J.~Peraire, and B.~Cockburn}, {\em An implicit high--order hybridizable discontinuous Galerkin
method for linear convection--diffusion equations}, J. Comput. Phys. 
  \textbf{288} (2009), 3232--3254.
  \bibitem{NguyenPeraireCockburnHDGNCD09}
{N.C. Nguyen, J.~Peraire, and B.~Cockburn}, {\em An implicit high--order hybridizable discontinuous Galerkin
method for nonlinear convection--diffusion equations}, J. Comput. Phys. 
  \textbf{288} (2009), 8841--8855.

\bibitem{Roos12}
H.-G.~Roos,
{\em Robust numerical methods for singularly perturbed differential equations: a survey
covering 2008--2012}, ISRN Appl. Math. 2012, 1--30.

\bibitem{Roos08}
{H.-G.~Roos,  M.~Stynes, and L.~Tobiska}, 
{\em Robust numerical methods for singularly perturbed differential equations},  
volume 24 of Springer Series in Computational Mathematics. Springer-Verlag, Berlin, 2008. 
Convection-diffusion and flow problems. Second edition.

\bibitem{ZarinRoos2005}
{H.~Zarin and H.G.~Roos}, {\em Interior penalty discontinuous approximations of convection-diffusion
problems with parabolic layers}, Numer. Math., \textbf{100(4)} (2005), 735--759.
  
\end{thebibliography}
\end{document}